\documentclass[12pt]{amsart}
\usepackage[colorlinks=true, pdfstartview=FitV, linkcolor=blue, citecolor=blue, urlcolor=blue]{hyperref}

\usepackage{amssymb,amsmath, amscd}
\usepackage{times, verbatim}
\usepackage{graphicx}
\usepackage[english]{babel}
 \usepackage[usenames, dvipsnames]{color}
\usepackage{amsmath,amssymb,amsfonts}
\usepackage{enumerate}
\usepackage{anysize}
\marginsize{3cm}{3cm}{3cm}{3cm}
\input xy
\xyoption{all}
\usepackage{pb-diagram}
\usepackage[all]{xy}
\input xy
\xyoption{all}

\DeclareFontFamily{OT1}{rsfs}{}
\DeclareFontShape{OT1}{rsfs}{n}{it}{<-> rsfs10}{}
\DeclareMathAlphabet{\mathscr}{OT1}{rsfs}{n}{it}

\begin{document}
\theoremstyle{plain}

\newtheorem{theorem}{Theorem}[section]
\newtheorem{thm}[equation]{Theorem}
\newtheorem{prop}[equation]{Proposition}
\newtheorem{cor}[equation]{Corollary}
\newtheorem{conj}[equation]{Conjecture}
\newtheorem{lemma}[equation]{Lemma}
\newtheorem{definition}[equation]{Definition}
\newtheorem{question}[equation]{Question}

%%    DEFINITIONS, Remark, EXAMPLES, NOTATIONS
\theoremstyle{definition}
\newtheorem{conjecture}[theorem]{Conjecture}

\newtheorem{example}[equation]{Example}
\numberwithin{equation}{section}

\newtheorem{remark}[equation]{Remark}

\newcommand{\Hecke}{\mathcal{H}}
\newcommand{\Liea}{\mathfrak{a}}
\newcommand{\Cmg}{C_{\mathrm{mg}}}
\newcommand{\Cinftyumg}{C^{\infty}_{\mathrm{umg}}}
\newcommand{\Cfd}{C_{\mathrm{fd}}}
\newcommand{\Cinftyfd}{C^{\infty}_{\mathrm{ufd}}}
\newcommand{\sspace}{\Gamma \backslash G}
\newcommand{\PP}{\mathcal{P}}
\newcommand{\bfP}{\mathbf{P}}
\newcommand{\bfQ}{\mathbf{Q}}
\newcommand{\Siegel}{\mathfrak{S}}
\newcommand{\g}{\mathfrak{g}}
\newcommand{\A}{\mathbb{A}}
\newcommand{\Q}{\mathbb{Q}}
\newcommand{\Gm}{\mathbb{G}_m}
\newcommand{\kk}{\mathfrak{k}}
\newcommand{\nn}{\mathfrak{n}}
\newcommand{\tF}{\tilde{F}}
\newcommand{\p}{\mathfrak{p}}
\newcommand{\m}{\mathfrak{m}}
\newcommand{\bb}{\mathfrak{b}}
\newcommand{\Ad}{{\rm Ad}\,}
\newcommand{\ttt}{\mathfrak{t}}
\newcommand{\frakt}{\mathfrak{t}}
\newcommand{\U}{\mathcal{U}}
\newcommand{\Z}{\mathbb{Z}}
\newcommand{\bfG}{\mathbf{G}}
\newcommand{\bfT}{\mathbf{T}}
\newcommand{\R}{\mathbb{R}}
\newcommand{\ST}{\mathbb{S}}
\newcommand{\h}{\mathfrak{h}}
\newcommand{\bC}{\mathbb{C}}
\newcommand{\C}{\mathbb{C}}
\newcommand{\F}{\mathbb{F}}
\newcommand{\N}{\mathbb{N}}
\newcommand{\qH}{\mathbb {H}}
\newcommand{\temp}{{\rm temp}}
\newcommand{\Hom}{{\rm Hom}}
\newcommand{\Aut}{{\rm Aut}}
\newcommand{\Ext}{{\rm Ext}}
\newcommand{\End}{{\rm End}\,}
\newcommand{\Ind}{{\rm Ind}\,}
\def\circG{{\,^\circ G}}
\def\M{{\rm M}}
\def\diag{{\rm diag}}
\def\Ad{{\rm Ad}}
\def\wG{{\widehat G}}
\def\G{{\rm G}}
\def\SL{{\rm SL}}
\def\PSL{{\rm PSL}}
\def\GSp{{\rm GSp}}
\def\PGSp{{\rm PGSp}}
\def\Sp{{\rm Sp}}
\def\St{{\rm St}}
\def\GU{{\rm GU}}
\def\SU{{\rm SU}}
\def\U{{\rm U}}
\def\GO{{\rm GO}}
\def\GL{{\rm GL}}
\def\PGL{{\rm PGL}}
\def\PD{{\rm PD}}
\def\GSO{{\rm GSO}}
\def\Gal{{\rm Gal}}
\def\SO{{\rm SO}}
\def\O{{\rm  O}}
\def\Sym{{\rm Sym}}
\def\sym{{\rm sym}}
\def\St{{\rm St}}
\def\tr{{\rm tr\,}}
\def\ad{{\rm ad\, }}
\def\Ad{{\rm Ad\, }}
\def\rank{{\rm rank\,}}

\subjclass{Primary 11F70; Secondary 22E55}

\title[Branching laws: the non-tempered case]
{Branching laws for Classical Groups: \\the non-tempered case}

\author{Wee Teck Gan, Benedict H. Gross and Dipendra Prasad}
\thanks{WTG is partially supported by an MOE Tier 2 grant R146-000-233-112.
  DP thanks  Science and Engineering research board of the
  Department of Science and Technology, India for its support
through the JC Bose
National Fellowship of the Govt. of India, project number JBR/2020/000006.
His work was also supported by a grant of
the Government of the Russian Federation
for the state support of scientific research carried out
under the  agreement 14.W03.31.0030 dated 15.02.2018. }

\address{W.T.G.: National University of Singapore,
Singapore 119076.}
\email{matgwt@nus.edu.sg}
\address{B.H.G: Department of Mathematics, University of California San Diego, La Jolla, 92093}\email{gross@math.harvard.edu}
\address{D.P.: Indian Institute of Technology Bombay, Powai, Mumbai-400076} 
\address{D.P.: St Petersburg State University, St Petersburg, Russia}
\email{prasad.dipendra@gmail.com}
\maketitle
    {\hfill \today}
\begin{abstract}
This paper  generalizes the GGP conjectures which were earlier formulated
for tempered or more generally generic L-packets to Arthur packets, especially for the
nongeneric L-packets arising from Arthur parameters.
The paper introduces the key notion of a {\bf relevant pair} of A-parameters which governs the branching laws for $\GL_n$ and all classical groups
over both local fields and global fields.
It plays a role for all the branching problems studied in \cite{GGP} including Bessel models and Fourier-Jacobi models. 
\end{abstract}    

\tableofcontents

\section{Introduction}

This paper is a sequel to our earlier work \cite{GP1, GP2, GGP, GGP2}, which discussed several restriction (or branching)
problems in the representation theory of classical groups. In the local case, the conjectural answer was given in terms of symplectic root numbers associated to the Langlands parameters (L-parameters for short). In the global case, for cuspidal automorphic representations, the conjectural answer was given in terms of the central value of automorphic L-functions. However, all of these predictions were for representations lying in local L-packets whose Langlands parameters are {\bf generic} (in particular, for tempered representations). In this paper, we attempt to generalize these conjectures to certain non-generic L-packets.
\vskip 5pt

We will consider only those representations of a classical group $G$ over a local field $k$ which arise as local components of   automorphic representations in the automorphic dual (a class of representations which was singled out by Arthur).
They have Arthur parameters (A-parameters for short) of the form
$$\phi_A: WD(k) \times \SL_2(\bC) \rightarrow  {^L}G$$
where the restriction of $\phi_A$ to the Weil-Deligne group $WD(k)$ of $k$ is an admissible homomorphism with bounded image and the restriction to $\SL_2$ is algebraic. By the Jacobson-Morozov theorem, the conjugacy class of $\SL_2(\bC)$ in the dual group
$\wG(\C)$ corresponds to a unipotent conjugacy class in $\wG(\C)$.
\vskip 5pt

One can obtain an L-parameter from an A-parameter as follows. The abelianization of the Weil group is isomorphic to the multiplicative group $k^*$, by local class field theory. Let $|- |$ be the canonical absolute value on $k^*$ and map the Weil-Deligne group to $WD(k) \times \SL_2(\bC)$ by
$$w \rightarrow (w, \diag(|w|^{1/2}, |w|^{-1/2})).$$
Composing $\phi_A$ with this homomorphism gives an L-parameter $\phi$. The map 
\[  \phi_A \mapsto \phi  \]
is an injection from the set of A-parameters to the set of L-parameters.
We will call  L-parameters arising in this way the L-parameters of Arthur type. 
 A particular example is when  the restriction of $\phi_A$ to $\SL_2(\C)$ is trivial, in which case $\phi$ is a tempered L-parameter. Hence, we have the containments
 \[  \{ \text{tempered L-parameters} \} \subset \{\text{L-parameters of Arthur type} \} \subset \{ \text{L-parameters}  \}. \]
 We will  consider the restriction problem only for those representations which are contained
 in an A-packet, to be called representations of Arthur type. Much of the paper will in fact deal
 with a sub-class  of representations of Arthur type, contained in the L-packet associated to
the L-parameter $\phi$ which is contained in the  A-packet of $\phi_A$.
\vskip 5pt

When  the L-parameter $\phi$ is tempered, the associated L-packet is generic (i.e., contains a generic member),
so are covered by our original conjectures in \cite{GGP}. In this case, the sum over the L-packet of a pair of classical groups of the multiplicity of restriction is always equal to one, and the adjoint L-function of $\phi$ is regular and non-zero at the point $s = 1$. On the other hand, when the restriction of $\phi_A$ to $\SL_2(\C)$ is non-trivial, the L-parameter $\phi$ is neither tempered nor generic. In particular, the adjoint L-function of the parameter $\phi$ is not regular at the point $s = 1$. Indeed, the Lie algebra of the dual group contains the representation $\bC \otimes \frak {sl}_2(\bC)$ of $WD(k) \times \SL_2(\bC)$, which contributes a simple pole to the adjoint L-function of $\phi$ at $s = 1$. For the general  L-parameters of Arthur type, we conjecture that the sum of multiplicities of the restriction over an L-packet of Arthur type for a pair of classical groups is either zero or one.
We will give a precise conjectural criterion  for which L-packets (of Arthur type)   is  the sum of the multiplicities equal to one. In this case, we further conjecture that the distinguished representation will be selected by the character of the component group of the L-parameter obtained from symplectic root numbers by the same recipe as in the tempered case.
\vskip 5pt

One can consider the analogous restriction problem in the setting of unitary representations, where one works with the direct integral decomposition of a restriction.
Work of Clozel \cite{Cl}  in this direction suggests that an irreducible unitary representation of a reductive group with a specific unipotent conjugacy class in its A-parameter can weakly contain only those representations of a subgroup with a specific closely related unipotent class in their A-parameters. The work of Venkatesh \cite{Ve} makes this  precise for the restriction of unitary representations of $\GL_n(k)$ to $\GL_{n-1}(k)$ (indeed, to $\GL_m(k)$ for any $m < n$), by explicating the map from unipotent classes of $\GL_n(\bC)$ to those of $\GL_{n-1}(\bC)$. 
\vskip 5pt

More precisely, let $\pi$ be a representation of $\GL_n(k)$ of Arthur type for
$k$ a non-archimedean local field, for which  the associated unipotent conjugacy class in $\GL_n(\C)$ corresponds to the
partition $$n_1 \geq n_2 \geq \cdots \geq n_r \geq 1.$$ 
 Then according to Venkatesh \cite{Ve}, the only 
unipotent conjugacy class  of $\GL_{n-1}(\C)$ involved in the restriction problem $\pi|_{\GL_{n-1}(k)}$
from $\GL_n(k)$ to $\GL_{n-1}(k)$,
is the one given by
$$n_1-1 \geq n_2-1 \geq \cdots \geq n_r-1 \geq 0,$$
omitting those $n_i$ which are 1, and adding a few 1's at the 
end if necessary. We should  add that the work of Clozel and Venkatesh deals only with a crude question:
that of  determining the possible ``types'' (i.e., the unipotent conjugacy class associated to the A-parameter)
of representations of $H$ which appear in the spectral decomposition of 
$\pi|_{H}$, and not the precise spectral decomposition or which representations of 
the correct type actually do appear in the spectral decomposition of $\pi|_H$. The extension of Venkatesh's results to the setting of classical groups has been carried out in the PhD thesis \cite{Hen} of A. Hendrickson (a student of the first author).
\vskip 5pt

The work of Clozel and Venkatesh is of course in the context of unitary representations. In this paper, on the other hand,  
we work in the setting of smooth representation theory and formulate a conjecture for the restriction of irreducible smooth representations of classical groups in terms of the notion of a {\bf relevant pair} of A-parameters.
We will find in particular that many more unipotent conjugacy classes 
of $\GL_{n-1}(\C)$ are involved in the restriction problem from $\GL_n(k)$ to $\GL_{n-1}(k)$: these are, so to say, of distance one apart from the unipotent conjugacy class of the representation of $\GL_n(k)$ we are starting with.
\vskip 5pt

The definition of a relevant pair of A-parameters is not too complicated but we defer its precise definition  to \S \ref{S:relevant}. 
An elegant reformulation of this notion was given by Zhiwei Yun and discussed in \S \ref{cor}; there, we will also give a geometric interpretation in terms of a moment map (in the sense of symplectic geometry) arising in the theory of reductive dual pairs. 
In the rest of this introduction, we take this notion as a black box.  \vskip 10pt

We first consider the restriction problem for $\GL_n$ in \S \ref{S:GLn}. The case of $\GL_n(k)$
is simpler than the case of classical groups  as A-packets and  L-packets for $\GL_n(k)$ are singleton sets.  
Let $\pi_M$ be an irreducible representation of $\GL_n(k)$ with A-parameter $M_A$ and associated L-parameter $M$   and let $\pi_N$ be an irreducible representation of $\GL_{n-1}(k)$ with A-parameter $N_A$ and associated L-parameter $N$. We conjecture that $\pi_N$ appears as a quotient of the restriction of $\pi_M$ to $\GL_{n-1}(k)$ if and only if $(M_A,N_A)$ is a relevant pair of A-parameters. We prove this in a number of cases for $p$-adic groups (such as when the Deligne $\SL_2(\C)$ in $WD(k)$ acts trivially) using the theory of derivatives of Bernstein and Zelevinsky \cite{BZ}. Recently,
M. Gurevich \cite{Gu} has extended this work to prove one direction of the conjecture in all cases. Namely, he showed that 
\[  
\Hom_{\GL_{n-1}(k)}(\pi_M, \pi_N) \ne 0 \Longrightarrow  \text{$(M_A,N_A)$ is a relevant pair of A-parameters.} \]
Gurevich also showed the converse in some cases, such as when at least one of $M_A$ or $N_A$ is
tempered\footnote{As this paper was being revised, we  became aware of a paper due to KY Chan in arXiv:2006.02623,  
  ``Restriction for general linear groups: the local non-tempered Gan-Gross-Prasad conjecture (non-Archimedean case)''
  proving the full conjecture for $\GL_n$ in the non-archimedean case.}.
 \vskip 5pt
 
 We also show that when the pair $(M_A,N_A)$ is relevant, the ratio of the local Langlands L-functions 
\begin{equation}  \label{E:R}
L(M,N, s) = \frac{L(M \otimes N^{\vee}, s + ~1/2) \cdot L(M^{\vee} \otimes N, s + ~1/2)}{ L(M \otimes M^{\vee}, s + 1)\cdot  L(N \otimes N^{\vee}, s + 1)} \end{equation}
does not vanish at the point $s = 0$ (but may have a pole).
Note that the L-functions in the denominator are the adjoint L-functions for $\GL_n$ and $\GL_{n-1}$ respectively. At least one of these L-functions has a pole at $s = 0$ in the non-tempered case, so the non-vanishing of $L(M,N, s)$ at $s = 0$ means that at least one of the L-functions in the numerator has a pole at $s = 0$. It is not clear to us if this observation about the analytic behaviour of $L(M,N,s)$ at $s=0$ plays a role in the local restriction problem. 
\vskip 5pt

We now consider the problem of restriction from the split group $G = \SO(V) = \SO_{2n+1}(k)$ to the subgroup $H = \SO(W) = \SO_{2n}(k)$ fixing a non-isotropic line in the representation $V$ (and acting on the orthogonal complement $W$). The representation
$$M_A = \sum_{i = 0}^d M_i \otimes \Sym^i(\bC^2)$$
of dimension $2n$ gives an A-parameter for $G = \SO_{2n+1}$ if and only if the $M_i$ are bounded selfdual representations of $WD(k)$ with $M_i$ symplectic for $i$ even and $M_i$ orthogonal for $i$ odd. Indeed, when these conditions are satisfied, $M_A$ is a symplectic representation of $WD(k) \times \SL_2(\bC)$ of  dimension $2n$.
Note that the action of $WD(k) \times \SL_2(\bC)$ on the Lie algebra of $\wG = \Sp_{2n}(\bC)$ is the representation $\Sym^2 M_A$. Similarly, the representation
$$N_A = \sum_{i = 0}^d N_i \otimes \Sym^i(\bC^2)$$
of dimension $2n$ is an A-parameter for $H = \SO_{2n}$ if and only if the $N_i$ are bounded selfdual  representations of $WD(k)$, with $N_i$ orthogonal for $i$ even, $N_i$  symplectic for $i$ odd and the quadratic character $\det N_A$ is given by the discriminant of the even orthogonal space $W$. Indeed, in this case $N_A$ is an orthogonal representation of the right dimension and determinant.
Note that the action of $WD(k) \times \SL_2(\bC)$ on the Lie algebra of $^L H = \O_{2n}(\bC)$ is the representation $\wedge^2 N_A$.
\vskip 5pt

We conjecture in \S \ref{S:classical} that there is a representation $\pi_G \otimes \pi_H$ in the L-packet associated to $M$ and $N$ with 
$$\dim \Hom_H(\pi_G \otimes \pi_H, \bC) = 1$$
if and only if the A-parameters $(M_A,N_A)$ form a relevant pair, and then the   representation $\pi_G \otimes \pi_H$ in the L-packet associated to $M$ and $N$ with $\dim \Hom_H(\pi_G \otimes \pi_H, \bC) = 1$ is unique.
Furthermore, this representation is determined by the character of the component group of the L-parameter, which is  given by the root numbers associated to the symplectic representation $M \otimes N$ using the recipe of \cite{GGP}. When $(M_A,N_A)$ form a relevant pair, we also show that the ratio of L-functions
\begin{equation}  \label{E:R2} 
L(M,N, s) = \frac{L(M \otimes N, s + ~1/2)}{L(\Sym^2 M \oplus \wedge^2N, s + 1)} \end{equation}
does not vanish at the point $s = 0$, but may have a pole. 

\vskip 5pt

In section \ref{S:A}, we offer a conjecture (Conjecture  \ref{relevant}) on which Arthur packets could have  representations with
$\Hom_H(\pi_G \otimes \pi_H, \bC) \not = 0$, though the conjecture here is not as precise as the one in previous section
for L-packets of Arthur type.  The point here is that we consider all representations in the A-packet and not just those in the associated L-packet. 
The less definitive nature of the conjecture is due to the fact that A-packets for classical groups are not disjoint. So, for example, 
Conjecture  \ref{relevant} predicts that if $\pi_G \times \pi_H$ is a representation belonging to an A-packet (say for  a pair of A-parameters $(M_A,N_A)$) and satisfies
$\Hom_H(\pi_G \otimes \pi_H, \bC) \not = 0$,  then  $\pi_G \times \pi_H$ belongs to some A-packet for a relevant pair $(M'_A, N'_A)$ of A-parameters.  We prove certain theorems in \S  \ref{S:A} to support this conjecture and give a general
construction of Arthur packets where multiplicity $>1$ is achieved. We also construct a counterexample
to the naive expectation that if  $\pi_G \times \pi_H$ is of Arthur type with A-parameter $(M_A , N_A)$ and
$\Hom_H(\pi_G \otimes \pi_H, \bC) \not = 0$, then $(M_A, N_A)$ must be relevant.

\vskip 5pt

We note that the notion of a relevant pair $(M_A,N_A)$ of symplectic and orthogonal representations makes sense in general, and we use this to extend the conjecture on Bessel models for $\SO_{2n+1} \times \SO_{2m}$ in \cite{GGP} to the non-tempered case. It also works well for conjugate symplectic and conjugate orthogonal representations, which allows us to extend the conjectures on Bessel models for $\U_{2n+1} \times \U_{2m}$ in \cite{GGP} to the non-tempered case. One can likewise formulate the analogous local conjecture in the setting of Fourier-Jacobi models for the symplectic/metaplectic groups and the unitary groups, following the procedure in \cite{GGP}. We omit the details in this paper and leave the precise formulation to the reader, using \cite{GGP} as a guide. 

\vskip 5pt

In \S \ref{S:global},  we consider the global setting over a number field. Here, one is considering the automorphic period integral of the automorphic forms belonging to the global L-parameters of Arthur type. In the non-tempered setting, these automorphic forms are not necessarily cuspidal and so some regularization may be needed to make sense of the period integrals. A systematic equivariant regularization procedure for reductive periods has been developed in a recent paper of Zydor \cite{Zy}, following earlier works of others.  Our local conjecture implies that there is at most one representation in the global L-packet (of Arthur type) which can have a non-vanishing period integral.  Whether this distinguished global representation is automorphic or not is governed by Arthur's multiplicity formula in which a certain quadratic character of the global component group (of the A-parameter) plays a prominent role. In the tempered case, this character is trivial, but in general it is given in terms of certain symplectic global root numbers built out of the adjoint representation. The interaction of Arthur's character with the distinguished character arising from the restriction problem is quite interesting and will be discussed in \S \ref{S:global2}.  In any case, we show that the global analog of the ratio of L-functions in (\ref{E:R}) or (\ref{E:R2}) is holomorphic at $s=0$ when $(M_A, N_A)$ is a relevant pair of global A-parameters and its corresponding global root number is $1$ when the distinguished representation in the global L-packet associated to $(M,N)$ is automorphic. We then
conjecture that its non-vanishing is equivalent to the non-vanishing of the (regularized) period integral.  
\vskip 5pt

In \S \ref{S:example}, we discuss several families of examples in low rank classical groups where the restriction problem has been addressed for non-tempered A-packets; these provide some additional support for our local conjecture.
\vskip 5pt
 It is interesting to note that in the ongoing work of Chaudouard and Zydor on the Jacquet-Rallis relative trace formula (which has been used to settle special cases of the global GGP conjectures for unitary groups), the notion of relevant pair of A-parameters seems to show up in the continuous spectrum on the spectral side. In another direction, we explain in \S \ref{S:descent} how the automorphic descent method discovered by Ginzburg-Rallis-Soudry \cite{GRS1, GRS2, GRS3, GRS4} and further extended in a recent work of Jiang-Zhang \cite{JZ} can be 
predicted using our global conjecture. 
\vskip 5pt

Thus, the main innovation of this paper is the introduction of the notion of a relevant pair of A-parameters,
which governs the branching laws of all classical groups, at least for representations belonging to
L-packets of Arthur type. It is meaningful for all local fields, and global fields, and makes sense for $\GL_n(k)$ as well as all classical groups,
and plays a role for all the branching problems studied in \cite{GGP} including Bessel models and Fourier-Jacobi models.\vskip 5pt

The  notion of relevant pair of A-parameters appears in this paper from 
three relatively  independent points of views. 
\vskip 5pt

\begin{enumerate}
\item The condition was discovered in the course of studying branching laws from $\GL_{n+1}(k)$ to $\GL_n(k)$ via the Bernstein-Zelevinsky filtration of a representation of $\GL_{n+1}(k)$ restricted to its mirabolic subgroup.
\vskip 5pt

\item After the work of Ichino-Ikeda,  cf. \cite{II}, it is natural to consider the ratio of L-functions $L(M,N, s)$ given in (\ref{E:R2}). From this L-function theoretic point of view,
 we prove in Theorem \ref{T:interlacing} that under some extra hypothesis,
   $L(M,N, s)$  has no zeros or poles at $s=0$ if and only if $(M, N)$ is a relevant pair of A-parameters.
\vskip 5pt

\item From the point of view of epsilon factors and tempered GGP, it is Theorem \ref{wald2} due
  to Waldspurger which brings out the notion of relevant pairs of A-parameters for those A-packets which
  contain a cuspidal representation.
  \end{enumerate}
   
Finally, from the global point of view, our local conjectures are still incomplete, as
Conjecture \ref{conj-class} 
only deals with an L-packet of Arthur type
which is a subset of the corresponding A-packet,  and Conjecture \ref{relevant} is not as precise as it could be.  It would have been ideal to have a precise conjecture for the entire A-packet, since any representation in the A-packet could have been the local component of an automorphic representation  and thus would intervene in the global period integral problem.
It will thus be very interesting to have a precise prediction for the sum of multiplicities over the entire A-packet (over relevant  pure inner forms) and a conjectural determination of the representations which have nonzero contribution. Unfortunately, despite trying for a few years, such a precise prediction continues to elude us.

\vskip 5pt

We end the introduction with a general comment on the strategy of the proof of tempered GGP,
where one exploits the transfer of  the stable distribution on a classical group associated to a tempered L-packet
to an invariant distribution on $\GL_n(k)$. 
The proof of tempered GGP uses such transfer to move  the branching problem  for  classical groups
to one on $\GL_n(k)$. 
For non-tempered A-packets, one still has such a transfer, but the stable distribution associated to an A-packet
is not simply the sum of irreducible characters. Instead, it is a linear combination of such with coefficients $\eta(z)$, where $z$ is the center of the Arthur $\SL_2$, so that there are  $\pm 1$ in the coefficients. Thus one is led to  wonder if there is a nicer
expression if, instead of summing over all multiplicities in the A-packet, one considers the signed sum of  these multiplicities reflecting  the signs in the stable characters. 
\vskip 5pt

As an example, one may consider a nontempered A-packet of $\U_3$  of the form $\{\pi_c,\pi_n\}$, with $\pi_c$ cuspidal and $\pi_n$ non-tempered.  The cuspidal representation $\pi_c$ belongs to a discrete L-packet $\{\pi_c,\pi_d\}$, 
where $\pi_d$ is a non-cuspidal discrete series representation. There is a exact sequence arising from a principal series representation:
$$0\rightarrow \pi_d  \rightarrow {\Ind} \rightarrow \pi_n \rightarrow 0,$$
so that
$$\pi_c-\pi_n = (\pi_c+\pi_d)-(\pi_d+\pi_n) = (\pi_c+\pi_d) - \Ind$$
is stable. This lends some support to the expectation that for a tempered representation $\pi$ of $\U(2)$,
$$\dim \Hom_{\U(2)}[\pi_c,\pi] - \dim \Hom_{\U(2)} [\pi_n,\pi],$$
may have a nicer expression than 
$$\dim \Hom_{\U(2)}[\pi_c,\pi] + \dim \Hom_{\U(2)} [\pi_n,\pi].$$

\vspace{5mm}

\noindent{\bf Acknowledgement:} Much of this paper was written when the first and the third authors were at MSRI in the Fall of 2014, where we had  discovered the condition -- {\it relevant pair} of A-parameters -- appearing in Conjecture \ref{conj-gln}, about branching laws from
$\GL_{n+1}(k)$ to $\GL_n(k)$, and then the analogous conjectures for classical groups.   The first and the third authors must thank  MSRI for a very stimulating semester. We thank Zhiwei Yun
for section \ref{cor} where he defines {\it correlators}, a concept equivalent to our  relevant pair of A-parameters.
We thank C. M{\oe}glin and J.-L. Waldspurger for the
contents of section \ref{MW}. Section \ref{S:GLn} has benefited greatly from incisive comments of K.Y. Chan. The authors thank the referee for a careful
reading of the manunscript and useful comments.

%In fact, at that time, we had.
%hoped that our proof of the  Conjecture \ref{conj-gln} given in Theorem \ref{thm-gln} could be fine tuned to give a
%full proof of Conjecture \ref{conj-gln} delaying this paper. Conjecture \ref{conj-gln}
%appears to be much harder than we had anticipated. There is now a proof of one part of
%Conjecture \ref{conj-gln} due to M. Gurevitch \cite{Gu}, in which he proves that if branching happens, then the parameters are as proposed here. 
\vskip 10pt

\section{Notation and Preliminaries}
In this paper, unless otherwise specified, we will use $k$ to denote a local field. All the conjectures in this paper are
formulated for all local fields,  archimedean and non-archimedean, but no proofs are given for archimedean local fields. It is hoped
that in due course, not only the few proofs we give in this paper, but also the conjectures that we formulate, will be considered
by others for archimedean fields.
\vskip 5pt

For a local field $k$, we let $W(k)$ be the Weil group of $k$ and $WD(k)$ the Weil-Deligne group of $k$.
So
\[ WD(k) = \begin{cases}
W(k) \text{ if $k$ is archimedean;} \\
W(k) \times \SL_2(\C), \text{ if $k$ is non-archimedean.} \end{cases} \]
In the non-archimedean case, the $\SL_2(\C)$ which comes up here is called the Deligne $\SL_2$.
\vskip 5pt

For a reductive algebraic group $G$ over $k$, let ${}^LG =
\wG(\C) \rtimes W(k)$ be the L-group of $G$ with $\wG(\C)$ the dual group
over $\C$. Langlands parameters for $G$ (L-parameters for short) are admissible homomorphisms
$$\phi: WD(k) \rightarrow  {^L}G,$$
up to conjugacy by $\wG(\C)$. Arthur parameters for $G$  (A-parameters for short) are admissible homomorphisms
$$\phi_A: WD(k) \times \SL_2(\bC) \rightarrow  {^L}G$$
up to conjugacy by $\wG(\C)$, where the restriction of $\phi_A$ to the Weil-Deligne group $WD(k)$
is an admissible homomorphism with bounded image for $W(k)$ and the restriction to $\SL_2(\C)$
 is algebraic. The extra $\SL_2(\C)$ which enters here will be called the Arthur $\SL_2(\C)$. 
\vskip 5pt

One can obtain an L-parameter from an A-parameter as discussed in the introduction, using the map
 $WD(k) \rightarrow WD(k) \times \SL_2(\bC)$ given by
$$w \rightarrow (w, \diag(|w|^{1/2}, |w|^{-1/2})).$$
Composing $\phi_A$ with this homomorphism gives an L-parameter $\phi$. The association
\[  \phi_A \mapsto \phi  \]
is an injective map from the set of A-parameters to the set of L-parameters,
and we will call  its image the set of L-parameters of Arthur type. 
\vskip 5pt

In the case of classical groups as discussed in \cite{GGP},  A-parameters are representations of $  WD(k) \times \SL_2(\bC)$
on finite dimensional complex vector spaces which may come equipped with a symmetric or skew-symmetric bilinear form. For
an A-parameter $M_A$, with associated L-parameter $M$,  we will often denote the associated complex vector
space (the representation space of $M_A$ or $M$) as $M$; sometimes, for sake of clarity, we may use a different symbol
$V$ or $W$ for the representation space of $M_A$ or $M$.

\vskip 5pt

By the work of Harris-Taylor  \cite{HT} for $\GL_n(k)$, Arthur  \cite{Art2} for orthogonal and symplectic groups and Mok \cite{Mok}  for unitary groups, the local Langlands conjecture  for these quasi-split  classical groups over local fields is now known.  The
extended version due to Vogan  \cite{Vo}, involving pure inner forms of these groups,  is also known thanks to the work of
Kaletha-Minguez-Shin-White \cite{KMSW} and M{\oe}glin \cite{M4}. We will assume this through the work, referring to \cite{GGP} for precise assertions. For the notion of A-packets, and their
relationship to A-parameters, we refer to \cite{Art1}, \cite{Art2}, as well as to many works of M{\oe}glin \cite{M1}, \cite{M2}, \cite{M3}.
Both for L-packet and A-packet, we will simultaneously consider all relevant pure inner forms of the pair of groups involved,
and use what may be called Vogan L-packet and Vogan A-packet. We will not dwell further on these matters, as they are totally analogous to \cite{GGP}.
\vskip 5pt

As is well-known,  $\Sym^a(\C^2)$ is the unique $(a+1)$ dimensional irreducible representation of  $\SL_2(\bC)$. We will
often denote  this $(a+1)$ dimensional irreducible representation of  $\SL_2(\bC)$ as $[a+1]$.
If we regard $[a+1]$ as a representation of the Arthur $\SL_2(\C)$, then $[a+1]$ is an A-parameter for $\GL_{a+1}(k)$ whose corresponding A-packet is the singleton set containing only the trivial representation. 
 Hence, by abuse of notation, we will also
denote by $[a+1]$ the trivial representation of $\GL_{a+1}(k)$. The context will make it clear whether $[a+1]$ stands for
an A-parameter, or  the trivial representation of $\GL_{a+1}(k)$.  
 \vskip 5pt

If $\pi_1$ is a representation of $\GL_m(k)$ and $\pi_2$ of $\GL_n(k)$, we let $\pi_1 \times \pi_2$ be  the representation of
$\GL_{m+n}(k)$ parabolically induced from the representation $\pi_1 \boxtimes \pi_2$ of the Levi subgroup $\GL_m(k) \times \GL_n(k)$ of $\GL_{m+n}(k)$. In particular, $[a] \times [b]$ denotes an irreducible (unitary)
representation of $\GL_{a+b}(k)$ which has 
A-parameter  $[a] \oplus [b]$. In \S \ref{S:GLn}, we shall give a more precise description of the A-packets for $\GL_n(k)$. 
\vskip 5pt

For any integer $n\geq 1$,  we set
\[ \nu := |\det|: \GL_n(k) \rightarrow \C^\times. \]
 When $n=1$ so that  $\GL_n(k)=k^\times$, we will also view $\nu$  as the absolute value map on $WD(k)$. 
\vskip 5pt

An L-parameter for a classical group $G$ (not $\GL_n(k)$) $$\phi: WD(k) \rightarrow  \GL(V),$$
is called a discrete L-parameter if it does not factor through a proper Levi subgroup of the dual group $\widehat{G}$; equivalently, if  it is a multiplicity free sum of
selfdual or conjugate-selfdual 
irreducible representations of $WD(k)$ of the correct parity. This terminology originates from the fact that an L-parameter of $G$ is discrete if and only if
the associated representations of $G(k)$ are discrete series representations. Similarly,
an A-parameter for a classical group (not $\GL_n(k)$) $$\phi_A: WD(k)\times \SL_2(\C)  \rightarrow  \GL(V),$$
is called a discrete A-parameter if it is a multiplicity free sum of selfdual or conjugate-selfdual
irreducible representations of $WD(k) \times \SL_2(\C)$.
\vskip 5pt

For all the number theoretic background on $L$ and $\epsilon$ factors, we refer to the article \cite{Tate}, usually without explicit mention. All our global $L$-functions will be completed $L$-functions $L(s, \Pi)$
for an automorphic representation $\Pi$ on $\GL_n(\A_F)$ for $F$
a global field, or an automorphic representation on a classical group treated as
an automorphic representation $\Pi$ on $\GL_n(\A_F)$ after transfer. 

\section{Relevant Pair of A-Parameters}  \label{S:relevant}

In this section, we formulate the notion of a relevant pair of A-parameters $(M_A, N_A)$ in the context of classical groups (including the case of $\GL$). This notion plays a pivotal role in this paper. We also discuss some basic properties of this notion.
\vskip 5pt

With $k$ a local field, recall that a finite dimensional complex representation $M_A$ of $WD(k) \times \SL_2(\bC)$ which arises from an Arthur parameter (of a classical group) has a canonical decomposition of the form
$$M_A = \sum_{i = 0}^d M_i \otimes \Sym^i(\bC^2),$$
where the $M_i$ are bounded admissible representations of $WD(k)$. We say that two representations $M_A$ and $N_A$ form a relevant pair of A-parameters if we have a decomposition of the respective representations of $WD(k)$ as
$$M_i = M_i^+ + M_i^-  \quad \text{and}  \quad N_i = N_i^+ + N_i^-,$$
with the property that
$$\text{$M_i^+ = N_{i+1}^-$ for $i \geq 0$} \quad \text{and} \quad \text{$M_i^- = N_{i-1}^+$ for $i \geq 1$.}$$

Therefore, if we write the decomposition of $M_A$ as
$$M_A = M_0^+ + M_0^- + \sum_{i = 1}^d (M_i^+ + M_i^-) \otimes \Sym^i (\bC^2),$$
we have the following decomposition of $N_A$:
$$N_A = N_0^- + \sum_{i = 0}^d M_i^+ \otimes \Sym^{i+1} (\bC^2) + \sum_{i = 1}^d M_i^- \otimes \Sym^{i-1}(\bC^2).$$

The notion of a relevant pair $(M_A,N_A)$ is symmetric, as we can also write the conditions on the summands as 
$$  \text{$N_i^+ = M_{i+1}^-$ for $i \geq 0$ }  \quad \text{and} \quad  \text{$N_i^- = M_{i-1}^+$ for $i \geq 1$.}  $$
 Therefore, if we write the decomposition of $N_A$ as
$$N_A = N_0^+ + N_0^- + \sum_{i = 1}^d (N_i^+ + N_i^-) \otimes \Sym^i (\bC^2),$$
we have the following decomposition of $M_A$:
$$M_A = M_0^- + \sum_{i = 0}^d N_i^+ \otimes \Sym^{i+1} (\bC^2) + \sum_{i = 1}^d N_i^- \otimes \Sym^{i-1}(\bC^2).$$
Note that the two summands $M_0^-$ and $N_0^-$ in a relevant pair are not constrained by the other parameter. In particular, any two bounded representations $(M_0,N_0)$ of $WD(k)$ with the trivial action of $\SL_2(\bC)$ form a relevant pair of A-parameters. 
\vskip 5pt

The decomposition of the representations $M_i$ and $N_i$ of $WD(k)$ in a relevant pair into components $M_i^{\pm}$ and $N_i^{\pm}$ is unique. Indeed, suppose that the highest (non-zero) summand in the decomposition of $M_A$ is $M_d$. Then the highest summand in the decomposition of $N_A$ is either $N_{d-1}, N_d$ or $N_{d+1}$. We now consider the three cases separately:
\vskip 5pt

\begin{itemize}
\item  Suppose first that it is $N_d$. Then we conclude that $M_d = M_d^-$ and $N_d = N_d^-$. This implies that $N_{d-1}^+ = M_d^-$, which determines $N_{d-1}^-$. Similarly, $M_{d-1}^+ = N_d^-$, which determines $M_{d-1}^-$. Continuing to descend in this manner determines all of the decompositions. 
\vskip 5pt

\item Next assume that the highest summand in the decomposition of $N_A$ is $N_{d-1}$. Again we conclude that $M_d = M_d^-$, but now we also have $M_{d-1} = M_{d-1}^-$ as $N_d = 0$. Then $N_{d-1}^+ = M_d^-$, which determines $N_{d-1}^-$ and $N_{d-2}^+ = M_{d-1}^-$, which determines $N_{d-2}^-$. Continuing to descend in this manner gives the full decomposition. 
\vskip 5pt

\item Finally, if the highest summand in the decomposition of $N_A$ is $N_{d+1}$, then one can switch the roles of $M_A$ and $N_A$ and use the previous argument.
\end{itemize}
\vskip 5pt

We will have occasion to use the following lemma about relevant pair of A-parameters whose simple proof is omitted.

\begin{lemma} \label{EPlemma} Let
$$M_A = M_0^+ + M_0^- + \sum_{i = 1}^d M_i \otimes \Sym^i (\bC^2),$$
$$N_A = N_0^+ +N_0^- + \sum_{i = 1}^d N_i \otimes \Sym^{i} (\bC^2),$$
be a relevant pair of A-parameters. Then we have the relations,
 \begin{eqnarray*}
  \sum_{2i-1 \geq 1}M_{2i-1} & = & \sum_{2i \geq 0} N_{2i} - N^-_0\\
  \sum_{2i-1 \geq 1} N_{2i-1}  & = &  \sum_{2i\geq 0}M_{2i}  - M^-_{0}  .
\end{eqnarray*}

\end{lemma}

Suppose now that $(M_A, N_A)$ is a  pair of representations with $M_A$  symplectic and $N_A$ even orthogonal. Then $M_A$ is the A-parameter of an odd special orthogonal group whereas $N_A$ is that of an even special orthogonal group.  If $(M_A, N_A)$ is relevant, the argument we gave for the unicity of the decomposition of $M_i$ and $N_i$ shows that the summands $M_i^{\pm}$ are each symplectic for $i$ even and orthogonal for $i$ odd, and the summands $N_i^{\pm}$ are each orthogonal for $i$ even and symplectic for $i$ odd. 
\vskip 5pt

Likewise, suppose that  $k/k_0$ is a separable quadratic extension and $(M_A, N_A)$ is a   pair of representations of $WD(k) \times \SL_2(\bC)$ with $M_A$ conjugate symplectic and $N_A$ conjugate orthogonal, then $M_A$ is an A-parameter for an even unitary group whereas $N_A$ is an A-parameter of an odd unitary group. When $(M_A, N_A)$ is relevant, one sees that the summands $M_i^{\pm}$ and $N_i^{\pm}$ are conjugate symplectic and conjugate orthogonal respectively.
\vskip 5pt

On the other hand, suppose that $F$ is a global field with conjectural Langlands group $L(F)$, so that global A-parameters of classical groups can be thought of as finite-dimensional complex representations of $L(F) \times \SL_2(\bC)$.  Of course, since the existence of $L(F)$ is not known, one needs to interpret an irreducible $n$-dimensional representation of $L(F)$ as a cuspidal automorphic representation of $\GL_n(\A_F)$.  In the context of classical groups, one needs to interpret an irreducible  $n$-dimensional symplectic (respectively orthogonal) representation of $L(F)$ as a cuspidal automorphic representation of $\GL_n(\A_F)$  for which the exterior square (respectively symmetric square L-function) has a pole at $s=1$.  With this caveat,  one can similarly define the notion of a relevant pair of global A-parameters.  
\vskip 5pt

Since the work of Ichino-Ikeda, cf. \cite{II}, it is natural to consider in the global setting the ratio of L-functions:
\[  L(M,N, s) = \frac{L(M \otimes N^{\vee}, s + ~1/2) \cdot L(M^{\vee} \otimes N, s + ~1/2)}{ L(M \otimes M^{\vee}, s + 1)\cdot  L(N \otimes N^{\vee}, s + 1)} \]
if $M \times N$ is an L-parameter for $\GL_m \times \GL_n$,   or the ratio 
\[  L(M,N, s) = \frac{L(M \otimes N, s + ~1/2)}{L(\Sym^2 M \oplus \wedge^2N, s + 1)}  \]
if $M \times N$ is an L-parameter for $\SO_{2m+1} \times \SO_{2n}$. 
One may of course consider $L(M,N, s)$ in the local context as well. 
 We conclude this section by highlighting some results about the analytic properties of $L(M,N, s)$ at $s=0$. The proofs of these results, which proceed by explicit computation, 
will be given in \S \ref{S:L-GL} and \S \ref{classical} at the end of the paper.
\vskip 5pt

\begin{thm} \label{pole}
Let $k$ be a non-archimedean local field and  let $(M_A, N_A)$ be a relevant pair of A-parameters for $\GL_m(k) \times \GL_n(k)$ with associated pair of L-parameters $(M,N)$.
Then the order of pole at $s=0$ of 
\[  L(M,N, s) = \frac{L(M \otimes N^{\vee}, s + ~1/2) \cdot L(M^{\vee} \otimes N, s + ~1/2)}{ L(M \otimes M^{\vee}, s + 1)\cdot  L(N \otimes N^{\vee}, s + 1)} \]
  is greater than or equal to zero.
\end{thm}
\vskip 5pt

\begin{thm} \label{interlacing}
 Let $k$ be a non-archimedean local field and  let $(M_A, N_A)$ be a pair of A-parameters for $\SO_{2m+1}(k) \times \SO_{2n}(k)$ with associated pair of L-parameters $(M,N)$.
 \vskip 5pt
 
 (i) If $(M_A, N_A)$ is a relevant pair of A-parameters, then  the order of pole at $s=0$ of 
 the function
 \[  L(M,N, s) = \frac{L(M \otimes N, s + ~1/2)}{L(\Sym^2 M \oplus \wedge^2N, s + 1)}  \]
 is greater than or equal to zero.

 \vskip 5pt
 
 (ii) Suppose  that $M_A$ and $N_A$ are multiplicity-free representations of $WD(k) \times \SL_2(\bC)$ on which the Deligne $\SL_2(\C)$ acts trivially. 
  Then, at $s = 0$,  the function
 $  L(M,N, s)$
has  a zero of order $\geq 0$. It  has neither a zero nor a pole at $s=0$ if and only if
$(M_A, N_A)$ is a relevant pair of A-parameters.
\vskip 5pt

\end{thm}

\section{Correlator} \label{cor}
In this section, we describe an elegant formulation of the notion of a relevant pair of A-parameters which is due to Zhiwei Yun. 
\vskip 5pt

Consider  a pair $(M_A, N_A)$ of selfdual finite-dimensional representations of $WD(k) \times \SL_2(\C)$
with $M_A$ symplectic and $N_A$ orthogonal, realized on vector spaces  $V$  and $W$ respectively.  
The action of the diagonal torus of $\SL_2(\C)$ induces a $\Z$-grading on $V$ and $W$.  More precisely,  if we identify $\Gm$ with the diagonal torus, 
 taking $t \in \Gm$ to the diagonal matrix $(t, t^{-1})$,
  then the degree $n$ part of $V$ is the eigenspace for the character $t \mapsto t^n$.  
\vskip 5pt

Let $e$ (resp. $e'$) be the nilpotent endomorphism of $V$ (resp. $W$) given by the image of the usual upper triangular element in
the Lie algebra $\mathfrak{sl}_2$ of $\SL_2(\C)$.  Then the action of $e$ and $e'$ shifts degree by $2$ on $V$ and $W$. In other words, $e$ and $e'$ are degree $2$ elements in $\End(V)$ and $\End(W)$  (equipped with the induced grading). 
\vskip 5pt

Here is the key definition of Zhiwei Yun:
\vskip 5pt

\begin{definition}  For   a pair $(M_A, N_A)$ of finite-dimensional representations of $WD(k) \times \SL_2(\C)$ together with invariant
  nondegenerate bilinear forms,
  a correlator for $(M_A, N_A)$ is a $WD(k)$-equivariant linear map 
\[ T: V \rightarrow W, \]
  such that
  \begin{enumerate}
  \item  $T$ shifts degree by 1, i.e. $T$ is an element of degree $1$ in $\Hom(V, W)$;
     \item $T^*T=e$ and $TT^*=e'$, where $T^*:W \rightarrow V$ is the adjoint of $T$.
       \end{enumerate}
  \end{definition}
  \vskip 5pt
  
\begin{remark} Observe that  condition (2) forces the parity of the forms on $V$ and $W$ to be opposite, unless $e$ and $e'$ are zero. \end{remark}
\vskip 5pt

Now we have the main observation:
\vskip 5pt

\begin{lemma} 
The pair $(M_A,N_A)$ of local A-parameters is relevant if and only if there exists a correlator $T: V \rightarrow W$.
 \end{lemma}
\vskip 5pt

\begin{proof}Given a relevant pair of A-parameters $(M_A,N_A)$, we construct a correlator as follows.
 Writing
 \[ V  = \bigoplus_n V_n \otimes \Sym^n(\C^2), \]  
  the action of the diagonal torus of $\SL_2(\C)$ gives a further decomposition  of the form 
  \[ V = \sum_n \sum_a V_n \otimes t_n^a, \]
  where  the inner sum is taken over integers $|a| \leq  n$ with
  $a \equiv n (\bmod 2)$.  Moreover, we have written $t_n^a$ for a nonzero eigenvector in 
  $\Sym^n(\C^2)$ for the character $t \mapsto t^a$.
We normalize the choice of $t_n^a$ so that  the  nilpotent element $e$  maps $V_n \otimes t_n^a$ to   $V_n \otimes t_n^{a+2}$.  In particular, the kernel of $e$ on $V$ is $\sum_n V_n \otimes t_n^n$.
\vskip 10pt

 Now  we define a correlator $T$ by giving $WD(k)$-equivariant surjective maps
 \[    \begin{cases}
   V_n^+ \otimes t_n^a   \longrightarrow W_{n+1}^- \otimes t_{n+1}^{a+1} \\
    V_n^- \otimes t_n^a \longrightarrow W_{n-1}^+ \otimes t_{n-1}^{a+1}\end{cases} \]
 for each $n$. In particular, 
 \[ {\rm Ker}(T) = \bigoplus_n V_n^- \otimes t_n^n,\]
 and one checks easily that $T$ is a correlator.
 \vskip 5pt
 
 Conversely, one argues that a correlator $T$ must send $V_n \otimes t_n^a$ to
 $W_{n+1}^- \otimes t_{n+1}^{a+1} \oplus  W_{n-1}^+ \otimes t_{n-1}^{a+1}$ for any $|a| \leq n$. Then ${\rm Ker}(T)$ determines $V_n^-$ for each $n$ and hence the representation $V_n^+$ (since $V_n$ is semisimple).  Hence the existence of $T$ implies that 
 the pair $(M_A, N_A)$ is relevant.
   \end{proof}
\vskip 10pt

\begin{remark} 
In the
 trivial case, when $e=0$ and $e'=0$, one can take $T=0$ and $(M_A,N_A)$ is indeed always relevant in this case.
 \end{remark}
 \vskip 5pt
 
 For $(M_A,N_A)$ a pair of local A-parameters for $\GL_n \times \GL_m$, we can apply the same reformulation above to the selfdual A-parameters
 $M_A+M_A^\vee$ and $N_A+N_A^\vee$ (on vector spaces $V,W$ of dimensions $2n,2m$).  A correlator in this situation becomes a pair of maps of $WD(k)$-representations $T:V \rightarrow W$ and $T^*: W \rightarrow V$, each of degree 1, such that $T^*T=e$  and $TT^*=e'$. One can then check by a similar argument as above that the existence of such a correlator is equivalent to the relevance of $(M_A,N_A)$.
 
\vskip 10pt

 The notion of a correlator $T$ defined above should remind 
the reader familiar with the theory of reductive dual pairs (in the sense of Howe)  of the moment map which arises there.
Let us explicate this connection here.  Given a pair $(M_A, N_A)$ of A-parameters
where $M_A$ (respectively $N_A)$ is a symplectic  (resp. orthogonal) representation of $WD(k)$ on a vector space $V$ (resp. W), one has a symplectic representation $V \otimes W$ of $WD(k)$:
\[  WD(k) \longrightarrow \Sp(V) \times \O(W) \longrightarrow \Sp(V \otimes W). \]
The $\Sp(V) \times \O(W)$-symplectic variety $V \otimes W$ is Hamiltonian and possesses a $WD(k)$-equivariant moment map
\[  \mu = \mu_V \times \mu_W: V \otimes W \longrightarrow   \mathfrak{sp}(V)  \times \mathfrak{so}(W), \]
where $\mathfrak{sp}(V)$  (resp. $\mathfrak{so}(W)$) is the Lie algebra of $\Sp(V)$ (resp. $\SO(W)$)
which we prefer to write as a double fibration:
 \[
 \xymatrix{ &  V \otimes W
 \ar[dl]_{\mu_V} \ar[dr]^{\mu_W}& \\
 \mathfrak{sp}(V)
   & &
 \mathfrak{so}(W).
     &   }
\]

Since $V$ and $W$ are equipped with non-degenerate bilinear forms, we have isomorphisms
\[ V \otimes W \cong \Hom(V,W) \cong \Hom(W,V). \]
The maps $\mu_V$ and $\mu_W$ are then given by:
\[  \mu_V(T) = T^* T \quad \text{and} \quad \mu_W(T) = T  T^* \quad \text{for $T \in \Hom(V,W)$.} \]
The double fibration diagram gives one a correspondence between the set of $\Sp(V)$-adjoint orbits and $\O(W)$-adjoint orbits: say that a $\Sp(V)$-adjoint orbit  $\mathcal{O}_V$ and an $\O(W)$-adjoint orbit $\mathcal{O}_W$ correspond if there exists $T \in V \otimes W$ such that
\[ \mu_V(T) \in \mathcal{O}_W \quad \text{and} \quad \mu_W(T) \in \mathcal{O}_W. \]
\vskip 5pt

Now we may reformulate the notion of a correlator in terms of this moment map. 
Given a pair of A-parameters $(M_A, N_A)$,  recall that one has a pair of nilpotent elements
\[ e \in \mathfrak{sp}(V) \quad \text{and} \quad e' \in \mathfrak{so}(W)\]
which are fixed by $WD(k)$. Recalling that $V \otimes W \cong \Hom(V,W)$ has a $\Z$-grading, 
we see that a correlator for $(M_A, N_A)$ is a degree $1$ element $T \in \Hom(V,W)$ which is fixed by $WD(k)$ and such that 
\[ \mu_V(T) = e \quad \text{and} \quad \mu_W(T) = e'. \]
Hence we have a nice geometric formulation of the notion of relevance: $(M_A, N_A)$ is relevant if and only if the nilpotent orbits of $e$ and $e'$ are in the moment map correspondence  via a $WD(k)$-fixed degree $1$ element $T$.
\vskip 5pt

We note that  such correlator maps $T$ have also appeared in the work of Gomez-Zhu \cite{Z} where they studied the transfer of generalized Gelfand-Graev models (attached to nilpotent orbits) under the local theta correspondence. In their work, the  correlator maps 
were considered on the side of representation theory of the real or p-adic groups, whereas in our work, they appeared on the side of the Langlands dual groups.
%{\color{red}{Question: The correlator $A$ is unique up to  an isomorphism $(\phi_1,\phi_2)$ of $(M,N)$ as $WD(k) \times \SL_2(\C)$-modules, making certain diagrams commute?}}
\vskip 15pt

\section{Local Conjecture for $\GL_{n}$}  \label{S:GLn}

Having introduced the key notion of relevant pairs of A-parameters, we can begin the consideration of restriction problems. 
In this section, we discuss the restriction problem for   $\GL_n(k) \subset \GL_{n+1}(k)$  in the context of  L-parameters of Arthur type. We first explicate the representations of $\GL_n(k)$ which belong to such L-packets. 

\vskip 10pt

An A-parameter of $\GL_n(k)$ is an $n$-dimensional representation of $WD(k) \times \SL_2(\bC)$  of the form
\[  M_A = \bigoplus_{i=1}^r   M_i \boxtimes \Sym^{d_i}(\bC^2), \]
where $M_i$ is an irreducible $m_i$-dimensional representation of $WD(k)$.  As discussed in the introduction, $M_A$ gives rise to an L-parameter $M$. Moreover, the local A-packet associated to $M_A$ is equal to the local L-packet associated to $M$ which is a singleton set. We can describe this unique representation $\pi_M$ as follows. 
By the local Langlands correspondence, each $M_i$ corresponds to a discrete series representation $\pi_{M_i}$ of $\GL_{m_i}(k)$. The representation in the A-packet associated to the A-parameter $M_i \otimes \Sym^{d_i}(\bC^2)$ is a Speh representation ${\rm Speh}(M_i, d_i)$, which is the unique irreducible quotient of the standard module
\[   \pi_{M_i} |\det|^{d_i/2}   \times \pi_{M_i} |\det|^{d_i/2 -1} \times......\times \pi_{M_i} |\det|^{-d_i/2}, \]
where we have used the standard notation for parabolic induction.  Then the representation $\pi_M$ is given by the irreducible parabolic induction
\[  \pi_M  =    {\rm Speh}(M_1, d_1)  \times  {\rm Speh}(M_2, d_2) \times....\times  {\rm Speh}(M_r, d_r). \]
As an example, when $M = \chi \boxtimes \Sym^{n-1}(\C^2)$ with $\chi$ a 1-dimensional character of $WD(k)$, the associated representation $\pi_M$ is the character $\chi \circ \det$ of $\GL_n(k)$. 

\vskip 5pt

We can now consider the restriction problem. Suppose that $M_A$ is an A-parameter of $\GL_{n+1}(k)$ and $N_A$ one for $\GL_n(k)$. 
Let $\pi_M$ and $\pi_N$ be the irreducible representations in the respective L-packets (of Arthur type), which we have described above. 
By \cite{AGRS}, it  is known that
\[  \dim \Hom_{\GL_n(k)}( \pi_M, \pi_N)  \leq 1. \]
When $M_A$ and $N_A$ are tempered A-parameters, so that $\pi_M$ and $\pi_N$ are both tempered (and  generic), it is well-known that the above Hom space is $1$-dimensional.
In general, we make the following conjecture.
 \vskip 5pt
 
\begin{conj} \label{conj-gln}
In the above context ($k$ any local field, archimedean or non-archimedean),  
\[  \Hom_{\GL_n(k)}(\pi_M, \pi_N) \ne 0 \Longleftrightarrow  \text{the pair $(M_A, N_A)$ of A-parameters is relevant.}  \]
\end{conj}

 \vskip 5pt
 Of course, one could have formulated a restriction problem for an arbitrary $\GL_n \times \GL_m$, i.e. including the case of Bessel
 models and Fourier-Jacobi models. We highlight this particular case here because of its intrinsic simplicity, and will defer the general case  to the next section where we deal with classical groups.  
 
 \vskip 5pt
 Before proceeding further, it is good to do a reality check. Suppose that 
 \[  M_A = \Sym^n(\C^2) \]
 is the A-parameter of the trivial representation of $\GL_{n+1}(k)$,
 so that $\pi_M$ is the trivial representation $\GL_{n+1}(k)$.  Then its restriction to $\GL_n(k)$ is certainly the trivial representation, i.e. the only A-parameter $N_A$ of $\GL_n(k)$ for which $\Hom_{\GL_n(k)}(\pi_M, \pi_N) \ne 0$ 
 is 
 \[  N_A = \Sym^{n-1}(\C^2). \]
  The reader can easily verify that the only $N_A$ (of dimension $n$) such that $(\Sym^n(\C^2), N_A)$ is relevant is $N_A = \Sym^{n-1}(\C^2)$. 
 \vskip 5pt
 
 On the other hand, if we start with $N_A = \Sym^{n-1}(\C^2)$, then the A-parameters $M_A$ (of dimension $(n+1)$) of $\GL_{n+1}(k)$ such that $(M_A, N_A)$ is relevant are:
 \vskip 5pt
 
 \begin{itemize} 
 \item $M_A = \Sym^n(\C^2)$, or
 \item $M_A = \Sym^{n-2}(\C^2)  \oplus  M_0$ where $M_0$ is a bounded 2-dimensional representation of $WD(k)$.
 \end{itemize}
 This reflects the known fact that the irreducible admissible representations of $\GL_{n+1}(k)$ which are of Arthur type and which are $\GL_{n}(k)$-distinguished are either trivial or theta lifts of tempered representations of $\GL_2(k)$
 (see the paper \cite{Venk} for the precise results for {\em all} irreducible representations of $\GL_{n+1}(k)$ for $k$
 a non-archimedean local field).

 \vskip 5pt
 The rest of this section is devoted to proving a special case of Conjecture \ref{conj-gln}.
\vskip 5pt

\begin{thm} \label{thm-gln}
Let $k$ be a non-archimedean local field. Consider a pair of local A-parameters $(M_A, N_A)$ for $\GL_{n+1}(k) \times \GL_n(k)$ satisfying the following extra properties: 
\vskip 5pt

\begin{itemize}
\item[(a)]  on every irreducible summand of $M_A$ or $N_A$, at least one of the two $\SL_2(\C)$'s in
 \[ WD(k) \times \SL_2(\C) =  W(k)\times \SL_2(\C) \times \SL_2(\C) \]  
acts trivially;
 \vskip 5pt
 
\item[(b)]  for any irreducible representation $\rho$ of $W(k)$ and
  $a \geq 1$ an integer,
 \[ \rho \otimes \C \otimes \Sym^a(\C^2)  \oplus \rho \otimes \Sym^a(\C^2) \otimes \C \nsubseteq M_A, \]
 and likewise for $N_A$.  
  \end{itemize} 
   Then  
   \[  {\rm Hom}_{\GL_n(k)}(\pi_M,\pi_N) \not = 0   \Longleftrightarrow \text{$(M_A, N_A)$ is  relevant.} \]
  In particular, the above properties are satisfied and hence the result holds in the case when
  the Deligne $\SL_2(\C)$ (i.e. the $\SL_2(\C)$ in $WD(k)$) acts trivially on both $M_A$ and $N_A$.
  \end{thm}

\begin{proof}
  The proof proceeds by analyzing the restriction of the irreducible representation $\pi_M$
  of $\GL_{n+1}(k)$ to its mirabolic subgroup.
  This restriction was described in a classic theorem of Bernstein-Zelevinsky in \cite{BZ},
   in terms of certain induced representations arising from the Bernstein-Zelevinsky derivatives
  $\pi_M^i$ of the representation $\pi_M$. Since the mirabolic subgroup contains  $\GL_n(k)$, this gives one a way to understand 
the restriction of $\pi_M$ to $\GL_n(k)$. 
\vskip 5pt

By this theorem of Bernstein-Zelevinsky, if $ {\rm Hom}_{\GL_n(k)}(\pi_M,\pi_N) \not = 0$, then 
for some $i \geq 0$, there exists
\vskip 5pt

\begin{itemize}
\item an irreducible composition factor $A$ of $\nu^{1/2} (\pi_M)^{i+1}$,
\item  an irreducible composition factor $B$ of  $((\pi_N^\vee)^{i})^\vee$,
\end{itemize}
such that
\[ {\rm Hom}_{\GL_{n-i}(k)} (A, B) \not = 0. \]

For the converse, given some $(i, A, B)$ as above,
suppose that  we are lucky enough to have the further property:
\vskip 5pt

\begin{itemize}
\item for any $ j\leq i$ and any $(C,D) \not = (A,B)$, where $C$ is an irreducible composition factor  of
$\nu^{1/2} (\pi_M)^{j+1}$ and $D$ is an irreducible composition factor of
$((\pi_N^\vee)^{j})^\vee$,  
$${\rm Ext}^1_{\GL_{n-j}(k)}[ C, D] = 0.$$
\end{itemize}
Then the nonzero homomorphism from a
certain submodule of $\pi_M$ to $\pi_N$ represented by a homomorphism in
${\rm Hom}_{\GL_{n-i}(k)}[ A, B] \not = 0,$ extends to all of $\pi_M$.

To have this vanishing of $\Ext^1[C,D]$, the simplest reason 
would be when the cuspidal support of any irreducible composition factor of $ \nu^{1/2} (\pi_N^\vee)^j$ and of
$ (\pi_M^{j+1})^\vee$ are different (except for a unique choice $(A,B)$ among all pairs $(C,D)$).  
The conditions (a) and (b) in the theorem are imposed to ensure this. 
\vskip 10pt

Denote by $[d]$ the trivial representation of $\GL_d(k)$, and by $[d]^{\times m}$ the irreducible admissible representation of $\GL_{d\cdot m}(k)$
which is $[d] \times \cdots \times [d]$. If $\sigma$ is a cuspidal representation of $\GL_r(k)$, let $\sigma[d]$ be the Speh representation 
of $\GL_{r\cdot d}(k)$ constructed from $\sigma$, i.e., in the Zelevinsky notation $\sigma[d]= Z[\sigma \nu^{-(d-1)/2},\cdots, \sigma \nu^{(d-1)/2} ]$. It is known that the
only nonzero derivative of $\sigma[d]$ is $\sigma[d]^i$ for $i=0, r$, and $\sigma[d]^r = \nu^{-1/2}\sigma[d-1]= Z[\sigma\nu^{-(d-1)/2},\cdots, \sigma\nu^{(d-3)/2}]$.

Denote by $\St_d[\rho]$, the generalized  Steinberg representation of $\GL_{dd'}(k)$  
 where $\rho$ is a cuspidal unitary representation of the  general linear group $\GL_{d'}(k)$.  The derivatives of  
$\St_d[\rho]^i$ are known to be nonzero only for $i = jd'$ for some $0\leq j \leq d$ with   $\St_d[\rho]^{jd'} = \nu ^{j/2} \St_{d-j}[\rho]. $ 
A tempered representation of $\GL_m(k)$ is built as a product of 
the generalized  Steinberg representations $\St_d[\rho]$.
\vskip 5pt

For a cuspidal representation $\rho$ of $\GL_m(k)$ (or the corresponding irreducible representation of $W(k)$),
we call the set of representations $\{\rho \nu^i| i \in \Z\}$, the cuspidal line passing through $\rho$. By Leibnitz rule, if a representation has cuspidal support contained in the cuspidal line passing through $\rho$,
all its derivatives have the same property. Therefore, the analysis below which mostly involves dealing with the non-tempered parts of $\pi_M$ and $\pi_N$, to find 
an irreducible composition factor A of $\nu^{1/2} (\pi_M)^{i+1}$,
and an irreducible composition factor $B$ of  $((\pi_N^\vee)^{i})^\vee$,
such that
${\rm Hom}_{\GL_{n-i}(k)}[ A, B] \not = 0,$
we can focus attention on a cuspidal line passing through a fixed cuspidal representation $\rho$, i.e., we can assume
that the non-tempered parts of $M_A$ and $N_A$ when restricted to $W(k)$ are multiples of $\rho$, without any
constraint on the tempered parts  of $M_A$ and $N_A$. Replacing $\rho$ (an irreducible bounded
representation of
$W(k)$) by the trivial representation
of $W(k)$ has no effect on the analysis done below, and this is what we assume 
in the rest of the proof, i.e., we assume in the rest of the proof that
the non-tempered parts of $M_A$ and $N_A$ when restricted to $W(k)$ are multiples of the trivial representation.

\vskip 5pt

Let
\begin{eqnarray*}
  \pi_M & = & [a_1] \times \cdots \times [a_r] \times [b_1] \times \cdots \times [b_s] \times \St[c_1] \times \cdots \times \St[c_t]
  \times \St[d_1] \times \cdots \times \St[d_u], \\
\pi_N & = & [a'_1] \times \cdots \times [a'_{r'}] \times [b'_1] \times \cdots \times[b'_{s'}] \times \St[c'_1] \times \cdots \times \St[c'_{t'}]
\times \St[d'_1] \times \cdots \times \St[d'_{u'}],
\end{eqnarray*} 
where $\St[a]$ denotes the Steinberg representation of $\GL_a(k)$. There could be a further non-Steinberg tempered part in $\pi_M$, and in $\pi_N$,  which we do not write here out for brevity of notation and which will not play any role in the arguments
below.
Note also that $\St[1]=[1]$, the trivial representation of $\GL_1(k)$; in fact, we will use the notation $\St[a]$ in the above expressions for
$\pi_M,\pi_N$  only for $a>1$. 

The notation above is so made that for calculating $(j+1)$-th derivative of $\pi_M$, and $j$-th derivative of $\pi_N$, we differentiate all the terms
$[a_i], [a'_i], \St[d_i], \St[d'_i]$ at least once, and no more for $[a_i], [a'_i]$; whereas none of  the terms involving
$[b_i], [b'_i], \St[c_i], \St[c'_i]$ are to be  differentiated. If the cuspidal support of
an irreducible composition
factor $C$ of $\nu^{1/2}\pi_M^{j+1}$ and $D$ of $(\pi_N^\vee){^{j}}^\vee$ are
to be the same, we must have equality of cuspidal supports of $\tilde{C}$, and $\tilde{D}$ defined as follows:
\begin{eqnarray*}
 \tilde{C} & = & \nu^{1/2}  \{ \nu^{-1/2}[a_1-1] \times \cdots \times \nu^{-1/2}[a_r-1] \times [b_1] \times \cdots \times [b_s]\\
  & \times  &     \St[c_1] \times \cdots \times \St[c_t]
  \times \nu^{e_1/2}\St[d_1-e_1] \times \cdots \times \nu^{e_u/2}\St[d_u-e_u] \} , \\
\tilde{D} & = & \nu^{1/2}[a'_1-1] \times \cdots \times \nu^{1/2}[a'_{r'}-1] \times [b'_1] \times \cdots \times[b'_{s'}] \times \\
  & \times &  \St[c'_1]
  \times \cdots \times \St[c'_{t'}]
  \times \nu^{-f_1/2}\St[d'_1-f_1] \times \cdots \times \nu^{-f_{u'}/2}\St[d'_{u'}-f_{u'}].
  \end{eqnarray*}

If the          cuspidal supports of the representations $\tilde{C}$ and $\tilde{D}$ 
are
to be the same, we draw the following conclusions from Lemma \ref{support}. (Let us remind ourselves that the cuspidal support
of $\St[a]$ and $[a]$ is the same, so when using Lemma \ref{support}, we will replace all $\St[a]$ in the above expressions
for $\tilde{C}$ and $\tilde{D}$ by the corresponding $[a]$. Further, the hypothesis in Lemma \ref{support}  is easily seen to be the same as the hypothesis in the
Theorem \ref{thm-gln} that for any integer
$a \geq 1$, the representation 
 $\C \otimes \Sym^a(\C^2)  \oplus \Sym^a(\C^2) \otimes \C$
  does not appear in $M_A$ or in $N_A$.

\begin{enumerate}

\item There is no differentiation for $\St[d_i]$ and $\St[d'_i]$.

\item For each $[a_i]$ in $\pi_M$ (which get differentiated once), there is exactly one of
  $[a_i-1]$ or $\St[a_i-1]$ amongst $[b'_i]$ or $\St[c'_i]$ for $\pi_N$
  (these are not differentiated in achieving $\pi_N^j$).
  
\item For each $[a'_i]$ in $\pi_N$ (which get differentiated once), there is exactly one of
  $[a'_i-1]$ or $\St[a'_i-1]$ amongst $[b_i]$ or $\St[c_i]$ for $\pi_M$
  (these are not differentiated in achieving $\pi_M^{j+1}$).
\item Each of the terms in  
  $ \tilde{C}$ and $\tilde{D}$ (in expressing these representations as a sum -- in the Grothendieck group of representations -- of products of representations of the form $[a]$ and $\St[b]$) 
  is matched through (1),(2),(3).
\end{enumerate}

It follows that if the          cuspidal supports of $\tilde{C}$ and $\tilde{D}$ are
to be the same, we must have the following structure for $\pi_M$ and $\pi_N$
(up to multiplication by a tempered representation $\pi_M^t$  in $\pi_M$ and
up to multiplication by a tempered representation $\pi_N^t$  in  $\pi_N$; in using the Leibnitz rule to calculate
$\pi_M^{j+1}$ and $\pi_N^j$, one takes the highest possible derivatives for $\pi_M^t$ and $\pi_N^t$, which has the effect that
the terms $\pi_M^t$ and $\pi_N^t$ do not show up in $\pi_M^{j+1}$ and $\pi_N^j$):

\begin{eqnarray*}
  \pi_M & \stackrel{(1)}= & [b'_1+1] \times \cdots \times[b'_{s'}+1] \times [c'_1+1] \times \cdots \times [c'_{t'}+1]   \times [b_1] \times \cdots \times [b_s]  \\
    &  & \,  \times \St[c_1] \times \cdots \times \St[c_t], \\
  \pi_N & \stackrel{(2)} = & [b_1+1] \times \cdots \times [b_{s}+1] \times
     [c_1+1] \times \cdots \times[c_{t}+1] \times [b'_1] \times \cdots \times[b'_{s'}]  \\
     &  & \,  \times  \St[c'_1] \times \cdots \times \St[c'_{t'}].
\end{eqnarray*}

For these representations $\pi_M,\pi_N$, if there is a composition factor $C$ of 
$\nu^{1/2} \pi_M^{j+1}$, and $D$ of ${(\pi_N^\vee)^{j}}^\vee,$
for which $C,D$ have the same cuspidal support as representations of $\GL_{n-j}(k)$,
the representations $\tilde{C}$ and $\tilde{D}$ must be:

\begin{eqnarray*}
  \tilde{C}
  & \stackrel{(3)}= & [b'_1] \times \cdots \times[b'_{s'}] \times [c'_1] \times \cdots \times [c'_{t'}] 
   \times \nu^{1/2} \{ [b_1] \times \cdots \times [b_s] \times \St[c_1] \times \cdots \times \St[c_t]\}, \\
          \tilde{D}          & \stackrel{(4)}= & 
         \nu^{1/2}\{  [b_1] \times \cdots \times [b_{s}] \times
          [c_1] \times \cdots \times[c_{t}] \}  
           \times  [b'_1] \times \cdots \times[b'_{s'}] \times  \St[c'_1] \times \cdots \times \St[c'_{t'}].
\end{eqnarray*}

By Lemma \ref{irr}, such representations $\tilde{C}, \tilde{D}$ are irreducible, and for
${\rm Hom}_{\GL_{n-j}(k)}[ \tilde{C}, \tilde{D}]$ to be nonzero, 
we must have:
\begin{eqnarray*}
c_i & = & 1 {\rm~ for~ all~} i \\
  c'_i & = & 1 {\rm~ for~ all~} i.
\end{eqnarray*}
Thus, as a consequence, we find that if any of the spaces
${\rm Hom}_{\GL_{n-j}(k)}[ C, D]$ is nonzero, then
$M_A,N_A$ are a  relevant pair of A-parameters.

Next, we prove that $\Hom_{\GL_n(k)}[\pi_M,\pi_N] \not = 0$ when the A-parameter  $M_A$ for an irreducible representation $\pi_M$
  of $\GL_{n+1}(k)$ 
  and $N_A$ for an irreducible representation $\pi_N$ of $\GL_n(k)$
    satisfies all conditions in the
    statement of this theorem with $M_A,N_A$, a relevant pair of A-parameters. With the structure of $\pi_M,\pi_N$ given in (1) and (2),
the only way to get 
an irreducible composition
factor $\tilde{C}$ of $\nu^{1/2}\pi_M^{j+1}$ and $\tilde{D}$ of $(\pi_N^\vee){^{j}}^\vee$ to have  same cuspidal supports
is    that none of  the terms $\St(c_i), \St(c'_i)$ should be
    present in the expressions for $\pi_M$ and $\pi_N$ given in (1) and (2), and also in $\tilde{C},\tilde{D}$ given in (3) and (4),
    i.e.,  we must  take full derivative of all tempered factors of $\pi_M$ which do not match-up with
presence of a $[2]$ in $\pi_N$ and  we must take full derivative on all tempered factors of $\pi_N$ which do not match-up with
presence of a $[2]$ in $\pi_M$.  

Thus the strategy outlined in the beginning of the proof of the theorem that there is a unique $j$, and a
unique irreducible composition
factor $C$ of $\nu^{1/2}\pi_M^{j+1}$ and $D$ of $(\pi_N^\vee){^{j}}^\vee$ with the   same cuspidal supports holds;  we elaborate more on this now.

Write:
  \begin{eqnarray*}
  \pi_M & = & [b'_1+1] \times \cdots \times[b'_{s'}+1] 
   \times [b_1] \times \cdots \times [b_s] \times \tau_1, \\
  \pi_N & = & [b_1+1] \times \cdots \times [b_{s}+1] \times
          [b'_1] \times \cdots \times[b'_{s'}] \times \tau_2 ,
\end{eqnarray*}
where $\tau_1,\tau_2$ are
tempered representations, not generalized Steinberg,  of $\GL_{d_1}(k)$ and $\GL_{d_2}(k)$, for certain integers $d_1,d_2$ with
\begin{eqnarray*}
s'+d_1 + \sum_{i=1}^{s} b_i + \sum_{i=1}^{s'} b'_i & = & n+1 \\
 s+d_2 + \sum_{i=1}^{s} b_i + \sum_{i=1}^{s'} b'_i & = & n.
\end{eqnarray*}

  We see that there is exactly one integer $j$ which by the above equations has the property that:
 \begin{eqnarray*}
j+1 & = & s'+ d_1\\
 j & = & s+d_2,
 \end{eqnarray*}
 and for this integer $j$, there is
 only one composition factor, say $A$,   of $\nu^{1/2} \pi_M^{j+1}$ and  one of ${(\pi_N^\vee)^{j}}^\vee$, say $B$, for which 
${\rm Hom}_{\GL_{n-j}(k)}[ A, B]$
is nonzero, and all the other
composition factors  of $\nu^{1/2} \pi_M^{j+1}$ and   of ${(\pi_N^\vee)^{j}}^\vee$ have different cuspidal
supports, and so is the case also for any composition factors  of $\nu^{1/2} \pi_M^{i+1}$ and   of ${(\pi_N^\vee)^{i}}^\vee$
for $i \not = j$. (A subtlety in this argument may be pointed out,  which is that a particular
composition factor may appear with higher (Jordan-H\"older) multiplicity, for example,  the first derivative
of the representation $ 1 \times 1$ of $\GL_2(k)$ is the trivial representation of $\GL_1(k)$ with multiplicity 2, one which appears as a submodule, and the other which appears as a quotient. This subtlety will show up as soon as
$ b_i=b_j'+1$ for some $i,j$. Since we are not trying to prove multiplicity 1 theorem
for $\Hom(\pi_M,\pi_N)$, but rather only the existence of a nonzero homomorphism, this subtlety does not create a problem for us.)

Therefore, $${\rm Ext}^\ell_{\GL_{n-i}(k)}[ C, D] = 0, {\rm~for~} (A,B) \not = (C,D), {\rm ~and ~for ~all~} \ell \geq 0, $$
completing the proof of the `existence' part of the theorem, i.e., when the parameters $M_A,N_A$ are a
relevant pair of A-parameters,
and  are representations of $WD(k)  \times \SL_2(\C)$ as in the statement of the theorem, 
$\Hom_{\GL_{n}(k)}[ \pi_M, \pi_N] \not = 0$. The other part of the theorem was already proved earlier, thus the proof of the theorem
is complete.
\end{proof}

\begin{lemma} \label{support}
Suppose that  the cuspidal supports of
\begin{eqnarray*}
  V & = & [a_1] \times \cdots \times [a_r] 
  \times \nu^{e_1/2}[b_1] \times \cdots \times \nu^{e_s/2}[b_s] , \\
  W& = &  [c_1] \times \cdots \times[c_{t}] \times \nu^{1/2}\{ [d_1] \times \cdots \times [d_{u}] \}
    \times \nu^{-f_1/2}[g_1] \times \cdots \times \nu^{-f_{v}/2}[g_{v}],
  \end{eqnarray*}
are the same (where $e_i > 0$, $f_i> 0$ are integers). Assume that the following two conditions (V), (W) are satisfied:

\[ \tag{V} a_i +1 \not = b_j+e_j-1 {\rm~~ for ~any~ pair ~} (i,j) {\rm ~with~ }  e_j>1, {\rm ~and} \]
 \[ \tag{W} d_i+1  \not = g_j+f_j {\rm ~~ for ~any ~pair ~ } (i,j). \]

    Then,
  \begin{enumerate}
\item $e_i =1$ for all $i$.
  \item $f_j=0$ for all $j$.
  \item The set of irreducible representations appearing in writing $V$ as a product and $W$ as a product is unique
    up to permutation of the representations involved.
    
\end{enumerate}
\end{lemma}
\begin{proof} We prove the lemma by choosing a component $[a_i]$ or $\nu^{e_j/2}[b_j]$ from $V$ with the property
  that $V$ contains $\nu^x$, with $x$ the maximum half integer such that $\nu^x$ is contained in the cuspidal support
  of $V$, and proving that this component ($[a_i]$ or $\nu^{e_j/2}[b_j]$)  lies in $W$ too. Removing this common component from $V$ and $W$ proves the
  lemma by an inductive process.
 (Recall that the cuspidal support of $V$ is the union with multiplicities of the support of individual
  factors in it, and that the cuspidal support of $[a]$ is $\{\nu^{-(a-1)/2}, \cdots, \nu^{(a-1)/2} \}$.

  The proof of the lemma will be carried out in the following two cases.

  \vspace{2mm}
  
  {\bf Case 1:}
  Suppose $V$ contains $\nu^x$, with $x$ the maximum half integer such that $\nu^x$ is contained in the cuspidal support
  of $V$, and that $\nu^x \in [a_i] \subset V$.

  In this case, $\nu^{-x}$ also belongs to the support of $V$,  $\{-x,x \}$ are the extreme points of the support
  of $[a_i]$ as well as $V$  (since $e_i>0$). We would like to prove that $[a_i] = [c_j]$ for some $[c_j] \subset W$. Assume the contrary,
  i.e., $x \not = c_j$ for any $[c_j] \subset W$.  

  Since  $f_j>0$, if $\nu^x \in \nu^{-f_j/2}[g_j] \subset W$, then $\nu^{-y} \in W$ for some $y>x$, a contradiction to the fact that $\nu^{-y}, y>x$
  does not belong to the support of $V$.

  If $\nu^x \in \nu^{1/2} [d_j] \subset W$, then $x=d_j/2$. Since $\nu^{-x} \in V $, $\nu^{-x}$ also belongs to $W$, but $\nu^{-x}$
  does not belong to terms of the form $[c_j], \nu^{1/2} [d_j] \subset W$, so the only option is that $\nu^{-x}$ belongs to
  $\nu^{-f_j/2}[g_j] \subset W$, i.e., $ 2x = f_j+g_j-1$. On the other hand, $x=d_j/2$.

  Thus $W$ contains both $\nu^{1/2} [d_j]$ and   $\nu^{-f_j/2}[g_j]$ with
  $$2x = d_j = f_j+g_j-1,$$
  which is not allowed by the condition (W) in the statement of the lemma. Thus the only option left is that  $\nu^x \in [c_j] \subset W$.

  \vspace{2mm}
         {\bf Case 2:}
         Suppose $V$ contains $\nu^x$, with $x$ the maximum half integer such that $\nu^x$ is contained in the cuspidal support
         of $V$, and that $\nu^x \in \nu^{e_i/2} [b_i] \subset V$ with $e_i$ minimal positive integer, and
         with $\nu^x \not \in [a_i]$ for any $[a_i] \subset V$. 

  In this case, $\nu^{-x}$ does not belong to the support of $V$, hence does not belong to the support of $W$ either. It follows that
  $\nu^x$ which is supposed to belong to $W$ must belong to either $\nu^{1/2}[d_j]$ or $\nu^{-f_j/2}[g_j]$. Clearly,  $\nu^x$ cannot
  belong to  $\nu^{-f_j/2}[g_j]$ since otherwise $\nu^{-x}$ will belong to the support of $W$.

  If   $\nu^x$ belongs to  $\nu^{1/2}[d_j] \subset W$, $x$ being the largest exponent in $W$, $x= d_j/2$ and  $\nu^{-x+1} \in \nu^{1/2}[d_j]$, hence
  $\nu^{-x+1} \in V$. If $e_i>1$, the only option is that  $\nu^{-x+1} \in [2x-1] \subset V$. But we already have, $\nu^x \in \nu^{e_i/2} [b_i] \subset V$,
  so $V$ contains $[2x-1]$ as well as $ \nu^{e_i/2} [b_i]$ with
  $$2x=e_i+b_i-1,$$ 
  once again a contradiction to our condition (V) in the statement of the lemma. Thus $e_i=1$, $\nu^x \in \nu^{1/2}[b_i] \subset V$ and $\nu^x\in \nu^{1/2}[d_j] \subset W$,
  completing the proof of the lemma.
\end{proof}

\begin{example} We present two examples to show that the conditions in  Lemma \ref{support} are necessary.

  \begin{enumerate}
  \item $[a] \times \nu [a]$ and $\nu^{1/2}[a+1] \times \nu^{1/2}[a-1]$ have the same cuspidal support; said other way,
    $\nu^{1/2}[a] \times \nu^{-1/2} [a]$ and $[a+1] \times [a-1]$ have the same cuspidal support.
    
    \item $\nu^{1/2}[n-1] \times \nu^{-(n-1)/2}$ and $[n]$ have the same cuspidal support. \end{enumerate}
    
  \end{example}

We thank A. Minguez for the proof of the following lemma.
\begin{lemma} \label{irr}
The representations:
\begin{eqnarray*}
  V & = & 
  [a_1] \times \cdots \times [a_r]  \times  \nu^{1/2}\{[b_1] \times \cdots \times   [b_s] 
  \times \St[d_1] \times \cdots \times \St[d_u] \} , \\
  W& = & \nu^{1/2}\{ [a'_1] \times \cdots \times [a'_{r'}] \} \times [b'_1] \times \cdots \times[b'_{s'}] \times 
    \St[d'_1] \times \cdots \times \St[d'_{u'}],
  \end{eqnarray*}
are irreducible, and are isomorphic  if and only if the expressions for $V, W$ as products are the same up to a
permutation, in particular, if $V\cong W$, then $d_i,d'_i \leq 1$.
\end{lemma}
\begin{proof} Irreducibility is part of Theorem 3.9 of the paper \cite{BLM} due to Badulescu, Lapid and Minguez which gives a
  sufficient condition for a product $Z({\frak m}) \times L( {\frak m} ')$ to be irreducible: when  no segment of ${\frak m}$ is juxtaposed (in the obvious sense)
  to a segment in ${\frak m} '$.

  To prove that $V\cong W$  if and only if the expressions for $V, W$ as products are the same up to a
  permutation, define  $$R= \sum_{n \geq 0} R_n,$$
  where $R_0=\Z$, and $R_d$ is the Grothendieck
  group of finite length representations of $\GL_d(k)$. Then $R$  is a ring under the product $R_n \times R_m \rightarrow R_{n+m}$
  given by parabolic induction.

  The ring $R$ is known to be a polynomial ring $\Z[S(C)]$ on indeterminates  $S(C)$, where $S(C)$ is the set of all
  segments $\{\rho, \rho \nu,\cdots, \rho \nu^{n-1} \}$ where $\rho$ is a cuspidal representation of some $\GL_d(k)$.
  There is a natural map from the polynomial ring $\Z[S(C)]$ to $R$ 
taking   $S(C)$ to the unique irreducible submodule of the full parabolic induction:
$\rho \times  \rho \nu \times \cdots \times  \rho \nu^{n-1} $,  which is an isomorphism of rings.
In particular, 
  $R$ is a UFD. Further, it is known (by Tadic \cite[Section 3]{Tad}) that the representations $[a]$ and $\St[a]$ are prime elements
  of the ring  $R$, proving the second part of the Lemma.    \end{proof}

\begin{remark}
  An explicit example of a representation $\pi_M$ of $\GL_5(k)$ and $\pi_N$
  of $\GL_4(k)$ was provided by K. Y. Chan where the proof above cannot be carried out to prove that
  $\Hom_{\GL_4(k)}(\pi_M,\pi_N) \not = 0$. For this, take the representation $\pi_M$ of $\GL_5(k)$ and $\pi_N$
  of $\GL_4(k)$ which in the above notation are,
\begin{eqnarray*}
  \pi_M & = & [3] \times [1] \times [1], \\
  \pi_N & = & \St [2] \times [2]
  \end{eqnarray*}
  In this case, it is easy to see $\nu^{1/2} (\pi_M)^{i+1}$,
  and   $((\pi_N^\vee)^{i})^\vee$ have the same cuspidal support for $i = 1, 2$, and the strategy
  of the proof of the theorem  to conclude $\Hom_{\GL_4(k)}(\pi_M,\pi_N) \not = 0$ fails. In fact, we do not know at
  this point
  whether   $\Hom_{\GL_4(k)}(\pi_M,\pi_N)  = 0$ or is nonzero although the representations have a pair of relevant
  A-parameters.
  \end{remark}

The proof of the following proposition follows from the proof of Theorem \ref{thm-gln}.

\begin{prop} Let $\pi_M$ be an irreducible admissible representation of $\GL_{n+1}(k)$ with  A-parameter $M_A$  of dimension $(n+1)$, and  $\pi_N$ be an irreducible admissible representation of $\GL_{n}(k)$ with  A-parameter $N_A$ of  dimension $n$. Assume that the representations
  $M_A, N_A$ of $WD(k) \times \SL_2(\mathbb C)$ are in fact representations of  $W(k) \times \SL_2(\mathbb C)$.
 Then if $${\rm Ext}^i_{\GL_n(k)}(\pi_M,\pi_N)\not = 0, \hspace{1 cm} {\rm for ~some~} i \geq 0,$$
  $M_A,N_A$ are a relevant pair  of A-parameters.
\end{prop}

\begin{remark} It appears to be an interesting problem to understand when ${\rm Ext}^i_{\GL_n(k)}(\pi_M,\pi_N)$
  is nonzero   for some $ i \geq 0,$ among irreducible representations $\pi_M,\pi_N$ of $\GL_n(k)$ with an A-parameter $M_A,N_A$
  (where there is the necessary condition that $M_A, N_A$ restricted to $W(k) \times \Delta \SL_2(\C)
  \subset WD(k) \times \SL_2(\C) $ are isomorphic),
  and also for the restriction problem
  when $\pi_M$ is an irreducible representation of $\GL_{n+1}(k)$ and  $\pi_N$ of $\GL_n(k)$, both  with an A-parameter, where the above proposition
  provides the answer in the first non-trivial case. See a recent work of K. Y. Chan in \cite{Ch} for some results in both the cases.
\end{remark}

\section{Local Conjecture for Classical Groups}  \label{S:classical}
 
In this section,  we consider the restriction problem for classical groups. 
Hence we shall work in the context of \cite{GGP}. More precisely, we fix local fields $k_0 \subset k$ with $[k:k_0] \leq 2$.  Then, with $\epsilon = \pm$, we consider a pair of non-degenerate $\epsilon$-Hermitian spaces  $W \subset V$ over $k$ such that
\vskip 5pt

\begin{itemize}
\item $\epsilon  \cdot (-1)^{\dim W^{\perp}} =  -1$;
\item $W^{\perp}$ is a split space.
\end{itemize}
Setting $G(V)$ to be the identity component of the automorphism group of the space $V$, we assume further that $G(V)$ is quasi-split.
We thus have a diagonal embedding
\[  G(W)^{\Delta} \hookrightarrow  G(V) \times G(W).  \]
Indeed, as was explained in \cite{GGP}, one has a subgroup 
\[  G(W)^{\Delta}  \subset H = G(W)^{\Delta} \cdot N \subset G(V) \times G(W) \]
where $N$ is a certain unipotent subgroup of $G(V)$ normalized by $G(W)^{\Delta}$.  Moreover, a certain representation $\nu$ of $H(k)$ was defined in \cite[\S 14]{GGP}: it is a 1-dimensional character if $\dim W^{\perp}$ is odd and is essentially a Weil representation when $\dim W^{\perp}$ is even.  When $\dim W^{\perp}  =1$, for example, $\nu$ is simply the trivial representation. 

\vskip 5pt

Given an irreducible representation $\pi_V \otimes \pi_W$ of $G(V) \times G(W)$, we defined in \cite[\S 14]{GGP} a multiplicity $d(\pi_V, \pi_W)$, given by
 \[  d(\pi_V, \pi_W)  = \dim \Hom_{G(W)} (\pi_V \otimes \pi_W, \nu). \]
 It is known that 
 \[  d(\pi_V, \pi_W) \leq 1. \] 
We have also introduced in \cite{GGP}  a notion of relevant pure inner forms $G(V') \times G(W')$ of $G(V) \times G(W)$, and one can likewise consider the multiplicity $d(\pi_{V'} , \pi_{W'})$ for representations of the pure inner form.  Our goal is to determine the multiplicity $d(\pi_V , \pi_W)$ when $\pi_V \otimes \pi_W$ belongs to local L-packets of Arthur type. Before formulating the conjecture, we recall some basic facts about A-packets of classical groups. 
\vskip 10pt

Suppose that $M_A$ is a local A-parameter for $G(V)$, with associated L-parameter $M$. By the local Langlands correspondence, $M$ gives rise to a Vogan L-packet $\Pi_M$ consisting of irreducible representations of pure
inner forms $G(V')$ of $G(V)$. Moreover, fixing generic data as in \cite{GGP}, there is a bijection
\[   \Pi_M  \longleftrightarrow  {\rm Irr}(A_M)  \]
where $A_M$ is the component group of the L-parameter $M$ and  is an elementary abelian 2-group with a canonical basis.
Hence
\[   \Pi_M  = \{  \pi_{\eta} \in {\rm Irr}(G(V')):  \eta \in {\rm Irr}(A_M)  \}. \]
On the other hand, 
the A-parameter $M_A$ gives rise to a local A-packet $\Pi_{M_A}$, which is a priori a multi-set of unitary representations of the pure inner forms of $G(V)$ in correspondence with the irreducible characters of the component group $A_{M_A}$ (which is also an elementary abelian 2-group with a canonical basis).
In other words, one has
\[  \Pi_{M_A}  = \{ \pi_{\eta} : \eta \in {\rm Irr}(A_{M_A}) \}\]
where now $\pi_{\eta}$ is a (possibly zero and possibly reducible) finite length  unitary representation of some pure inner form $G(V')$ of $G(V)$.  
The work of M{\oe}glin \cite{M1, M2, M3} and M{\oe}glin-Renard \cite{MR1} has shown that (except possibly over $\R$) the local A-packets $\Pi_{M_A}$ are sets rather than multi-sets. This means that the representations $\pi_{\eta}$ are multiplicity-free and $\pi_{\eta}$ and $\pi_{\eta'}$ have no common constituents if $\eta \ne \eta'$. One expects similar results over $\R$ for which special cases have been shown by M{\oe}glin-Renard \cite{MR2, MR3}.

\vskip 5pt

Regarding $\Pi_{M_A}$ as a set of irreducible representations (by considering the irreducible summands of $\pi_{\eta}$), one thus has a map
\[  \Pi_{M_A}  \longrightarrow {\rm Irr}(A_{M_A}), \]
which is neither surjective nor injective in general.
There is a natural surjective map 
\[  A_{M_A}  \longrightarrow A_M,  \]
giving rise to a natural injective map
\[   {\rm Irr}(A_M)  \hookrightarrow {\rm Irr}(A_{M_A}), \]
and a commutative diagram
\[  \begin{CD}
\Pi_M  @>\text{inj.}>>  \Pi_{M_A}  \\
@V\text{bij.}VV  @VVV  \\
 {\rm Irr}(A_M)  @>>\text{inj}>   {\rm Irr}(A_{M_A}).  \end{CD} \]
 In other words, $\Pi_M$ is a subset of $\Pi_{M_A}$ (this is actually part of the defining property of an A-packet) and the labelling of its elements by characters of the component groups $A_M$ or $A_{M_A}$ is consistent, cf. Conjecture 6.1 (iv) of \cite{Art1}. 
\vskip 5pt

 After this short preparation, we can now formulate the conjecture.
\vskip 5pt

\begin{conj} \label{conj-class}
Let $M_A$ and $N_A$ be local A-parameters for $G(V)$ and $G(W)$ respectively, with associated A-packets $\Pi_{M_A}$ and $\Pi_{N_A}$.  Let $M$ and $N$ be the L-parameters associated to $M_A$ and $N_A$ with corresponding L-packets $\Pi_M$ and $\Pi_N$.

\vskip 5pt

\begin{enumerate}
\item[(a)]  If $(M_A, N_A)$ is not relevant, then 
\[  d(M,N) := \sum_{(\pi, \pi') \in (\Pi_{M} \times \Pi_{N})^{\rm rel}} d(\pi, \pi')  = 0. \]
Here, $(\Pi_{M} \times \Pi_{N})^{\rm rel}$ denotes the subset of representations of relevant pure inner forms of $G(V) \times G(W)$.
\vskip 5pt

\item[(b)]  If $(M_A, N_A)$ is relevant, then there is a unique relevant pair of representations $(\pi_M, \pi_N) \in \Pi_M \times \Pi_N$ such that the multiplicity $d(\pi_M, \pi_N)  =1$. 
In other words,
\[  d(M,N) =  \sum_{(\pi, \pi') \in (\Pi_{M} \times \Pi_{N})^{\rm rel}} d(\pi, \pi')  = 1. \]
  
 \vskip 5pt
 
 \item[(c)] In the setting of (b), the distinguished representation $\pi_M \boxtimes \pi_N$ corresponds to the distinguished character $\chi$ of the component group $A_M \times A_N$  given by the same recipe as in \cite{GGP}.
\end{enumerate} 
\end{conj}

 \vskip 5pt
 The recipe for $\chi$ is sufficiently intricate that we have no desire to repeat it here; we refer the reader to \cite{GGP} for its precise definition. 
 Let us make a few  remarks concerning this conjecture:
 \vskip 5pt

 \begin{remark} \label{conj-class-remark}
 
 \begin{enumerate}
 \item  Modulo the notion of ``relevant A-parameters," the above conjecture is a direct generalization of our original conjectures for tempered L-packets  in \cite{GGP} to the setting of L-packets of Arthur type. 
 \vskip 5pt

\item Unlike the case of $G=\GL_n$, we cannot treat the special case of this conjecture for those A-parameters $(M_A,N_A)$
  whose associated L-parameters  $(M,N)$ are unramified (in particular trivial on the Deligne $\SL_2(\C)$ in $WD(k)$). The work of Hendrickson \cite{Hen}
offers a possible strategy for establishing this unramified case, using a combination of Mackey theory and theta correspondence. 

 \end{enumerate}
 \end{remark}

\section{A Conjecture for A-packets} \label{S:A}

In the previous section,  Conjecture \ref{conj-class}  addressed the restriction problem for the local L-packet $\Pi_M \times \Pi_N$ contained in a local  A-packet $\Pi_{M_A} \times \Pi_{N_A}$. In this section, we attempt to address the restriction problem for all members of A-packets, but our conjecture in this section is not as precise.
  Ideally, one would like to know the multiplicities for all relevant pairs of representations in the A-packet.  Namely, one would like to have a prediction for the sum
 \[  d(M_A, N_A):= \sum_{ (\pi, \pi') \in (\Pi_{M_A} \times \Pi_{N_A})^{\rm rel}} d(\pi, \pi') \]
 and a prediction of which relevant pairs $(\pi, \pi')$ give nonzero contribution. 
 Examples of restriction problems for low rank groups (some of which we will discuss later) indicate that the above sum can be $>1$; see Remark \ref{geq2} below.

\vskip 5pt

  Recall that unlike the case of $\GL_n$ where $\Pi_{M_A}  = \Pi_M$ (i.e. an A-packet is equal to its associated L-packet) are singletons and where A-packets corresponding to distinct A-parameters are disjoint, the A-packets of classical groups are more complicated and more interesting:
  \vskip 5pt
  
  \begin{itemize}
  \item Firstly, it is rarely the case that the containment $\Pi_M \subset \Pi_{M_A}$ is an equality. 
  \item Secondly, the sets $\Pi_M$ and $\Pi_{M_A}$ are typically not singletons. 
  \item Thirdly, A-packets are not necessarily disjoint. 
  \end{itemize}
These complications make the problem of predicting the multiplicities in an A-packet rather tricky. 
At this moment, we can only offer the following (optimistic) conjecture, which may be treated more as a question. 
\vskip 10pt

\begin{conj} \label{relevant}
  \begin{enumerate}
  \item If $(M_A, N_A)$ is a pair of A-parameters (for a pair of classical groups $G(V) \times G(W)$) such that the Deligne $\SL_2(\C)$ 
  acts trivially, then 
  \[ d(M_A, N_A) \ne 0 \Longleftrightarrow \text{$(M_A, N_A)$ is relevant.}  \]
  \vskip 5pt

\item  For a representation $\pi_1 \times \pi_2$, an irreducible representation of
  $G(V) \times G(W)$, and belonging to an Arthur packet,
  \[ d(\pi_1,\pi_2) \neq 0 \Longrightarrow \pi_1 \times \pi_2 \in  \Pi_{M_A} \times \Pi_{N_A} \, \, \text{for some relevant pair $(M_A, N_A)$}. \]

  \end{enumerate}
\end{conj}

\vskip 5pt

We make a few remarks about this conjecture.
\vskip 5pt

\begin{remark} We outline a heuristic global justification --- without in the least pretending to be complete ---
 of  part (2) of Conjecture \ref{relevant}. Assume that $d(\pi_1,\pi_2) \neq 0$. Since
  $\pi_1$ and $\pi_2$ are of Arthur type, we have automorphic representations  $\Pi_1 = \otimes \Pi_{1,v}$, and
  $\Pi_2 = \otimes \Pi_{2,v}$ with $\pi_1=\Pi_{1,v_0}$ and $\pi_2=\Pi_{2,v_0}$ for some place $v_0$ of a global field $F$. Now we can assume by a form of the Burger-Sarnak principle that $\Pi_1$ and $\Pi_2$ are so chosen that
  $\Pi_1 \otimes \Pi_2$  has a nonzero global period integral. By Arthur,   $\Pi_1 \otimes \Pi_2$ has a  (uniquely determined) global A-parameter,
  such that all the local components of $\Pi_1 \otimes \Pi_2$ (in particular $\pi_1 \times \pi_2$)
  have the corresponding local A-parameter. Since the period integral  on $\Pi_1 \otimes \Pi_2$ is nonzero,
   it follows from the global conjecture \ref{C:global} (see later) that the global A-parameter of $(\Pi_1, \Pi_2)$ is relevant. 
  Hence $(\pi_1,\pi_2)$ has a relevant pair of local A-parameters.
  \end{remark}
\vskip 5pt

\begin{remark} \label{Rem:multiplicities}
  It appears to us  that higher multiplicities in a given A-packet
  is the result of  other A-packets (with relevant A-parameters) intersecting this A-packet. To be more precise,
  we feel that for $(M_A, N_A)$, a pair of A-parameter for $G(V) \times G(W)$, 
  \[ d(M_A, N_A) \leq |X|,\]
where $X$ is a set of maximal cardinality, consisting of pairs  $(\pi_1' \times \pi_2', M_A' \times N_A')$ 
such that
\begin{enumerate}
\item $ M_A' \times  N_A'$ is any relevant pair of A-parameter for $G(V) \times G(W)$,
\item $\pi_1' \times \pi_2'$ is a representation belonging to both  the A-packets $\Pi_{M_A} \times \Pi_{N_A}$ and
 $\Pi_{M'_A} \times \Pi_{N'_A}$
\item  the projections from $X$ to both the first and second factor are injective. 
  \end{enumerate}
\end{remark}

 \vskip 5pt

We now offer some evidence for Conjecture \ref{relevant}, but before doing so, we need to recall some
results of M{\oe}glin. Given a discrete L-parameter $M$ of a classical group, M{\oe}glin determines precisely 
which elements of its L-packet $\Pi_M$ can be a supercuspidal representation. To formulate the result, we introduce the following two notions:
\vskip 5pt

\begin{itemize}
\item Say that a discrete L-parameter $M$ is {\em without gaps} (or holes) if the following holds: for any
 selfdual irreducible representation $\rho$ of $W(k)$ and $a \geq 1$,
 \[  \rho \boxtimes [a+2] \subset M \Longrightarrow \rho \boxtimes [a] \subset M. \]
Equivalently,  let 
\[ \rho \boxtimes [1] + \rho \boxtimes [3]  +\cdots \quad \text{or} \quad 
\rho \boxtimes [2] + \rho \boxtimes [4] +\cdots \]
be a maximal chain appearing as a direct summand in  $M$. Then $M$ is without gaps if 
these chains (as $\rho$ varies) span $M$. 
\vskip 5pt

\item Let $M$ be  a discrete L-parameter without gaps and let $A_M$ be its component group, so that $\Pi_M$ is in bijection with ${\rm Irr}(A_M)$. Recall that $A_M$ is a vector space over $\Z/2\Z$ with a canonical basis indexed by the irreducible summands of $M$. 
Say that  $\alpha \in {\rm Irr}(A_M)$  is {\em alternating} if the following holds: for any $\rho \boxtimes [a] + \rho \boxtimes [a+2] \subset M$, 
$$\alpha(\rho \boxtimes [a]) = -\alpha (\rho \boxtimes [a+2]),$$
with the convention that if $a=0$, then $\rho \boxtimes [a]=0$, and $\alpha(\rho \boxtimes[ 0])=1$, so that we must have
$\alpha(\rho \boxtimes[2])=-1$.
\end{itemize}

\vskip 10pt

Here then is M{\oe}glin's first result \cite{M2} that we shall use:
\vskip 5pt

\begin{thm} \label{T:moeglin1}
Let $M$ be a discrete L-parameter for a classical group over a $p$-adic field and let $\alpha \in {\rm Irr}(A_M)$ be a character of its component group. Then the  corresponding representation $\pi(M, \alpha) \in \Pi_M$  is supercuspidal if and only if $M$ is without gaps and $\alpha$ is alternating. 
\end{thm}
\vskip 10pt

The second result of M{\oe}glin \cite{M1, M2} we shall need concerns the possible overlaps of different local A-packets.  
If $M_A$ is the local A-parameter of a (classical) group $G$ over a non-archimedean local field $k$, let $M_A^{\Delta}$ be the tempered L-parameter defined by: 
\[ \begin{CD}
 M_A^{\Delta}:   W(k) \times \SL_2(\C) @>{\rm Id} \times \Delta>> W(k) \times \SL_2(\C) \times \SL_2(\C)  @> M_A >>  {}^L G \end{CD}  \]
where $\Delta$ is the diagonal embedding of $\SL_2(\C)$. 
\vskip 5pt

\begin{thm} \label{T:moeglin2}
Let  $M_A$ and $M'_A$ be two local A-parameters for a classical group  over a $p$-adic field with associated A-packets $\Pi_{M_A}$ and $\Pi_{M'_A}$. Then
\[ \Pi_{M_A} \cap \Pi_{M'_A} \ne \emptyset \Longrightarrow  
 M_A^{\Delta} \cong {M_A'}^{\Delta}. \]
  Further, if $\Pi_{M_A^{\Delta}}^{sc}$ denotes the set of supercuspidal representations in $\Pi_{M_A^{\Delta}}$, then
 \[ \Pi_{M_A^{\Delta}}^{sc} \subset \Pi_{M_A}. \]
\end{thm}
The first part of this theorem of M{\oe}glin is eventually related to the
observation that two A-parameters $M_A$ and $M'_A$  of $\GL_n(k)$ have the same cuspidal support if and only if
 $M_A^{\Delta}$ and ${M'_A}^{\Delta}$ are equivalent. 

%Within the Arthur packet  A-packet $A(\psi_2)$ for $\SO_{m-1}(k)$, one knows by the work of
%M{\oe}glin, the tempered representations which have L-parameter $\Delta \circ \psi_2$;
%in particular, if $\Delta \circ \psi_2$ is a cuspidal tempered parameter, then its cuspidal
%members are in the A-packet  $A(\psi_2)$. Thus if our
%Conjecture \ref{conj-class} is to be true, the tempered components of the restriction of a tempered representation $\pi$
%of a group $G$ which belongs to an  A-packet, restricted to one of the subgroups $H$,
%must have a very special structure! Proposition \ref{wald} below due to Waldspurger proves this.
%Although it is a consequence of the proof of the GGP conjecture (for tempered representations)
%due to
%Waldspurger, the conclusion, which involves calculating certain epsilon factors of representations of $W'_k$, is neither obvious nor  trivial.
\vskip 5pt

At this point, let us take note of the following consequence of tempered GGP which was pointed out to us by Waldspurger:
\vskip 5pt

\begin{prop} \label{wald} 
Let $M_0 \times N_0$ be a tempered L-parameter for $\SO_{2n+1} \times \SO_{2m}$ over a non-archimedean local field $k$ with associated L-packets $\Pi_{M_0} \times \Pi_{N_0}$. 
Let $(\pi, \sigma)$ be the unique member of $\Pi_{M_0} \times \Pi_{N_0}$ such that $d(\pi, \sigma) \ne 0$, so that $(\pi, \sigma)$ corresponds to the distinguished character $ \chi := \chi_{M_0, N_0}$ of the component group by \cite{GGP}. 
\vskip 5pt

\noindent (i) Suppose that $\rho$ is an irreducible selfdual parameter for $\GL_d(k)$ such that for some $a \geq 3$,
both the representations $\rho \otimes [a]$ and $\rho \otimes [a-2]$ of $WD(k)=W(k) \times \SL_2(\C)$
appear with multiplicity 1 in  $M_0$.
In this case, these two summands give rise to two basis elements of the component group associated to $M_0$. Let  
   \[   \chi(\rho,a), \chi(\rho,a-2) \in  \{\pm 1\} \]
be the value of the distinguished character $\chi$ on these two elements respectively.   Then 
  $$\chi(\rho,a) = - \chi(\rho,a-2)$$
 if and only if the parameter $N_0$ contains $\rho \otimes [a-1]$ with odd multiplicity.
\vskip 5pt

\noindent (ii)  Further, if  $\rho \otimes [2]$  appears with multiplicity 1 in  $M_0$,
  then $$\chi(\rho,2) = - 1$$
  if and only if the parameter $N_0$ contains $\rho$ with odd multiplicity.

  Furthermore, there are similar assertions when
  both the representations $\rho \otimes [a]$ and $\rho \otimes [a-2]$ of $WD(k)=W(k) \times \SL_2(\C)$
appear with multiplicity 1 in  $N_0$.
\end{prop}

\begin{proof}
  Let us begin by recalling Conjecture 20.1, page 76 of  \cite{GGP}, giving a  recipe 
  for  $\chi(\rho,a)$:
  $$ \chi(\rho,a) = \epsilon(\rho \otimes [a] \otimes  N_0)\cdot
  (\det \rho)(-1)^{(a\dim N_0)/2}
  \cdot (\det N_0)(-1)^{(a\dim \rho)/2} .$$
  Define more generally for any selfdual representation  $\mu$ of $W(k) \times \SL_2(\C)$ with even
  parity (i.e. same as  that of $N_0$), 
  $$\chi_{\mu}(\rho, a)= \epsilon(\rho \otimes [a] \otimes  \mu) \cdot (\det \rho)(-1)^{(a\dim N_0)/2} \cdot (\det \mu)(-1)^{(a\dim \rho)/2} ,$$
  so that if $N_0 = \oplus_i \mu_i$, a sum of irreducible
selfdual representations  $\mu_i$ of $W(k) \times \SL_2(\C)$ with even
  parity (i.e. same as  that of $N_0$), 
  then
\[  
  \chi(\rho, a)= \prod_i \chi_{\mu_i}(\rho, a). \]
  Thus, we need to compare $\chi_{\mu_i}(\rho, a)$ and $\chi_{\mu_i}(\rho, a-2)$ for any $\mu_i$ as before. 
   \vskip 5pt
     
  Recall  from \cite{Tate} that for an irreducible representation  $\lambda \otimes [n]$ of $W(k) \times \SL_2(\C)$, one has
  $$  \epsilon( \lambda \otimes [n]) = \epsilon(\lambda)^n \cdot \det (-F, \lambda^I)^{n-1}, $$
  where $\lambda^I$ denotes the subspace of $\lambda$ fixed by the inertia group $I$  
  and $F$ denotes the Frobenius element of $W(k)/I$. In particular, if $\lambda$ is an irreducible selfdual representation of $W(k)$, then,
  \[  \epsilon( \lambda \otimes [n]) = \begin{cases} \epsilon(\lambda)^n   \;\;\;\;\;\; {\rm if } \lambda \not =1; \\
     (-1)^{n-1}  \;\;  {\rm if } \lambda =1 . \end{cases} \]

  Using the Clebsch-Gordon theorem, \[[a] \otimes [b] = [a+b-1] \oplus [a+b-3] \oplus \cdots \oplus  [|a-b|-1],\]
  it follows  that if $\rho, \tau$ are irreducible selfdual representation of $W(k)$, then
  \[  \epsilon( \rho \otimes [a] \otimes \tau \otimes [b]) = \begin{cases} \epsilon(\rho \otimes \tau )^{ab}   \;\;\;\;\;\;\;\;\;\;\;\;\;\;\;\;\;\;
    \text{ if $\rho \not \cong \tau$,} \\
  \epsilon(\rho \otimes \tau )^{ab}   (-1)^{n(a,b)} \;\;\;\; \text{ if $\rho \cong \tau$,} 
  \end{cases} \]
  where $n(a,b) = {\rm min}\{a,b\} \cdot [ {\rm max}\{a, b\} -1]$. In particular, observe that for $a,b$ positive integers of different parity,
  \[  \epsilon( [a] \otimes  [b]) \cdot \epsilon( [a-2] \otimes  [b])
  = \begin{cases} \;\; \, 1  \;\;\;\;\;\;
    \text{ if $b \not = a-1$,} \\
  -1  \;\;\;\; \;\; \text{ if $b = a-1$.} 
  \end{cases} \]

  More generally,  it is easily seen that for any irreducible representation  $\mu$ of $W(k) \times \SL_2(\C)$ as above and $a \geq 3$,
 \[ 
  \chi_{\mu}(\rho, a) = \begin{cases}  
  \;\;\, \chi_{\mu}(\rho, a-2), \text{ if $\mu \ne \rho \otimes [a-1]$;} \\
  -   \chi_{\mu}(\rho, a-2), \text{ if $\mu =  \rho \otimes [a-1]$.}
  \end{cases} 
  \]
   It follows that $\chi(\rho, a) = - \chi(\rho, a-2)$ if and only if $\rho \otimes [a-1]$ occurs in $N_0$ with odd multiplicity, thus completing the proof of (i).    The proof of (ii), where  $a=2$,  follows along the same lines; we omit the details.
  \end{proof}

\vskip 5pt

After the above preparation, we can prove the following theorem.
Part (b) of the theorem is a contribution to  
Conjecture \ref{relevant}(1) and (2).

\begin{thm} \label{wald2}
(a)  Let $M_0 \times N_0$ be a discrete L-parameter for $\SO_{2n+1} \times \SO_{2n}$ without gaps
   with associated L-packet $\Pi_{M_0} \times \Pi_{N_0}$. 
  Let $(\pi, \sigma)$ be the unique member of $\Pi_{M_0} \times \Pi_{N_0}$ such that $d(\pi, \sigma) \ne 0$.
 Let  $D(M_0) \times D(N_0)$ be the A-parameter obtained from
  the L-parameter  $M_0 \times N_0$ by interchanging the Deligne $\SL_2(\C)$ by Arthur  $\SL_2(\C)$.   \vskip 5pt
 
 Then  $(\pi, \sigma)$ is  supercuspidal  if and only if $(D(M_0),  D(N_0))$ is relevant.
\vskip 5pt

  (b)  Let $M_A \times N_A$ be a discrete A-parameter for $\SO_{2n+1} \times \SO_{2n}$
  over a non-archimedean local field $k$, on which the Deligne $\SL_2(\C)$ acts trivially,
with associated A-packet $\Pi_{M_A} \times \Pi_{N_A}$. Assume that the A-parameter  $M_A \times N_A$ has no gaps (relative to the Arthur $\SL_2(\C)$).
    Then there exists supercuspidal  $(\pi, \sigma) \in \Pi_{M_A} \times \Pi_{N_A}$ such that  $d(\pi,\sigma) \ne 0$
  if and only if $M_A\times N_A$ is relevant.
   \end{thm} 
\begin{proof}
 (a)  This is a consequence of  the tempered GGP  \cite{W}, Theorem \ref{T:moeglin1}, Theorem \ref{T:moeglin2}, and Proposition \ref{wald}.
\vskip 5pt

\noindent (b)   This follows from part $(a)$ on noting that the Aubert-Zelevinsky involution
\[\pi \rightarrow D(\pi) \]
on the category  of admissible representations of a group $G(k)$ takes the tempered L-packet
  $\Pi_{M_0} \times \Pi_{N_0}$ to   $D(\Pi_{M_0}) \times D(\Pi_{N_0})$ which is the A-packet associated to the A-parameter
  $D(M_0) \times D(N_0)$. Now $\pi = D(\pi)$ (up to a sign) for
  supercuspidal representations $\pi$. Hence, 
  if $(\pi, \sigma)$ is a supercuspidal member of $\Pi_{M_0} \times \Pi_{N_0}$ such that $d(\pi, \sigma) \ne 0$,
  then $(\pi, \sigma)$ is also a member of the A-packet  $D(\Pi_{M_0}) \times D(\Pi_{N_0})$ with
  $d(\pi, \sigma) \ne 0$. \end{proof}

\begin{remark} \label{geq2}
  In Theorem \ref{wald2}(b), the representation 
   $\pi \times \sigma$ we produced  in the A-packet  $\Pi_{M_A} \times \Pi_{N_A}$ with $d(\pi, \sigma) \ne 0$ is supercuspidal 
  and hence lies outside of the associated L-packet $\Pi_M \times \Pi_N$  (since the L-packet underlying an A-parameter with non-trivial restriction
  to the Arthur $\SL_2(\C)$ consists only of non-tempered representations). Thus by Conjecture \ref{conj-class}, the sum of multiplicities
  in such A-packets is $\geq 2$. This remark could be considered to be a contribution to
  Remark \ref{Rem:multiplicities}.
\end{remark}
\vskip 10pt

In the rest of this section, we shall construct examples  of {\em non-relevant} A-parameters $(M_A, N_A)$ 
for which $d(M_A,N_A) > 0$, a subtlety which Conjecture
\ref{relevant}(2) takes into account.
 \vskip 10pt

For any subset $J \subset  \{1,2,...,n\}$ (for $n$ fixed),  consider the representation of $\SL_2(\C) \times \SL_2(\C)$ given by
\[ M_{A,J} =  \bigoplus_{i \notin J} [2i] \otimes [1]  \oplus \bigoplus_{j \in J} [1] \otimes [2j]. \]
Then $M_{A,J}$ is a multiplicity-free sum of symplectic representation of $WD(k) \times \SL_2(\C)$ with $W(k)$ acting trivially.   Hence $M_{A,J}$ is a discrete A-parameter
for an odd special orthogonal group of the appropriate size.  
\vskip 5pt

Now observe that $M_{A,J}^{\Delta}$ is independent of $J$ and is a discrete L-parameter without gaps:
\[ M_{A,J}^{\Delta}   = M_{A, \emptyset}  = \bigoplus_{j=1}^n   1_{W(k)} \otimes [2j]. \]
By Theorem \ref{T:moeglin1}, the L-packet of $M_{A,J}^{\Delta}$ contains a unique supercuspidal representation $\pi_{sc}$
for an odd special orthogonal group of the appropriate size, attached to the unique alternating character $\alpha_{sc}$ of the component group:
\[ \alpha_{sc}(2j) = (-1)^j. \]
By Theorem \ref{T:moeglin2}, we see that
\[ \pi_{sc} \in \Pi_{M_{A,J}} \quad \text{ for any $J$.} \]
 \vskip 10pt

We now consider the restriction problem for the pair $(M_{A,J},  N)$ where $M_{A,J}$ is as above and $N$ is the discrete L-parameter
\[ N :=  [1] \oplus [3] \oplus \cdots \oplus [2n-1] \] 
 on which $W(k)$ acts trivially.  Now consider the restriction problem for the pair $(M_{A,\emptyset}, N)$ of discrete L-parameters. 
 One may apply Proposition \ref{wald} (together with Theorem \ref{T:moeglin1}) to deduce that
 \[  d(\pi_{sc}, N)  =1, \]
 where we recall that $\pi_{sc}$ is the unique supercuspidal representation  with L-parameter $M_{A,\emptyset}$. Since $\pi_{sc} \in \Pi_{M_{A,J}}$ for any $J$, we thus conclude that
 \[  d(M_{A,J}, N) \ne 0 \quad \text{ for any $J$.}   \]
 \vskip 10pt

 It is easy to see that 
 \[ \text{ $(M_{A,J}, N)$ is relevant} \Longleftrightarrow   \text{$J = \emptyset$ or $J = \{1 \}$. } \]
  Hence, if $n > 1$, we have produced non-relevant pairs $(M_{A,J}, N)$ such that $d(M_{A,J}, N) > 0$. In other words, ``relevance" is not a necessary condition for branching in A-packets. 
   \vskip 10pt
   
   The simplest example one can take is $n=2$. In this case, we have the A-parameters
   \[ M_A = [1] \boxtimes [4] + [2] \boxtimes [1]  \quad \text{or}  \quad   [1] \boxtimes [4] + [1] \boxtimes [2], \]
   and
  \[  N = [3] \boxtimes [1] + [1] \boxtimes [1]. \]
   The pair $(M_A,N)$ is not relevant, yet  $d(M_A, N) > 0$. (We note that the A-packets of $\SO_7$ corresponding to the two choices of $M_A$ above can be constructed    by theta lifting from an appropriate L-packet and an A-packet  on ${\rm Mp}_2$ respectively and the unique supercuspidal representation  in them is simply the theta lift (with respect to $\psi$) of the odd Weil representation $\omega_{\psi}$). 
 \vskip 10pt
 
\section{A Special Case of the Conjecture} \label{MW}
The results of M{\oe}glin and Waldspurger that were used in the previous section to provide a counterexample to the necessity of relevance for the restriction problem for A-packets can be used in other circumstances to verify special cases of our Conjecture \ref{conj-class}. 
In this section, we consider one such special case of Conjecture \ref{conj-class}, the verification of which is due to C. M{\oe}glin. As a general reference for A-packets on not necessarily quasi-split classical groups, we refer to \cite{MR1}. In what follows, we will
use $\SO(V)^+$
or $\SO_{2\ell+1}^+$ (resp., $\SO(V)^-,$ or $ \SO_{2\ell+1}^-$)
to denote  the split (resp. non-split) orthogonal group in odd number of variables of discriminant 1; the
superscript $\pm$ will be omitted when we consider either of the two possibilities.

%This gives some hints of what Arthur packets look like, and how our present conjectures
%may be related to the earlier conjectures in the tempered case {\it both} because of
%some (parabolic) inductive process, and through A-packets which contain tempered representations.

\vskip 5pt

To describe the special case, we shall introduce the following notations. 
Let $\rho$ be an irreducible $d$-dimensional representation of $W(k)$ and write $[b]$ for the $b$-dimensional irreducible representation of $\SL_2(\C)$.  Then an irreducible representation of $WD(k) = W(k) \times \SL_2(\C) \times \SL_2(\C)$ has the form $\rho \otimes [a] \otimes [b]$ for integers $a,b >0$.   

\vskip 5pt

 Assume now that $\rho$ is an irreducible  $d$-dimensional orthogonal representation of $W(k)$.  We shall consider the following A-parameters
\[  M_A = \rho \otimes [2] \otimes [1] + \rho \otimes [1] \otimes [2],   \]
and
\[ N_A = \rho \otimes [1] \otimes [3] + N_0, \]
where $N_0$ is a $d$-dimensional tempered orthogonal L-parameter. 
Then $M_A$ is an A-parameter for $\SO_{4d+1}$ whereas $N_A$ is an A-parameter for $\SO_{4d}$ (orthogonal group
of a quadratic space whose discriminant is dictated by the A-parameter $N_A$).  Observe that the pair $(M_A, N_A)$ is relevant. 
Thus, according to our conjecture, we expect that there should be a pair of representations $(\pi, \sigma)$ in the L-packet $\Pi_M \times \Pi_N$ associated to $M_A$ and $N_A$ such that $\Hom_{\SO_{4d}}(\pi, \sigma) \ne 0$. 
To verify this, we shall first describe the elements in these L-packets more concretely.

\vskip 5pt

The L-parameter associated to $M_A$ is given by:
\[ M =  \rho \otimes [2] + \rho \nu^{1/2} + \rho \nu^{-1/2}, \] 
where $\nu$ is the character of $W(k)$ corresponding to the absolute value of $k^{\times}$ under the local class field theory. The associated L-packet 
$\Pi_M$  consists of various Langlands quotients of the split group $\SO^+_{4d+1}$ and the non-split $\SO_{4d+1}^-$ defined as follows. 
\vskip 5pt

Let $P_d$ be the maximal parabolic subgroup of $\SO_{4d+1}$ stabilizing a $d$-dimensional isotropic space, so that its Levi factor is isomorphic to $\GL_d \times \SO_{2d+1}$. Then $\Pi_M$ consists of the unique irreducible quotients of the standard modules
\[ {\rm Ind}_{P_d}^{\SO_{4d+1}}  \pi_{\rho} |\det|^{1/2}  \otimes  \tau  \twoheadrightarrow J(\rho, \tau),  \]
where $\pi_{\rho}$ is the irreducible cuspidal representation of $\GL_d$ with L-parameter $\rho$ and $\tau$ runs over the Vogan L-packet of $\SO_{2d+1}$ associated to the L-parameter $\rho \otimes [2]$.
\vskip 5pt

To further explicate the L-packet $\Pi_M$, we need to describe the L-packet  of $\SO_{2d+1}$ with L-parameter $\rho \otimes [2]$.
The component group of the discrete L-parameter $\rho \otimes [2]$   is $\Z/2\Z$. Hence,  its associated L-packet has the form   
\[ \Pi_{\rho \otimes [2]} = \{ \tau^+, \tau^-\}, \]
 with $\tau^+$ a discrete series representation of the split group $\SO^+_{2d+1}$ and $\tau^-$ that of the non-split inner form $\SO^-_{2d+1}$. More precisely, $\tau^+$ is the unique irreducible submodule of the induced representation ${\rm Ind}_{P}^{\SO^+(V)} \rho |\det|^{1/2}$, where $P$ has Levi factor $\GL_d$, and $\tau^-$ is a supercuspidal representation of $\SO_{2d+1}^-$. 
\vskip 5pt

To summarize, we have
\[ \Pi_M =  \{ \pi_L^+, \pi_L^-\}  \]
where 
\[\pi_L^+ = J(\rho, \tau^+) \in {\rm Irr}(\SO^+_{4d+1}) \quad \text{and} \quad \pi_L^- = J(\rho, \tau^-) \in {\rm Irr}(\SO^-_{4d+1}). 
\]

\vskip 5pt

 In fact, we can explicate not just the L-packet $\Pi_M$ but also the entire A-packet $\Pi_{M_A}$. 
For the split group  $\SO^+_{4d+1}$, the A-packet $\Pi^+_{M_A}$ has two elements:
\[ \Pi^+_{M_A}=  \{ \pi^+_L, \pi^+_T \} \]
where $\pi^+_L$ lies in the L-packet $\Pi_M$ as described before, and  $\pi^+_T$ is a tempered
representation. 
The standard module with  $\pi^+_L$ as its Langlands quotient has composition series given by:
$$0 \rightarrow \pi^+_g \rightarrow \Ind^{\SO^+_{4d+1}}_{P_d} (\rho \nu^{1/2} \otimes \tau^+) \rightarrow \pi^+_L \rightarrow 0,$$
with $\pi^+_g$ an irreducible generic tempered representation of $\SO^+_{4d+1}$.
The representation $\pi^+_T$, on the other hand,  is the non-generic component of the following
induced representation  whose other irreducible constituent is $\pi^+_g$:
  $$   \pi^+_T + \pi^+_g = \Ind^{\SO^+_{4d+1}}_{P_{2d} } \St(\rho,2). $$
Here $\St(\rho,2)$ is the generalized Steinberg representation of $\GL_{2d}$ (the Levi factor of $P_{2d}$). 
Observe that the character of $\pi^+_L-\pi^+_T$ is a stable distribution, since:
 \[ \tag{A}\pi^+_L-\pi^+_T = \Ind^{\SO^+_{4d+1}}_{P_d} (\rho \nu^{1/2} \otimes \tau^+) -  \Ind^{\SO^+_{4d+1}}_{P_{2d}} \St(\rho,2).\]
\vskip 5pt

  From (A), we obtain  the packet $\Pi^-_{M_A}$ on the non-split orthogonal group $\SO^-_{4d+1}$ by transfer of the stable distribution
 (A). Since the second induced representation on the right of (A) does not transfer to $\SO^-_{4d+1}$ (as the parabolic subgroup $P_{2d}$ is irrelevant for $\SO^-_{4d+1}$), 
 this gives:
     \[ \tag{B}   \Pi^-_{M_A} = \{ \pi_L^-  \} = \{   \Ind^{\SO^-_{4d+1}}_{P_d^-} (\rho \nu^{1/2} \otimes \tau^- )  \} .\]
     The standard module $\Ind^{\SO^-_{4d+1}}_{P_d^-} (\rho \nu^{1/2} \otimes \tau^- )$ is known to be
     irreducible,  thus is equal to
     its Langlands quotient $\pi_L^-$.
   This fact will be crucially used in the restriction problem considered below.
     \vskip 5pt

Now we consider the orthogonal A-parameter 
\[ N_A = \rho \otimes [1] \otimes [3] + N_0. \]
Its associated L-parameter is
\[ N = (N_0 + \rho) + \rho \nu + \rho \nu^{-1}. \]
If $Q_d$ denotes the maximal parabolic subgroup of $\SO_{4d}$ with  Levi factor $\GL_d \times \SO_{2d}$, the elements of the L-packet $\Pi_N$ consists of Langlands quotient of the standard modules
\[  {\rm Ind}_{Q_d}^{\SO_{4d}} \pi_{\rho} |\det|  \otimes \tau' \twoheadrightarrow J(\rho, \tau'), \]
as $\tau'$ runs over the Vogan L-packet associated to the L-parameter $N_0 + \rho$. 
 
\vskip 10pt

 To understand the restriction problem for the pair $(M_A, N_A)$ on $\SO_{4d+1} \times \SO_{4d}$, we first need to understand the restriction problem for the tempered pair $( \rho \otimes [2], \rho + N_0)$ of $\SO_{2d+1} \times \SO_{2d}$. This latter problem has been understood since Waldspurger has proven the tempered GGP conjecture. 
 \vskip 5pt
 
 More precisely,  for any tempered L-parameter $\Sigma$ of $\SO_{2d}$, with associated L-packet $\Pi_{\Sigma}$, let us set
 \[ d(\tau^{\pm}, \Sigma) = \sum_{\sigma \in \Pi_{\Sigma}} d(\tau^{\pm}, \sigma), \]
 where we recall that
 \[  \Pi_{\rho \otimes [2]} = \{ \tau^+, \tau^-\}. \]
  Then one knows by tempered GGP  that 
 \[ d(\tau^+, \Sigma) + d(\tau^-, \Sigma) = 1. \]
One may ask: for which  $\Sigma$ is $d(\tau^-, \Sigma) = 1$?  Applying Proposition \ref{wald}, we deduce:
\vskip 5pt

\begin{cor} \label{wald-cor}
  Let $\rho$ be an irreducible orthogonal   representation of $W(k)$ of dimension $d$.
 Let $\{ \tau^+, \tau^-\}$ be the L-packet of $\SO_{2d+1}$ associated to the discrete L-parameter $\rho \otimes [2]$ and let $\Sigma$ be a tempered L-parameter of $\SO_{2d}$.
  Then
  \[ d(\tau^-, \Sigma) = 1 \Longleftrightarrow  \Sigma = \rho +  N_0 \quad \text{ for some tempered $N_0$.} \]
  \end{cor}
\vskip 10pt

 %We will use the following corollary in what follows. It is a consequence of the
%proof of the GGP conjectures and the above proposition.

%The following lemma is clear.

%\begin{lemma} Let $\tau'$ be a tempered representation of $\SO(2d)$ (of any rank) with L-parameter $\rho + \psi_{\temp}$. Then the induced representation   $\Ind^{\SO(4d)}_{\GL(d) \times \SO(2d) } \rho \nu \otimes \tau'$ has a unique irreducible quotient
 % which is in the Langlands packet inside the Arthur packet $A(\psi_2)$.
 % \end{lemma}

The following proposition, which is the main result of this section,  lends some support to Conjecture \ref{conj-class}.

\begin{prop} Let 
  \begin{eqnarray*} M_A & = & \rho \boxtimes [2] \boxtimes [1] + \rho \boxtimes  [1] \boxtimes [2], \\
    N_A & = & \rho \boxtimes [1] \boxtimes [3] + N_0, \end{eqnarray*}
  be a relevant pair of A-parameters for $\SO_{4d+1} \times \SO_{4d}$
where $\rho$ is an irreducible orthogonal   representation of $W(k)$ of dimension $d$. Let
    \[ 
  (\pi, \sigma)  \in \Pi_M  \times \Pi_N \]
  be the representations in the associated L-packets indexed by the distinguished character of the relevant component group as given in Conjecture \ref{conj-class}(c).
  Then
  \[ d(\pi, \sigma)  \ne 0. \] 
        \end{prop}
\begin{proof}
In the context of Corollary \ref{wald-cor}, let 
\[ (\tau , \tau') \in \Pi^-_{\rho \boxtimes [2]} \times  \Pi^-_{\rho + N_0} \]
be such that
\[ \Hom_{\SO^-_{2d}}(\tau, \tau') \ne 0, \]
so that the pair $(\tau, \tau')$ corresponds to the distinguished character of the component group for $(\rho\boxtimes [2]) \times (\rho + N_0)$
\vskip 5pt

Let $P^-_d $ be a parabolic subgroup in  $G=\SO^-_{4d+1}$  stabilizing an isotropic subspace of dimension $d$.
 The subgroup $\SO^-_{4d} \subset  \SO^-_{4d+1}$
  operates on the corresponding flag variety $ \SO^-_{4d+1}/P^-_d$  with two orbits: an open dense orbit and  and a closed orbit which is
 given by $\SO^-_{4d}/Q^-_d$ where $Q^-_d$ is the parabolic in $\SO^-_{4d}$ with Levi subgroup $\GL_d \times \SO^-_{2d}$. 
 
 %Take note however that  the unipotent radical of $Q^-_d$ is not contained the unipotent radical of $P^-_d$.
\vskip 5pt

Via restriction to the closed orbit $\SO^-_{4d}/Q_d$, one has a surjective $\SO^-_{4d}$-equivariant homomorphism
  $$ \pi_L^- =   \Ind^{\SO^-_{4d+1}}_{P_d^-} \rho \nu^{1/2} \otimes \tau  \rightarrow
    \Ind^{\SO^-_{4d}}_{Q_d^-} \rho \nu \otimes \tau|_{\SO^-_{2d}}  \rightarrow \Ind^{\SO^-_{4d}}_{Q^-_d}  \rho \nu \otimes \tau'    , $$
 Composing this map with the projection from the last standard module to its Langlands quotient $\sigma^-$ (which is an element of $\Pi^-_N$), 
 we have thus shown that $d(\pi^-_L, \sigma^-) \ne 0$. Moreover, by the properties of the local Langlands correspondence, the characters of component groups associated to the representations $\pi_L^-$ and $\sigma^-$ are inherited from those of $\tau$ and $\tau'$ and is equal to the distinguished character for $(M, N)$. 
 This completes the proof of the proposition.    \end{proof}

%\begin{remark}
%The proof above shows the existence of some quotients $\Hom_{\SO(4d)}(\pi, \tau)$
%using the closed orbit, but  it is too simple to give anything further. For example,
%it says nothing about what quotients $\pi_L$ has for the subgroup $\SO(4d)$ inside the split
%$\SO^+(4d+1)$, and neither 
%gives an `if and only if' result for non-split
%$\SO^-(4d+1)$.  \end{remark}

\section{Global Conjecture}  \label{S:global}
In this section, we shall formulate a conjecture for the global analog of the restriction problem, which concerns the non-vanishing of automorphic period integrals. 
Thus, let $F$ be a global field with ring of ad\`eles $\A$ and $F/F_0$ a separable extension with $[F:F_0] \leq 2$. 
\vskip 5pt

For a connected reductive group $G$ over $F$,   let $\mathcal{A}(G)$ denote the space of automorphic forms on $G$ (with a fixed unitary central character if $G$ is not semisimple) which is a $G(\A)$-module. Any irreducible subquotient of $\mathcal{A}(G)$ is called an automorphic representation of $G$. 
One has the following natural submodules  
\[  \mathcal{A}_{cusp}(G) \subset \mathcal{A}_{disc}(G)   \subset \mathcal{A}(G)  \]
consisting of the cusp forms and the square-integrable (modulo center) automorphic forms respectively.
The submodules $\mathcal{A}_{cusp}(G)$ and $\mathcal{A}_{disc}(G)$ are semisimple and any irreducible summand of
$\mathcal{A}_{cusp}(G)$  is called a cuspidal automorphic representation and any irreducible summand of
$\mathcal{A}_{disc}(G)$  is called a discrete automorphic representation. For quasi-split classical groups $G$,  the discrete automorphic representations are classified by Arthur (for symplectic and orthogonal groups) and Mok (for unitary groups)  in terms of discrete global A-parameters.
In particular, the local components of discrete automorphic representations are among those considered by our local conjectures in \S \ref{S:GLn} and \S \ref{S:classical}. \vskip 5pt

One may also consider the unitary representation of $G(\A)$ on $L^2(G(F) \backslash G(\A))$ (with a fixed central character), which possesses a direct integral decomposition.  The 
irreducible representations which are weakly contained in this direct integral decomposition are unitary
automorphic representations (by the theory of Eisenstein series).
They give a subset of the unitary dual of $G(\A)$,  the closure of which (in the Fell topology)
is called the automorphic dual $\widehat{G(\A)}_{aut}$.  One can consider elements of $\widehat{G(\A)}_{aut}$
as submodules of $\mathcal{A}(G)$, and let $\mathcal{A}_{aut}(G)$ denote
the submodule generated by the irreducible automorphic sub-representations in the automorphic dual, so that 
\[  \mathcal{A}_{cusp}(G) \subset \mathcal{A}_{disc}(G)   \subset \mathcal{A}_{aut}(G) \subset \mathcal{A}(G) \]
 For classical groups, the automorphic representations in the automorphic dual are in fact classified by (not-necessarily-discrete) A-parameters through the work of Arthur, cf. \cite{Art2} and Mok, cf. \cite {Mok}. In other words, the local components of these automorphic representations are precisely the representations considered in our local conjecture.      

\vskip 5pt

Now we place ourselves in the global setting of \cite{GGP}, so that we have a pair $W \subset V$ of non-degenerate 
$\epsilon$-Hermitian spaces over $F$ satisfying the conditions highlighted in \S \ref{S:classical}.  As in the local case, we have a pair of groups
\[   H = G(W) \cdot N  \hookrightarrow G = G(V) \times G(W). \] 
As explained in \cite[\S 23]{GGP}, there is a automorphic representation $\nu$ of $H(\A)$ which is a character when $\epsilon = +$ and is essentially a Weil representation when $\epsilon = -$.  Thus, one may consider the global period integral
\[  F(\nu):  \mathcal{A}_{cusp}(G)  \longrightarrow \C  \]
defined by an absolutely convergent integral 
\[  F(\nu) (f)  = \int_{H(\A) \backslash H(\A)} f(h) \cdot \overline{\nu(h)} \, dh \]
when $\epsilon = +$ and by an analogous integral involving theta functions when $\epsilon= -$. 
\vskip 5pt

The above definition of $F(\nu)$ on the cuspidal spectrum is sufficient for the restriction problem considered in \cite{GGP} since we only dealt with tempered global A-packets (or equivalently tempered L-packets) there: the automorphic representations in these tempered L-packets are necessarily cuspidal. In this paper, we are dealing with possibly non-tempered A-parameters. The automorphic representations in global A-packets are no longer necessarily cuspidal. This means that, for a meaningful consideration  of  the global period problem, one needs to extend the definition of $F(\nu)$ to the larger space $\mathcal{A}_{disc}(G)$ or even  the still larger space $\mathcal{A}_{aut}(G)$.
\vskip 5pt

The definition of a regularized period integral on the non-cuspidal part of the spectrum is a basic problem in the analysis of the spectral side of the relative trace formula. Recently, a general definition of such  a regularized period integral over reductive subgroups of $G$
was given by Zydor \cite{Zy}, following earlier work of Jacquet-Lapid-Rogawski \cite{JLR} and Ichino-Yamana \cite{IY}. In particular, the work of Ichino-Yamana  \cite{IY} provides a regularized period integral on the space of automorphic forms on $\GL_n \times \GL_{n+1}$ and $\U_n \times \U_{n+1}$.
\vskip 5pt

{\em For the purpose of this section, we shall assume the working hypothesis that one has a canonical equivariant extension of $F(\nu)$ to $\mathcal{A}_{aut}(G)$.} 
\vskip 5pt

With this background and caveat, we can now formulate our global conjecture.

 \vskip 5pt

\begin{conj}   \label{C:global}
Let $\pi \otimes \pi'$ be an irreducible representation of $G(\A) = \G(V)(\A) \times \G(W)(\A)$ which occurs as a sub-representation of $\mathcal{A}_{aut}(G)$. Then the restriction of the regularized period integral $F(\nu)$ to $\pi \otimes \pi'$ is nonzero  if and only if the following conditions hold:
\vskip 5pt

\begin{itemize}
\item[1.] the  pair of global A-parameters $(M_A, N_A)$ associated to $\pi \otimes \pi'$ is relevant.  

\vskip 5pt

\item[2.] the local multiplicity $d(\pi_v, \pi'_v)$ is nonzero  for all places $v$ of $F$.  
\vskip 5pt

\item[3.] The ratio of L-functions $L(M,N, s)$ defined by (\ref{E:R}) in the case of $\GL_n$ and (\ref{E:R2}) in the case of classical groups
is nonzero at $s=0$.
\end{itemize}
Further, if conditions (1) and (3) hold, then there exists a globally relevant pure inner form 
$G' = \G(V') \times \G(W')$ of $G$ and  an automorphic representation $\tilde{\pi} \otimes \tilde{\pi}'$  of $G'$ with the same A-parameter $(M_A, N_A)$ such that $F(\nu)$ is nonzero on
$\tilde{\pi} \otimes \tilde{\pi}'$.
  \end{conj}

\vskip 10pt

Let us make a few remarks about this global conjecture. 
\vskip 5pt

\begin{itemize}
\item As we noted in \cite[Pg. 88]{GGP}, the formulation of the above conjecture essentially decouples the global restriction problem from the local  ones.  
Unfortunately, as we noted before, our local conjecture is somewhat deficient, and does not allow one to completely address the condition (2). 
 In the next  section, we shall consider the global restriction problem starting from our local conjecture and examine its interaction with the Arthur multiplicity formula analogous to what we did in \cite[\S 26]{GGP}.
\vskip 5pt

\item  It is reasonable to ask what one knows a priori about the analytic behaviour of the function $L(M,N,s)$ in (3) at $s = 0$.  Condition (3) seems only reasonable if one knows a priori that $L(M,N, s)$ is holomorphic at $s= 0$. If $M_A$ and $N_A$ are tempered, this is automatic since the adjoint L-functions in the denominator of $L(M,N, s)$ do not vanish at $s=0$. 
  We shall see later that when $(M_A, N_A)$ is a relevant pair, then $L(M,N,s)$ is holomorphic at $s= 0$.  Moreover, the function
  $L(M,N, s)$ decomposes into a product of various automorphic L-functions  and the curious reader will wonder which factors in $L(M,N, s)$ have the potential to contribute a zero at $s = 0$. We address this question at the end of this section when we compute the function $L(M,N,s)$ more explicitly.

\vskip 5pt
\item As mentioned earlier, 
  Ichino and Yamana have defined a regularized global period integral in \cite{IY}  on automorphic forms of $G = \GL_{n} \times \GL_{n+1}$. On the cuspidal spectrum, the global period integral is known to be nonzero on any irreducible cuspidal $\pi \otimes \pi'$ if and only if
$$L(\frac{1}{2}, \pi \otimes \pi') \not = 0.$$
  Ichino-Yamana showed further that when $\pi \otimes \pi'$ is in the part of the discrete spectrum orthogonal to the cuspidal spectrum and the  one dimensional summands, their regularized period integral $F(\nu)$ vanishes on $\pi \otimes \pi'$. As we now explain, this is consistent with the above conjecture, and in fact follows from our local theorem \ref{thm-gln}.
  \vskip 5pt
  
 Indeed, by a result of M{\oe}glin-Waldspurger, a representation $\pi \subset \mathcal{A}_{disc}(\GL_n)$ has an irreducible A-parameter, i.e. one of  the form $M \otimes \Sym^d(\C^2)$ where $M$ is a cuspidal representation of $\GL_r$ with $r \cdot (d+1)  =n$.  Moreover, $\pi$ is cuspidal if and only if $d=0$. Now given $\pi \otimes \pi'$ in the discrete spectrum of $\GL_n \times \GL_{n+1}$, we have a pair of A-parameters
\[   M \otimes \Sym^d(\C^2) \quad \text{and} \quad N \otimes \Sym^e(\C^2).  \]
It is easy to see that for this pair of A-parameters to be relevant,  we must have either 
\vskip 5pt
\begin{enumerate}
\item  $d = e = 0$, so that $\pi \otimes \pi'$ is cuspidal,  or
\item   $d +1= n = e$ and $M = N$ is a 1-dimensional character $\chi$ of $\GL_1$, in which case $\pi = \chi \circ \det$ and $\pi' = \chi \circ \det$. 
\end{enumerate}
\vskip 5pt

Since at all but finitely many places of $F$, the local representations $\pi_v$ and $\pi'_v$ are unramified (hence parametrized by representations of $WD(F_v) \times \SL_2(\C)$ on which $WD(F_v)$ acts through $W(F_v)$), 
Theorem \ref{thm-gln} can be applied to  give a local proof of the theorem of Ichino-Yamana from \cite{IY} about
the vanishing of the regularized period integral on the (non-one-dimensional and non-cuspidal part of) the discrete spectrum of $\GL_n(\A) \times \GL_{n+1}(\A)$.

\vskip 5pt

\item  On the other hand, still in the setting of $G= \GL_n \times \GL_{n+1}$,
  an ongoing work of Chaudouard and Zydor on the spectral side of the Jacquet-Rallis relative trace formula indicates that the regularized periods of some representations in the automorphic dual do intervene in the analysis. Moreover, only relevant pairs of A-parameters can contribute in this analysis. These results, together with the ones of Ichino-Yamana recalled above, indicate that in the global setting, it is natural  to consider the restriction problem for $\mathcal{A}_{aut}(G)$  and not just for the discrete spectrum $\mathcal{A}_{disc}(G)$. However, for classical groups, where discrete A-parameters are simply multiplicity-free sum of
  selfdual representations of the appropriate sign and are not necessarily irreducible,  the consideration of the discrete spectrum should already provide many interesting examples and checks on our global conjecture.
\end{itemize}
\vskip 10pt

\begin{remark}
  It can happen that two global A-parameters are not relevant but they are locally relevant at all places. For example, for a non-CM automorphic representation $\Pi$ on $\PGL_2(\A_F)$, which is a principal series at all places $v$ of $F$, then at each
  local place $v$,  the L-parameter of $\Sym^2(\Pi_v)$ will contain the trivial representation of $WD(F_v)$. Hence
  the A-parameters $\Sym^1(\C^2)$ of $L(F) \times \SL_2(\C)$ on which $L(F)$ acts trivially, and the
  L-parameter of $\Sym^2(\Pi)$ are locally relevant at all places, but not globally relevant.
Thus the local-global principle does not hold for {\it relevance} in general. However, in many practical situations, it
poses no problem to conclude the  relevance of global parameters by local arguments (if one knew local relevance at unramified places for example) as we  saw above for the local proof of the result of Ichino-Yamana in the case of $(\GL_n(\A_F),\GL_{n-1}(\A_F))$.
\end{remark}
\vskip 5pt

\begin{remark}
  Just as understanding multiplicities in a local $A$-packet is of interest (see  Remark \ref{Rem:multiplicities}), so is
  a global analogue of it: for a given group $G(\A) = G(V)(\A) \times G(W)(\A)$ and a global A-packet of automorphic
  representations  on $G(\A)$ (thus a collection of  nearly equivalent automorphic representations on $G(\A)$),
one may ask how many representations in this global A-packet support a nonvanishing global period integral. In the case of tempered A-parameters, the answer is at most one. 
For general A-parameters, considering the restriction problem for the trivial representation of $\PGL(2) \times \PGL(2)$ to $\PGL(2)$ (resp., the
trivial representation of $\PD^\times \times \PD^\times$ to $\PD^\times$
where $D$ is any quaternion division algebra over the global field $F$), we find that global period integral is nonzero on $\PGL(2)$ as well as
$\PD^\times$ for any $D$, so global multiplicity 1 for non-vanishing of period integrals fails, at least when we allow the group   $G(\A) = G(V)(\A) \times G(W)(\A)$ to vary. Can one formulate a reasonable answer to this question?
  \end{remark}
  \vskip 5pt
We conclude this section with a result about the function $L(M,N, s)$ for a relevant pair $(M_A, N_A)$ of discrete A-parameters.  To describe it, we need to introduce some more notation.  Assume for simplicity that we are dealing with the Bessel case, so that $M_A$ and $N_A$ are selfdual or conjugate-selfdual representations of opposite sign; the Fourier-Jacobi case can be similarly treated.  To fix ideas, we shall assume that the sign of $M_A$ is $b(M_A) = -1$ whereas the sign of $N_A$ is $b(N_A) = +1$.  In the following, we will use the symbol $M_i$ to denote  a selfdual (or conjugate-selfdual) representation of sign $-1$ whereas $N_j$ will denote one with sign $+1$. 
With this convention and understanding,  one can decompose the relevant pair $(M_A, N_A)$ of discrete A-parameters as follows:
\vskip 5pt

\begin{equation} \label{E:MA0}
  M_A =       \bigoplus_{i \in I} M_i \boxtimes \Sym^{m_i-1}(\C^2)   \oplus \bigoplus_{j \in J}  N_j \boxtimes \Sym^{n_j-1}(\C^2),  \end{equation} 
and
\begin{equation} \label{E:NA0}   
N_A =    \bigoplus_{i \in I} M_i \boxtimes \Sym^{m'_i -1}(\C^2)  \oplus \bigoplus_{j \in J} N_j \boxtimes  \Sym^{n_j'-1}(\C^2), \end{equation}
where  the numbers $m_i$, $n_j$, $m_i'$ and $n'_j$ are integers $\geq 0$ satisfying
\[  m_i-1 \equiv  n_j \equiv   0 \mod 2 ,\]
and
\[  |m_i - m_i' |  = 1  = |n_j - n_j'|. \]
Here, we understand $\Sym^{-1}(\C^2)$ to be $0$.
Moreover, we assume that $M_i$ (for $i \in I$) and $N_j$ (for $j \in J$) are irreducible.   In other words,  each irreducible summand in $M_A$ has a partner in $N_A$ and vice versa.
As $M_A$ and $N_A$ are discrete A-parameters, these are multiplicity-free decompositions.  
 \vskip 5pt
 
 Now consider a pair of indices $(i,j) \in I \times J$. Assume without loss of generality that $m_i > n_j$ (recalling that they are of different parity). Since $m_i' = m_i \pm 1$ and $n_j' = n_j \pm 1$, we see that typically, one would expect $m_i' > n_j'$.  This is the case, for example, if $m_i - n_j > 2$. We now make the following definition:
 \vskip 5pt
 
 \begin{definition}(Special pairs) \label{D:special} Let $M_A,N_A$ be a pair of relevant discrete global A-parameters, which are a
   sum of distinct irreducible representations as in \ref{E:MA0} and \ref{E:NA0}.
 We call a pair $(i,j) \in I \times J$ special if $m_i  - n_j$ and $m'_i - n'_j$ are of opposite sign and denote the subset of special $(i,j)$ by $(I \times J)^{\heartsuit}$, thus
\[(I \times J)^{\heartsuit} = \{(i,j) \in (I \times J)| \, (m_i,m'_i)= (n'_j,n_j) \}.\]
Equivalently, special pairs are direct summands in $M_A$ and $N_A$ of the form
\[ 
    V \boxtimes \Sym^{i-1}(\C^2)   \oplus   W \boxtimes \Sym^{i}(\C^2) \subset M_A, \]
\[
V \boxtimes \Sym^{i}(\C^2)  \oplus  W \boxtimes  \Sym^{i-1} (\C^2) \subset N_A,
   \]
for some $i \geq 0$, with $V,W$ irreducible, one  orthogonal, and the other symplectic. 
 \end{definition}

 Here is the main result of this section.

 \begin{thm}  \label{T:globalR} Let $M_A,N_A$ be a pair of relevant discrete global A-parameters, which are a
   sum of distinct irreducible representations as in \ref{E:MA0} and \ref{E:NA0},  with $M_A$ symplectic and $N_A$ orthogonal. 
   Then the  function
\begin{eqnarray*}  L(M,N, s)  & = & \frac{L(M \otimes N, s + ~1/2)}{L(\Sym^2 M \oplus \wedge^2 N, s + 1)} \end{eqnarray*}
   is holomorphic at $s=0$.  Moreover, with the notation as above, and with   $(I \times J)^{\heartsuit} $ the set of special pairs,  one has
\[  L(M,N, 0)  = C \cdot \prod_{(i,j) \in (I \times J)^{\heartsuit}}  L( M_i \otimes N_j, 1/2)\]
 for some $C \in \C^{\times}$; more precisely, the order of zero at $s=0$ of the  L-function on the left is the same as that of the L-function on the right.
 \end{thm}

 \begin{proof}
To simplify notation, for any global L-parameter $\Pi$, let us set:
\[ L(\Pi \boxtimes \Sym^{d-1}(\C^2), s)  := \prod_{i=0}^{d-1} L(\Pi, s  + \frac{d-1-2i}{2}). \]
Before we begin the proof of the theorem, let's observe that for any irreducible cuspidal automorphic representation
   $\Pi_1$ of $\GL_m(\A)$, and $\Pi_2$ of $\GL_n(\A)$, we have the following assertions due to Jacquet, Piatetski-Shapiro and Shalika \cite{JPSS} for  the completed, i.e., including the factors at infinity, Rankin product L-function $L(s, \Pi_1 \otimes \Pi_2)$: 
\vskip 5pt

   \begin{enumerate}
   \item[(a)]  $L( (\Pi_1 \otimes \Pi_2) \boxtimes \Sym^{d-1}(\C^2), s+ \frac{1}{2})$ has a zero at $s=0$ if and only if $d$ is odd and  $L(  \Pi_1
     \otimes \Pi_2, \frac{1}{2}) = 0$.
     \vskip 5pt
     
\item[(b)]  $L( (\Pi_1 \otimes \Pi_2) \boxtimes \Sym^{d-1}(\C^2), s+ \frac{1}{2})$ has a pole at $s=0$ if and only if $d \ne 0$ is even, $m=n$,
  and  $\Pi_1 \cong \Pi_2^\vee$, in which case the pole is simple.
\vskip 5pt

\item[(c)]  $L( (\Pi_1 \otimes \Pi_2) \boxtimes \Sym^{d-1}(\C^2), s+ 1)$ has a zero at $s=0$ if and only if $d\ne 0$ is even, and  
$L(\Pi_1 \otimes \Pi_2, \frac{1}{2}) = 0$.
\vskip 5pt
  \item[(d)]  $L( (\Pi_1 \otimes \Pi_2) \boxtimes \Sym^{d-1}(\C^2), s+ 1)$ has a pole at $s=0$ if and only if $d$ is odd, $m=n$,
  and  $\Pi_1 \cong \Pi_2^\vee$, in which case the pole is simple.
   \end{enumerate}

       We first prove the theorem in the  case when $M_A,N_A$ are the global A-parameters,
   \begin{equation}  \label{E:SMA}
   M_A  =      M_1 \boxtimes \Sym^{i-1}(\C^2)   \oplus   N_1 \boxtimes \Sym^{j-1}(\C^2),  \end{equation}
   
 \begin{equation} \label{E:SNA}
  N_A  =    M_1 \boxtimes \Sym^{i-2}(\C^2)  \oplus  N_1 \boxtimes  \Sym^{j-1 +\epsilon} (\C^2),
   \end{equation}
where $\epsilon = \pm 1$, $M_1$ is symplectic and $N_1$ is orthogonal.
   It follows that:
   \[ M_A \otimes N_A = C +D, \]
   where
   \begin{eqnarray*}
     C & =  & (M_1\otimes M_1) \boxtimes  (\Sym^{i-1}(\C^2)  \otimes \Sym^{i-2}(\C^2)) +  (N_1\otimes N_1)  \boxtimes (\Sym^{j-1}(\C^2)  \otimes \Sym^{j-1 + \epsilon }(\C^2)) \\
     D & = & (M_1 \otimes N_1) \boxtimes [\Sym^{i-1}(\C^2)\otimes \Sym^{j-1 + \epsilon}(\C^2) +  \Sym^{i-2}(\C^2)  \otimes \Sym^{j-1}(\C^2) ]
   \end{eqnarray*}

   Similarly, we write,
\[ \Sym^2(M_A) \oplus  \Lambda^2(N_A) = C' +D', \]
   where
   \begin{eqnarray*}
     C' & =  & \Sym^2(M_1) \boxtimes  [\Sym^2(\Sym^{i-1}(\C^2)) + \Lambda^2(\Sym^{i-2}(\C^2))]  \\
     & & + \Lambda^2(M_1) \boxtimes  [\Lambda^2(\Sym^{i-1}(\C^2)) +
         \Sym^2(\Sym^{i-2}(\C^2))]  \\
     & & + \Sym^2(N_1) \boxtimes  [\Sym^2(\Sym^{j-1}(\C^2)) + \Lambda^2(\Sym^{j-1 + \epsilon}(\C^2))] \\
     & & + \Lambda^2(N_1) \boxtimes  [\Lambda^2(\Sym^{j-1}(\C^2)) +
         \Sym^2(\Sym^{j-1 + \epsilon}(\C^2))],  \\
     D' & = & (M_1 \otimes N_1) \boxtimes [\Sym^{i-1}(\C^2)  \otimes \Sym^{j-1}(\C^2)  \oplus \Sym^{i-2}(\C^2)\otimes \Sym^{j-1 + \epsilon}(\C^2)] .
   \end{eqnarray*}
   With these notations, we have,
   \begin{eqnarray*}  L(M,N, s)  & = &   \frac{L(C, s + ~1/2)}{L(C', s + 1)} \cdot \frac{L(D, s + ~1/2)}{L(D', s + 1)}.
\end{eqnarray*}

    Since $M_A$ is symplectic and $N_A$ is orthogonal, the integers $i$ and $j$ have opposite parity. From the assertions (a)-(d) at the beginning of the proof,    we find that:
    \begin{enumerate}
    \item The terms in $C$ can only give rise to a pole at $s=0$  for $L(C,s+\frac{1}{2})$.

    \item The terms in $C'$ can only give rise to a pole at $s=0$  for $L(C',s+1)$.

    \item The terms in $D$ can only give rise to a zero at $s=0$  for $L(D,s+\frac{1}{2})$.

    \item The terms in $D'$ can only give rise to a zero at $s=0$  for $L(D',s+1)$.

    \end{enumerate}

    Thus     
    \[ {\rm ord}_{s=0}\frac{L(C, s + \frac{1}{2})}{L(C', s + 1)} =
    {\rm ord}_{s=0} L(C, s+ \frac{1}{2}) - {\rm ord}_{s=0} L(C', s+1) \]
    is equal to
    \[
     {\rm ord}_{s=0}  L(M_1 \otimes N_1, s +1/2)^{ \min \{i, j + \epsilon\} + \min\{i-1,j\} -\min \{i-1, j + \epsilon\} - \min \{i,j\} } . \]
      It can be easily  checked that for $i$ and $j$ of different parity, the exponent is  0 in all cases, except when
   \begin{eqnarray*} \label{E:MA}
  M_A & = &     M_1 \boxtimes \Sym^{i-1}(\C^2)   \oplus   N_1 \boxtimes \Sym^{i-2}(\C^2)  \\
N_A & = &   M_1 \boxtimes \Sym^{i-2}(\C^2)  \oplus  N_1 \boxtimes  \Sym^{i-1} (\C^2),
   \end{eqnarray*}
  in which case it is $1$. Likewise, we have
     $$ {\rm ord}_{s=0}\frac{L(D, s + ~1/2)}{L(D', s + 1)}  = {\rm ord}_{s=0} L( M_1 \otimes N_1, s+1)^{ i-1 + \min \{j,j+\epsilon \} - \{ \frac{i-1}{2} + \frac{j}{2} +\frac{i-1}{2} + \frac{j+\epsilon -1}{2} \} } .$$
 It can be easily  that for either choice of $\epsilon = \pm 1$, the exponent is zero. Thus the theorem is proved in the case
when $(M_A,N_A)$ are given by (\ref{E:SMA}) and (\ref{E:SNA}). 
\vskip 5pt

There is the analogous case
when
\[ 
  M_A  =      M_2 \boxtimes \Sym^{i-1}(\C^2)   \oplus   N_2 \boxtimes \Sym^{j-1}(\C^2)  \]
\[
N_A  =    M_2 \boxtimes \Sym^{i}(\C^2)  \oplus  N_2 \boxtimes  \Sym^{j-1 +\epsilon} (\C^2),
   \]
   where $\epsilon = \pm 1$, $M_2$ is symplectic and $N_2$ is orthogonal. This can be analyzed in the same way, and we find that
if for some $i \geq 0$,
\[ 
  M_A  =      M_2 \boxtimes \Sym^{i-1}(\C^2)   \oplus   N_2 \boxtimes \Sym^{i}(\C^2) , \]
\[
N_A  =    M_2 \boxtimes \Sym^{i}(\C^2)  \oplus  N_2 \boxtimes  \Sym^{i-1} (\C^2),
   \]
we have   
   \[  {\rm ord}_{s=0} L(M,N, s)  = {\rm ord}_{s=0}   L( M_2 \otimes N_2, s+1/2),\]   
   and in all other cases (with $M_A,N_A$ as here):
   \[  {\rm ord}_{s=0} L(M,N, s)  = 0.\]   
\vskip 5pt

Now we prove  the theorem in the general case. Assume that
   \begin{eqnarray*} \label{E:MA3}
  M_A & = &     \bigoplus_{\alpha} V_\alpha,  \\
N_A & = &   \bigoplus_{\alpha} W_\alpha,
   \end{eqnarray*}
   is a relevant pair of discrete global A-parameters with $M_A$ symplectic and $N_A$ orthogonal and where $(V_\alpha, W_\alpha)$ are relevant pairs of irreducible global A-parameters. 

   Then
    \begin{align} \label{E1E2}
     M_A \otimes N_A &=  \sum_{\alpha} V_\alpha \otimes W_\alpha + \sum_{\alpha \not = \beta} (V_\alpha \otimes W_\beta + V_{\beta} \otimes W_{\alpha}), \\
     \Sym^2 M_A  + \Lambda^2 N_A &=   \sum_{\alpha} (\Sym^2 V_\alpha + \Lambda^2 W_\alpha) +
     \sum_{\alpha \not = \beta} (V_\alpha \otimes V_\beta + W_\alpha \otimes W_\beta) ,  \notag \end{align}
   where in both the sums above, the sum $\sum_{\alpha \not = \beta}$ is over un-ordered distinct pair
   of indices $\alpha, \beta$.
   \vskip 5pt
   
   We now  analyze $L(M, N, s)$ in two steps.
   \vskip 5pt

   \noindent{\bf Step 1.}
   We prove that   the diagonal terms
     $$ L(V_\alpha, W_\alpha,s)= \frac{L(V_\alpha \otimes W_\alpha, s + ~1/2)}{L(\Sym^2 V_\alpha \oplus \wedge^2 W_\alpha, s + 1)},$$
   have neither a zero nor a pole      at $s=0$.

   Let 
     \[ V_\alpha= V \boxtimes \Sym^i(\C^2) \quad \text{and} \quad W_\alpha= V \boxtimes \Sym^{i+ \epsilon}(\C^2),\]
      with $V$ irreducible and $\epsilon = \pm 1$. In this case it is easy to see that the numerator and denominators  of $L(s,V_\alpha, W_\alpha)$ can only have poles at $s=0$.
   Assume  first that $V$ is symplectic so that $i$ is even (since $V_{\alpha}$ is symplectic).   Then we have
    \[ - {\rm ord}_{s=0}L(V_\alpha \otimes W_\alpha, s + \frac{1}{2}) =  \min\{i+1, i+1+ \epsilon\} \]
and
\[ - {\rm ord}_{s=0}L(\Sym^2 V_\alpha \oplus \wedge^2 W_\alpha, s + 1) = \begin{cases}
 i, \text{ if $\epsilon =-1$,} \\
  i+1, \text{  if $\epsilon = 1$.} \end{cases} \]
 Therefore we find that  $ L(V_\alpha, W_\alpha, s)$ has neither
     a zero nor a pole at $s=0$. An analogous argument works in the case when $V$ is orthogonal.
     
   \vskip 5pt
     
   \noindent{\bf Step 2.}
   Observe that,
\begin{eqnarray*}
  L(V_\alpha+ V_\beta, W_\alpha+W_\beta, s) & = & \frac{L([V_\alpha + V_\beta]\otimes [W_\alpha +W_\beta], s + ~1/2)}{L(\Sym^2 [V_\alpha+V_\beta] \oplus \wedge^2 [W_\alpha+ W_\beta], s + 1)}, \\
        & = &  L(V_\alpha, W_\alpha, s)  L( V_\beta, W_\beta,s) \frac{L( V_\alpha \otimes W_\beta +
      V_\beta \otimes W_\alpha, s + ~1/2)}{L( V_\alpha \otimes V_\beta + W_\alpha \otimes W_\beta, s + 1)}.
     \end{eqnarray*}

Since  by step 1, $L(V_\alpha, W_\alpha,s)$ and $  L( V_\beta, W_\beta,s)$ have neither a zero nor a pole at $s=0$,
the order of zero or pole at $s=0$ of $ L(V_\alpha+ V_\beta, W_\alpha+W_\beta,s)$ is the same as that of

$$ \frac{L( V_\alpha \otimes W_\beta +
  V_\beta \otimes W_\alpha, s + ~1/2)}{L( V_\alpha \otimes V_\beta + W_\alpha \otimes W_\beta, s + 1)}.$$

Now, by the special case of the theorem treated above, we understand the order of zero or pole at $s=0$ of $ L(V_\alpha+ V_\beta, W_\alpha+W_\beta,s)$ (in terms of special pairs),
and therefore by the two equalities in (\ref{E1E2}), the proof of the theorem is completed.
 \end{proof}

   Using Theorem \ref{T:globalR},  we can reformulate  Conjecture \ref{C:global} for the discrete spectrum as follows:
  \begin{conj}
  Let $M_A \times N_A$ be a discrete A-parameter of the quasi-split group $G(\A) = \G(V)(\A) \times \G(W)(\A)$ with endoscopic decomposition as given in (\ref{E:MA0}) and (\ref{E:NA0}):
\[    M_A =       \bigoplus_{i \in I} M_i \boxtimes \Sym^{m_i-1}(\C^2)   \oplus \bigoplus_{j \in J}  N_j \boxtimes \Sym^{n_j-1}(\C^2) , \] 
and
\[
N_A =    \bigoplus_{i \in I} M_i \boxtimes \Sym^{m'_i -1}(\C^2)  \oplus \bigoplus_{j \in J} N_j \boxtimes  \Sym^{n_j'-1}(\C^2). \]
Then the regularized period integral  is nonzero on the submodule of the automorphic discrete spectrum asscoiated to $M_A \times N_A$ 
on some relevant pure inner form  $G'(\A) = \G(V')(\A) \times \G(W')(\A)$ of $G(\A) = \G(V)(\A) \times \G(W)(\A)$
 if and only if $L(1/2, M_i \otimes N_j)$ is nonzero for all special pairs $(i,j) \in (I \times J)^{\heartsuit}$ (as in Definition \ref{D:special}).
 
   \end{conj}

\vskip 5pt

Finally, it is natural to ask if the refined conjecture of Ichino-Ikeda \cite{II}, which gives a precise formula relating the period integral and the relevant L-value in the case of tempered L-parameters, can be formulated in the nontempered setting. An important issue
in the formulation of the Ichino-Ikeda conjecture
is the definition of the local period functionals which intervene in their conjectural formula. In the tempered case, these local functionals are given by (absolutely convergent) integrals of matrix coefficients and can be characterized using the spectral
decomposition in the $L^2$-theory.  
We do not know how to formulate an extension to the nontempered case, but hope that this paper will stimulate some work in this direction. 

\vskip 10pt

\section{Revisiting the Global Conjecture} \label{S:global2}

We now revisit the global conjecture \ref{C:global} from the viewpoint of the local conjecture \ref{conj-class}.
\vskip 5pt

 Fix a pair of spaces $W_0 \subset V_0$ 
and a pair of global A-parameters $(M_A, N_A)$ for $G(V_0) \times G(W_0)$. This gives rise to local A-parameters $(M_{A,v}, N_{A,v})$ at all places $v$ and their associated  local A-packets. Because our local conjecture does not address the restriction problem for the whole A-packet, we shall restrict our attention only to the associated Vogan L-packet $\Pi_{M_v} \times \Pi_{N_v}$ contained in the A-packet.  
Based on this, we can draw the following conclusions from our local conjecture.
\vskip 5pt

\begin{itemize}
\item[(i)] the pair $(M_{A,v}, N_{A,v})$ of local A-parameters should be relevant if we want a nonzero local multiplicity in $\Pi_{M_v} \times \Pi_{N_v}$. This is implied by the relevance of the pair of global A-parameters $(M_A, N_A)$.   Hence we assume that $(M_A, N_A)$ is relevant in the ensuing discussion.
\vskip 5pt

\item[(ii)] when $(M_A, N_A)$ is relevant, there is a unique representation in $\Pi_{M_v} \times \Pi_{N_v}$ with nonzero multiplicity. This unique representation is indexed by a character $\chi_v$ of the local component group $A_{M,v} \times A_{N,v}$ and is a representation of a pure inner form of $G_v := G(V)(F_v) \times G(W)(F_v)$. We denote it by $\pi_{\chi_v}$.

\vskip 5pt

\item[(iii)]  For almost all $v$, $\chi_v$ is trivial and the representation $\pi_{\chi_v}$ is an unramified representation of $G_{0,v} = G(V_0)(F_v) \times G(W_0)(F_v)$.
\end{itemize}
In view of the above, we can form the restricted direct product groups
\[   
  G_{\A}  =  \prod_v' G_v  \supset H_{\A} = \prod'_v  H_v \]
and the abstract representations
\[  \text{$\pi_{\chi} :=  \otimes'_v \pi_{\chi_v}$ of $G_{\A}$ and $\nu = \otimes_v \nu_v$ of $H_{\A}$.} \]
The abstract representation $\pi_{\chi}$ is the only one in the global L-packet $\Pi_M \times \Pi_N$ which could possibly have a nonzero global  period integral. 
However, before considering the global restriction problem, we need to ensure that it is actually an automorphic representation. For this, we need to address the following questions in turn:
\vskip 5pt

\begin{itemize}
\item[(1)]  Is $H_{\A} \hookrightarrow G_{\A}$  the group of adelic points of algebraic groups 
\[ H = G(W) \cdot N \hookrightarrow G = G(V) \times G(W) \]
associated to a relevant pair of spaces $W \subset V$?

\vskip 5pt
This question can be answered by the same consideration as in \cite[Pg 99]{GGP}. We see that the answer is Yes if and only if 
\[  \epsilon(1/2, M \otimes N)  = 1. \]
Assuming this, we are led to the next question:

\vskip 5pt

\item[(2)] Is the abstract representation $\pi_{\chi}$ of $G_{\A} = G(V)(\A)\times G(W)(\A)$ automorphic?  
 
\vskip 5pt

\item[(3)] Finally, if the answer to Question (2) is affirmative, one can consider the restriction of the regularized period integral $F(\nu)$ to $\pi_{\chi}$ and ask if this restriction is nonzero. 
\end{itemize}

We shall now consider Question (2) in some detail.
To address this question, we can reduce to the case of discrete A-parameters, so we assume that $M_A$ and $N_A$ are discrete. Then the automorphy of $\pi_{\chi}$ is determined by the Arthur multiplicity formula,  where a certain quadratic character $\epsilon_{Art}$ of the global component group of the A-parameter intervenes. More precisely, given the global A-parameter $M_A \times N_A$ with associated L-parameter $M \times N$,  one has the following commutative diagram of component groups
\[  \begin{CD}
 A_{M_A} \times A_{N_A}   @>\Delta>>  \prod_v  ( A_{M_{A,v}} \times A_{N_{A,v}})  \\
 @VVV  @VVV  \\
A_M \times A_N @>\Delta>>  \prod_v (A_{M_v} \times A_{N_v} ).
\end{CD} \]
 The representation $\pi_{\chi} = \otimes'_v \pi_{\chi_v}$ corresponds to a character $\otimes_v \chi_v$ of the group $ \prod_v (A_{M_v} \times A_{N_v} )$. Pulling back by the above diagram, one obtains a character $\chi$ of $A_{M_A} \times A_{N_A}$. Now 
 \[  \text{$\pi_{\chi}$ is automorphic if and only if $\chi = \epsilon_{Art}$.} \]
 The difference between \cite{GGP} and the situation here is that for tempered A-parameters considered in \cite{GGP}, this quadratic character $\epsilon_{Art}$ is trivial, whereas here it  is not necessarily trivial. 

\vskip 5pt

 To explicate the condition $\chi = \epsilon_{Art}$, we need to determine $\epsilon_{Art}$. For this, let us return to the decomposition of $M_A$ and $N_A$ given in (\ref{E:MA0}) and (\ref{E:NA0}).  Then we may write
 \[   A_{M_A}  =  \prod_{i \in I, m_i >0} \Z/2\Z \cdot a_i  \times  \prod_{j \in J, n_j >0}  \Z/2\Z \cdot b_j  \]
and
\[  A_{N_A} = \prod_{i \in I: m_i' >0}  \Z/2\Z \cdot a_i'  \times \prod_{j\in J: n'_j >0} \Z/2\Z \cdot b_j'. \]
In other words, these component groups are vector spaces over $\Z/2\Z$ equipped with  canonical bases.
To specify $\epsilon_{Art}$, it suffices to evaluate the signs $\epsilon_{Art}(a_i)$, $\epsilon_{Art}(b_j)$, $\epsilon_{Art}(a'_i)$ and $\epsilon_{Art}(b_i')$.
\vskip 5pt

\begin{lemma} \label{L:art}
With the above notation, one has:

\[ \epsilon_{Art}(a_i) =  \prod_{j \in J}  \epsilon(1/2, M_i \otimes N_j)^{\min(m_i, n_j)} =  \prod_{j \in J: m_i < n_j} \epsilon(1/2, M_i \otimes N_j). \]

\[  \epsilon_{Art}(b_j)  =  \prod_{i \in I} \epsilon(1/2, M_i \otimes N_j)^{\min(m_i, n_j)} =  \prod_{i\in I: m_i < n_j}  \epsilon(1/2, M_i \otimes N_j). \]

\[ \epsilon_{Art}(a_i')  = \prod_{j \in J} \epsilon(1/2, M_i \otimes N_j)^{\min(m'_i, n'_j)} =  \prod_{j \in J: m'_i > n'_j} \epsilon(1/2, M_i \otimes N_j) \]

\[  \epsilon_{Art}(b_j') = \prod_{i\in I} \epsilon(1/2, M_i \otimes N_j)^{\min(m'_i, n'_j)} = \prod_{i\in I: m'_i > n'_j} \epsilon(1/2, M_i \otimes N_j). \]
\end{lemma}

\vskip 5pt

On the other hand, we have: 
 
\vskip 5pt

\begin{lemma}  \label{L:chi}
The character $\chi$ is given by:
\[  \chi(a_i)  = \epsilon(1/2, M_i \otimes N)  = \prod_{j \in J} \epsilon(1/2, M_i \otimes N_j), \]
 and
\[  \chi(b'_j)  = \epsilon(1/2, M \otimes N_j)  =  \prod_{i\in I}  \epsilon(1/2, M_i \otimes N_j). \]
Moreover,  $\chi(b_j)  =1  =  \chi(a_i')$. 
\end{lemma}
\vskip 10pt

As a consequence of Lemma \ref{L:art} and Lemma \ref{L:chi}, we deduce:
\vskip 5pt

\begin{cor}  \label{C:auto}
The representation $\pi_{\chi}$ in the L-packet associated to $M \times N$ is automorphic if and only if the following holds:
\vskip 5pt
\[   \prod_{j \in J: m_i > n_j} \epsilon(1/2, M_i \otimes N_j) =1  \quad \text{for all $i \in I$ with $m_i >0$;} \]

\[   \prod_{j \in J: m'_i > n'_j}  \epsilon(1/2, M_i \otimes N_j) =1  \quad \text{for all $i \in I$ with $m'_i >0$;} \]

\[  \prod_{i\in I: m_i < n_j}  \epsilon(1/2, M_i \otimes N_j) =1 \quad \text{for all $j \in J$ with $n_j > 0$;} \]

\[  \prod_{i\in I: m'_i < n'_j}  \epsilon(1/2, M_i \otimes N_j) =1 \quad \text{for all $j \in J$ with $n'_j > 0$;} \]   
In particular, when these conditions hold, one has: 
\begin{equation} \label{E:I}
  \prod_{j: (i,j) \in (I \times J)^{\heartsuit}} \epsilon(1/2, M_i \otimes N_j)  =1  \quad \text{ for any fixed $i \in I$.} \end{equation}
Likewise,
\begin{equation}  \label{E:J}
  \prod_{i:  (i,j) \in  (I \times J)^{\heartsuit}} \epsilon(1/2, M_i \otimes N_j)  =1  \quad \text{ for any fixed $j\in J$.} \end{equation}
\end{cor}
\vskip 5pt

\begin{proof}
The first four identities simply follow by equating $\chi$ with $\epsilon_{Art}$.  Equation (\ref{E:I})  follows by dividing the first and second identities, whereas equation (\ref{E:J}) follows by dividing  the third and fourth identities.  
\end{proof}
\vskip 5pt

At this point, we have explicated the answer to Question (2), and it is instructive to recall Theorem \ref{T:globalR}, which says that
\[   L(M,N, 0)  = C \cdot \prod_{(i,j) \in (I \times J)^{\heartsuit}}  L( M_i \times N_j, 1/2) \quad \text{for some $C \in \C^{\times}$.}\]
One can rewrite:
\begin{equation} \label{E:I2}
  L(M,N,0) = C \cdot \prod_{i \in I}   \, \left(  \prod_{j: (i,j) \in (I \times J)^{\heartsuit}}  L(M_i \otimes N_j, 1/2) \right) \end{equation}
or
\begin{equation} \label{E:J2}
  L(M,N,0) = C \cdot \prod_{j \in J}   \, \left(  \prod_{i: (i,j) \in (I \times J)^{\heartsuit}}  L(M_i \otimes N_j, 1/2) \right) .\end{equation}
 Now observe that  (\ref{E:I}) is saying that the global root number of the interior  product of L-functions in (\ref{E:I2}) is $+1$. Likewise, (\ref{E:J}) is saying that the global root number of the interior  product of L-functions in (\ref{E:J2}) is +1. 
 \vskip 5pt
 Given the above, it is not unreasonable to conjecture that, under the hypothesis that $(M_A, N_A)$ is relevant and in the context of Question (3),  the regularized period integral $F(\nu)$ is nonzero on $\pi_{\chi}$ if and only if $L(M,N, 0) \ne 0$, given that the conditions of automorphy of $\pi_{\chi}$ expressed in Corollary \ref{C:auto} hold.
  
\vskip 10pt

\section{Low Rank Examples}  \label{S:example}

In this section, we consider a few examples of our conjectures in low rank groups. The restriction problem considered in this paper has been studied in several low rank examples by various people and we shall examine their results in light of our conjectures. 
\vskip 5pt

\begin{example} $\SO_2 \times \SO_3$ \end{example}
We begin with  this simple example where $G = \SO_3(k) = \PGL_2(k)$ and $\wG = \SL_2(\bC)$. 
Consider the A-parameter 
\[ M_A = M_{\alpha} \otimes \Sym^1(\bC^2)  \]
where $M_{\alpha}$ is  a 1-dimensional orthogonal representation given by a quadratic character
$$\alpha: k^*/k^{*2} \rightarrow \langle \pm 1 \rangle.$$
 The L-packet of the corresponding  L-parameter  $M$ contains a single   representation $\pi(\alpha)$ of dimension one, where the group $G$ acts through a composition of $\alpha$ with the determinant character
$$\det: \PGL_2(k) \rightarrow k^*/k^{*2}.$$
\vskip 5pt

Let $H =\SO_2$ be the subgroup of $G$ fixing a non-isotropic line in the standard three dimensional representation of $G$. Then there is an \'etale quadratic algebra $K$ over $k$ such that $H = K^*/k^*$ and the irreducible representations $\chi$ of $H$ all have dimension one. In this case, the A-parameters and  L-parameters of $H$ are both equal to the two dimensional orthogonal representation 
\[  N = \Ind(\chi) \]
 of $W(k)$. When $K = k + k$, $N = \chi + \chi^{-1}$. When $K$ is a field corresponding to a non-trivial quadratic character $\beta$ of $k^*$, one has $\det(N) = \beta$. In this case, the representation $N$ is irreducible unless $\chi$ is the composition of a quadratic character $\chi_k$ of $k^*$ with the norm, when $N = \chi_k + \chi_k \, \beta$. There are two representations in the Vogan L-packet, for the groups of the two orthogonal spaces of dimension $2$ with the correct discriminant, but only one of these spaces is relevant.
\vskip 5pt

The restriction of the one dimensional representation $\pi(\alpha)$ of $G = \SO_3$ to the subgroup $ H = \SO_2$ is the character $\chi_{\alpha}$ given by the composition of the quadratic character $\alpha$ of $k^*$ with the norm from $K^*$ to $k^*$. (Since the norms from $K^*$ form a subgroup of index two in $k^*$ when $K$ is a field, this character is also the composition of the quadratic character $\alpha .\beta$ with the norm.) As  $\chi_{\alpha} = \chi_{\alpha}^{-1}$, we have
$$\Hom_H(\pi(\alpha) \otimes \chi_{\alpha}, \bC) = \bC.$$
In this case, the A-parameters $(M,N)$ form a relevant pair: $N = \alpha.\beta + \alpha = N_0 + M_{\alpha}$. This is the only possible $N$ of dimension $2$ and the correct determinant which combines with $M$ to give a relevant pair. 
For characters $\chi$ of $H$ which are not equal to $\chi_{\alpha}$, we have
$$\Hom_H(\pi(\alpha) \otimes \chi, \bC) = 0.$$
In particular, the results of this elementary example is in accordance with our local conjecture \ref{conj-class}.
\vskip 5pt

The L-function of the tensor product representation is given by
$$L(M \otimes N, s) = L(\alpha \otimes \Ind(\chi), s + ~1/2) \cdot L(\alpha \otimes \Ind(\chi), s - ~1/2).$$
This has a pole at $s = 1/2$ if and only if the quadratic character $\alpha$ appears in the representation $\Ind(\chi)$, in which case the order of the pole is the multiplicity of the character $\alpha$ in this induced representation. The L-function of the adjoint representation is given by
$$L(\Sym^2 M, s) \cdot L(\wedge^2, s) = L(\bC, s + 1)\cdot  L(\bC, s) \cdot L(\bC, s - 1) \cdot L(\beta, s).$$ 
This has a simple pole at $s = 1$. We conclude that the ratio
$$L(M,N, s) = L(M \otimes N, s + ~1/2)\big /L(\Sym^2 M \oplus \wedge^2N, s + 1)$$
has order less than or equal to zero at $s = 0$ if and only if $\alpha$ appears in $\Ind(\chi)$, or equivalently, if  and only if
$$\Hom_H(\pi(\alpha) \otimes \chi, \bC) = \bC.$$
When $K$ is a field, the parameter is discrete, and we see that $L(M,N,s)$ is regular and non-zero at $s = 0$, When $K = k + k$, the parameter for $\chi_{\alpha}$ is not discrete, and the character $\alpha$ appears with multiplicity $2$ in the induced representation. In this case, the ratio
$L(M,N,s)$ has a simple pole at $s = 0$.

\vskip 15pt

\begin{example} Restriction of 1-dimensional characters. \end{example}
One can easily extend the above example to cover the restriction of one dimensional representations of the split group $\SO_{2n+1} = \SO(V)$ to the special orthogonal group $\SO_{2n}= \SO(W)$ of a codimension one subspace $W$. The one dimensional representations of $\SO_{2n+1}$ have symplectic Arthur parameters
\[  M = M_{\alpha} \otimes \Sym^{2n-1}(\bC^2), \]
 where $M_{\alpha}$ is given by a quadratic character $\alpha$ as above. The associated representation is the composition of $\alpha$ with the spinor norm. Let $\beta$ be the quadratic character given by the discriminant of the even dimensional subspace $W$. Then the only Arthur parameter for the subgroup $\SO_{2n} = \SO(W)$ which gives a relevant pair $(M,N)$ is 
 \[  N = N_0^- + N_{2n-2}^+ \otimes \Sym^{2n-2}(\bC^2) \]
  with $N_0^- = \alpha.\beta$ and  $N_{2n-2}^+ = \alpha$. This is indeed the A-parameter for the restricted one dimensional representation of $\SO_{2n}$.
\vskip 5pt

We can easily calculate the order of $L(M,N,s)$ at $s = 0$ in this case.  A short computation gives:
\begin{align}
\Sym^2(M) &= \Sym^{4n-2}(\bC^2) + \Sym^{4n-6}(\bC^2) + \cdots + \Sym^2(\bC^2), \notag \\
\wedge^2(N) &= \Sym^{4n-6}(\bC^2) + \Sym^{4n-10}(\bC^2) + \cdots + \Sym^2(\bC^2) ~+~ \Sym^{2n-2}(\bC^2) \otimes \beta, \notag \\
M \otimes N &= \Sym^{4n-3}(\bC^2) + \Sym^{4n-7}(\bC^2) + \cdots + \Sym^1(\bC^2) ~+~ \Sym^{2n-1}(\bC^2) \otimes \beta. \notag
\end{align}
Therefore, the L-function of the numerator $L(M \otimes N, s + ~1/2)$ has a pole of order $2n-1$ at $s = 0$ if $\beta \neq 1$, and a pole of order $2n$ at $s = 0$ if $\beta = 1$. The same is true for the L-function of the denominator $L(\Sym^2 M, s + 1) \cdot L(\wedge^2 N, s+1)$, provided that $n > 1$. Hence, when $n > 1$, the ratio $L(M,N,s)$ is regular and non-zero at the point $s = 0$. In the special case where $n = 1$ and $\beta = 1$ (so $\SO(W) \cong \mathbb G_m$), this ratio has a simple pole at $s = 0$ as we noted already at the end of the previous example.

\vskip 10pt
\begin{example} $\SO_3 \subset  \SO_4$. \end{example}
We consider the case of $\SO_3 \subset \SO_4$  with $\SO_4$ quasi-split.  
This corresponds to   

$$\PGL_2(k) \longrightarrow  \GL_2^+(K) /k^\times,$$    
where  $K$ is a quadratic algebra over  a $p$-adic field  $k$ and $\GL_2^+(K)$ is the subgroup of $\GL_2(K)$ with determinant in 
$k^\times$. We can in fact consider slightly more generally
the embedding $\GL_2(k) \hookrightarrow \GL_2^+(K)$. Thus, for irreducible representations $\pi_1$ of $\GL_2^+(K)$ and $\pi_2$ of $\GL_2(k)$,
we would like to compute $d(\pi_1, \pi_2) = \dim \Hom_{\GL_2(k)}(\pi_1, \pi_2)$. 

\vskip 5pt
We shall consider the case when $\pi_1$ is a generic (thus infinite-dimensional) representation. 
In this case, the 4-dimensional Langlands parameter associated to $\pi_1$ is more usually called the Asai representation, or tensor induction,
denoted as $As(\pi_1)$. This notion is typically defined for representations of $\GL_2(K)$, but  one can also define it for representations $\pi_1$ of $\GL_2^+(K)$ by choosing an arbitrary irreducible representation of  $\GL_2(K)$ containing $\pi_1$ and 
doing the Asai construction for this. The result is  independent of the choice of the representation of $\GL_2(K)$, so it is legitimate to 
call it $As(\pi_1)$ for $\pi_1$ on $\GL_2^+(K)$. 
\vskip 5pt

If the representation $\pi_2$ of $\GL_2(k)$ is infinite dimensional, this amounts to questions already 
considered in \cite{P}, which is in the setting of tempered GGP. So we only analyze the situation when the representation $\pi_2$ of $\GL_2(k)$ is 
one dimensional, which to simplify notation, we take to be the trivial representation.  
\vskip 5pt
 
Thus, we want to understand which generic representations $\pi_1$ of $\GL_2^+(K)$ have an invariant 
form for  $\GL_2(k)$. By our conjecture,  one would therefore expect that a representation of
$\GL_2^+(K)$
has a $\GL_2(k)$ invariant form    if and only if the 4-dimensional representation of $W'_k$ contains the trivial representation, 
or equivalently  the Asai L-function has a  pole at 0.    

\vskip 5pt

This is exactly what was proved by
Harder-Langlands-Rapoport globally \cite{HLR},    
and later also locally by others, see \cite{Fl}, \cite{Ka} for results for $\GL_2$, and Lemma 4.1 of \cite{AP} for $\SL_2$.     The nice fact is that this criterion ---which is usually stated for discrete series representations--- 
is valid for all generic  representations $\pi_1$ of $\GL_2^+(K)$.    
 \vskip 5pt
 
\begin{thm} 
For $K$ a quadratic extension of a $p$-adic field $k$, 
an infinite dimensional  irreducible representation $\pi$ of $\GL_2(K)$ whose central character 
is trivial when restricted to $k^\times$ has a nonzero  $\GL_2(k)$ invariant linear form if and only if its Asai L-function has a pole at $s=0$.
\vskip 5pt

Assuming that the representation $\pi$ of $\GL_2(K)$ has a $\GL_2(k)$ invariant form, an irreducible sub-representation
$\pi'$ of $\pi$ restricted to $\GL^+_2(K)$ has a $\GL_2(k)$ invariant linear form if and only if $\pi'$ is generic for a character
$\psi: K \rightarrow \C^\times$ which is trivial on $k$. \end{thm}

 In this case, where the A-parameter of $\pi_2$ is $[2]$, the distinguished character predicted by our  local conjecture is the trivial character, which indexes  the representation 
of $\GL_2^+(K)$ in the L-packet  with Whittaker model with respect to a character of $K/k$. This is indeed what the previous theorem proves.
Our may also consider the case when $\GL_2(k)$ is replaced by $D^{\times}$, where $D$ is the unique quaternion $k$-algebra. In this case, the restriction problem is addressed by the tempered GGP. 
\medskip

\begin{example} $\U_2 \subset \U_3$. \end{example}
  
  For a quadratic extension $K/k$ of local fields, we consider a pair of unitary group $\U_2 \subset \U_3$ and
  the restriction of representations from $\U_3$ to $\U_2$. In what follows, we will be  interested in
understanding  which representations of $\U_3$ contain the trivial character of $\U_2$ as a quotient; such representations of $U_3$ are said to be $\U_2$-distinguished. This has been considered by Gelbart-Rogawski-Soudry, culminating in their work \cite{GRS}.
We recall their results briefly, casting them in the language of  our local and global conjectures. 

\vskip 5pt

\begin{thm} Let $\U_2 \subset \U_3$ be a pair of quasi-split unitary groups defined for a quadratic extension $K/k$ of non-archimedean local fields.
Suppose that $M_A$ is an A-parameter for $\U_3$ (regarded as a conjugate orthogonal representation of $W(K)$) with associated A-packet $\Pi_{M_A}$. 
A representation $\pi \in \Pi_{M_A}$ is $U_2$-distinguished if and only if one of the following holds:
\begin{itemize}
\item $\pi$ is the trivial representation of $\U_3$, i.e. $M_A = [3]$;
\item $\pi$ is a  tempered generic representation such that  $L(s,BC(\pi))$ has a pole at $s=0$ (where $BC(\pi)$ is the base change of $\pi$ to $K$).
In this case, $M_A =  1+M_0$ with $M_0$ tempered. 
 \end{itemize}
In particular, the pair $(M_A, [2])$ is relevant and the unique representation $\pi$ which is $\U_2$-distinguished lies in $\Pi_M$ and corresponds to the distinguished character of Conjecture \ref{conj-class}(c). 
\end{thm}

Here is the corresponding global theorem from \cite{GRS}.

\begin{thm} Let $E/F$ be a quadratic extension of number fields with corresponding adele rings $\A_E$ and $\A_F$
and consider a pair $\U_2 \subset \U_3$ of unitary groups over $F$. 
 Let $\pi = \otimes\pi_{v}$ be an irreducible infinite-dimensional cuspidal automorphic
  representation of $\U_3(\A_F)$ with base change
   $BC(\pi)$ on $\GL_3(E)$. 
  Then  the period integral over $\U_2$  is nonzero on $\pi$  if and only if
\begin{enumerate}
\item for each place $v$ of $F$, $\pi_v$ is $\U_2(F_v)$-distinguished;
\item $L(s,BC(\pi))$ has a pole at $s=1$.   
\end{enumerate}
In particular, $\pi$ belongs to a tempered L-packet with L-parameter $M = 1 + M_0$, so that $(M, N_A)$ is relevant  (with $N_A = [2]$). Moreover, the L-function
\[  L(M, N, s)   = \frac{L_E( M, s) \cdot L_E(M,s+1)}{ \zeta_F(s) \cdot \zeta_F(s+1) \cdot \zeta_F(s+2) \cdot L( \omega_{E/F}, s+1) \cdot L( M, {\rm} Ad, s+1)}  \]
is finite nonzero at $s=0$. 
\end{thm}
\vskip 5pt

 One may also start with a  non-tempered A-parameter $M_A$  of $\U_3$ and 
 consider the restriction of the representations in the associated A-packets to $\U_2$. 
 Disregarding 1-dimensional representations of $\U_3$, these A-parameters are of the form $M_A = \mu + \nu \boxtimes [2]$, with $\mu$, $\nu$ 1-dimensional characters which are conjugate dual of appropriate sign.  This restriction problem  has been studied in two papers \cite{H1, H2}
 of Jaeho Haan. His results are in accordance with the expectations of Conjecture \ref{conj-class}. We refer the reader to his papers for the precise results. 
\vskip 5pt
 \begin{example}  $\SO_4 \subset \SO_5$.
 \end{example}
 
 We now consider the restriction problem for some non-tempered A-packets of $\SO_5 = \PGSp_4$, namely the Saito-Kurokawa packets and the A-packets of Soudry type. The restriction of representations in these packets to $\SO_4$ has been determined in \cite{GG} and \cite{GS}.
 \vskip 5pt
 
 Consider first the Saito-Kurokawa case over a non-archimedean local field $k$. Here the A-parameter $M_A$ of $\SO_5$ is of the form
 \[  M_A =    M_0 +  [2], \]
 where $M_0$ is a tempered L-parameter of $\SO_3$. The A-packet of $M_A$ can be explicated as follows. The component group of $M_A$ is
 \[ A_{M_A} =  \begin{cases}
 \Z/2\Z \times \Z/2\Z \text{  if $M_0$ irreducible;} \\
  \{1 \} \times \Z/2\Z \,\,\text{~~ if $M_0$ reducible.} 
  \end{cases} \]
  Here, we think of the first component in the direct product as associated to $M_0$, whereas the second $\Z/2\Z$ is associated to $[2]$.
 The A-packet $\Pi_{M_A}$ thus consists of $4$ or $2$ representations in the respective cases and we may label them by the irreducible characters of $A_{M_A}$ as:
 \[ \Pi_{M_A} = \{ \pi^{\epsilon_1, \epsilon_2}: \epsilon_1, \epsilon_2 = \pm \} \]
 where we understand that when $M_0$ is reducible, $\pi^{-, \epsilon} = 0$.  Moreover, a representation $\pi^{\epsilon_1, \epsilon_2}$ is a representation of the split group $\SO_5^+$ if and only if $\epsilon_1 = \epsilon_2$; otherwise it is a representation of the non-split inner form $\SO_5^-$.   The associated L-packet $\Pi_M$ is the subset
 \[ \Pi_M = \{ \pi^{\epsilon_1, +}: \epsilon_1 = \pm \} \subset \Pi_{M_A}. \] 
 
 \vskip 10pt

 The A-parameters $N_A$ of $\SO_4$  for which $(M_A, N_A)$ is relevant are those of the form:
 \vskip 5pt
 
 \begin{itemize}
 \item[(a)] $N_A = N_0 + 1$, where $N_0$ is a 3-dimensional orthogonal L-parameter.
 \item[(b)] $N_A = \chi_0 + [3]$, where $\chi_0$ is a 1-dimensional character.
 \end{itemize}
 For simplicity, we assume that the $\SO_4$ in question has trivial discriminant, in which case $N_0$ and $\chi_0$ have trivial determinant. 
 Thus, in (b), $N_A$ is the A-parameter of the trivial representation. As this case will be treated in some generality in the last example below, we shall focus on (a). 
 \vskip 5pt

  In this case, the results of \cite{GG} can be summarized as:
 \vskip 5pt
 
 \begin{thm}
 Fix a Saito-Kurokawa A-parameter $M_A = M_0 + [2]$ of $\SO_5$ and let $N_A$ be a tempered A-parameter of $\SO_4$ with $\det(N_A)$ trivial.
 For $\pi^{\epsilon_1, \epsilon_2} \in \Pi_{M_A}$, set 
 \[ d( \pi^{\epsilon_1, \epsilon_2}, N_A) =  \sum_{\sigma \in \Pi_{N_A}} d(\pi^{\epsilon_1, \epsilon_2}, \sigma). \]
 Also, set
 \[ d(M_A, N_A) =  \sum_{\epsilon_1, \epsilon_2} d(\pi^{\epsilon_1, \epsilon_2}, N_A). \]
 Then
 \[ d(M_A, N_A) \ne 0 \Longleftrightarrow  \text{$(M_A, N_A)$ is relevant.} \]
 When $(M_A, N_A)$ is relevant, with $N_A = N_0 +1$ tempered, we have:
 \[ d( \pi^{\epsilon_1, \epsilon_2}, N_A)  \ne 0 \Longleftrightarrow    \epsilon_1 = \epsilon(1/2, M_0 \otimes N_A) = \epsilon(1/2, M_0) \cdot \epsilon(1/2, M_0 \otimes N_0). \] 
 Further, if $ d( \pi^{\epsilon_1, \epsilon_2}, N_A)  \ne 0$,
 then $ d( \pi^{\epsilon_1, \epsilon_2}, N_A)  =1$.
 \end{thm}
 \vskip 5pt
 
 \begin{cor}
In the context of the above theorem,  if one considers the L-packet $\Pi_M$, then we have:
\[ d(M,N) = \begin{cases} 
0, \text{  if $(M_A, N_A)$ is not relevant;} \\
1, \text{  if $(M_A,N_A)$ is relevant. }
\end{cases} \]
When $(M_A,N_A)$ is relevant, the unique representation $\pi^{\epsilon_1, +} \in\Pi_M$ for which $d(\pi^{\epsilon_1, +} \, N) \ne 0$ is
given by
\[ \epsilon_1 =   \epsilon(1/2, M_0 \otimes N_A). \]
This is the distinguished character predicted by Conjecture \ref{conj-class}.
 \end{cor}
 \vskip 10pt

 Observe that the above theorem goes beyond our local Conjecture \ref{conj-class} as it addresses the restriction problem for the whole A-packet $\Pi_{M_A}$ and not just the L-packet $\Pi_M$. Note also that when $M_0$ is irreducible,  one has $d(M_A, N_A) = 2$.
 \vskip 10pt
 
There is also an analogous global theorem which we state here.
\vskip 10pt

\begin{thm}
Consider the non-tempered global A-parameter $M_A = M_0 \oplus [2]$  over a number field $F$, where $ M_0$ corresponds to a cuspidal representation of $\PGL_2(\A_F)$. Consider an element of the global A-packet $\Pi_{M_A}$ for $G(\A_F)=\SO_5(\A_F)$:
\[   \pi^{\underline{\epsilon}_1, \underline{\epsilon}_2} = \otimes_v \pi^{\epsilon_{1,v},  \epsilon_{2,v}}. \]
\vskip 5pt

(i)  This representation occurs in the automorphic discrete spectrum if and only if 
\[ \prod_v \epsilon_{1,v} = \epsilon(1/2, M_0)  =  \prod_v \epsilon_{2,v}. \]
\vskip 5pt

(ii) When the condition in (i) holds, the global period integral of $\pi^{\underline{\epsilon}_1, \underline{\epsilon}_2}$ against the cuspidal representations with tempered L-parameter $N_A$ is zero unless $(M_A, N_A)$ is relevant, i.e. when $N_A = N_0 + 1$. 
 \vskip 5pt
 
 (iii) When the condition in (i) holds and  $(M_A, N_A)$  is relevant, the period integral in (ii) is nonzero if and only if
 \[ L(1/2, M_0 \otimes N_0) \ne 0. \]
\end{thm}
Noting that 
\[ L(s, M_A , N_A) =   \frac{L(s+ 1/2,  M_0 \otimes N_0)}{\zeta_F(s+ 2) \cdot L(s +3/2, M_0) \cdot L(s+1, M_0, Ad)}, \]
we see that the condition in (iii) is equivalent to saying that $L(0, M_A, N_A) \ne 0$. 
\vskip 10pt

Now we consider the case of A-packets of Soudry type, i.e an A-parameter of $\SO_5$ of the form
\[ M_A = M_0 \otimes [2] \]
where $M_0$ is an  orthogonal L-parameter.  Over a local field $k$, we enumerate the possible  component group $A_{M_A}$ according to the following 4 disjoint scenarios:
\begin{itemize}
\item[(a)]  If $M_0 = \chi + \chi^{-1}$ with $\chi$ unitary but non-quadratic, then $\det(M_0)$ is trivial and $A_{M_A}$ is trivial;
\item[(b)]  If $M_0 =\chi + \chi$, with $\chi$ quadratic, then $\det(M_0)$ is trivial and $A_{M_A} = \Z/2\Z$;
\item[(c)]  If $M_0$ is irreducible, then $\det(M_0)$ is nontrivial and $A_{M_A} = \Z/2\Z$;
\item[(d)]  If $M_0 = \chi_1 + \chi_2$ with $\chi_1 \ne \chi_2$ quadratic, then $\det(M_0) = \Z/2\Z \times \Z/2\Z$. 
\end{itemize}
Accordingly, the local A-packet $\Pi_{M_A}$ has $1$, $2$, $2$ or $4$ representations in the respective cases. 
We may denote the A-packet by
\[ \Pi_{M_A} = \{ \pi^{\epsilon_1, \epsilon_2}: \epsilon_1, \epsilon_2 = \pm \} \]
with the understanding that:
\vskip 5pt

\begin{itemize}
\item[-]  in case (a), all representations are $0$ except for $\pi^{++}$;
\item[-]  in case (b),  the representations $\pi^{+,-}$ and $\pi^{-,+}$ are zero;
\item[-]  in case (c),  the representations $\pi^{-,-}$ and $\pi^{+, -}$ are zero.
\end{itemize}
Moreover, the representation $\pi^{\epsilon_1, \epsilon_2}$ is a representation of the split $\SO^+_5$ if and only if $\epsilon_1 = \epsilon_2$. 
Thus the A-packet $\Pi_{M_A}$ contains some representations of the non-split $\SO^-_5$ only in cases (c) and (d) (i.e. when $\det(M_0)$ is nontrivial). The associated L-packet $\Pi_M$ is the subset
\[ \Pi_M  = \{ \pi^{+,+} \} \subset \Pi_{M_A}. \]
\vskip 5pt

The A-parameters $N_A$ of $\SO_4$ for which the pair $(M_A, N_A)$ is relevant are those of the form
\begin{itemize}
\item [(a)]$N_A = M_0 + N_0$ with $N_0$ tempered orthogonal of dimension $2$;
\item[(b)]  $N_A = \chi_1 \otimes [3] + \chi_2$, if $M_0 = \chi_1 + \chi_2$ ($\chi_1$, $\chi_2$ not necessarily distinct).
\end{itemize}
Here is a sample of the local results of \cite{GS}.
\vskip 5pt

\begin{thm}
Consider the non-tempered local A-parameter $M_A = M_0 \boxtimes [2]$ for $\SO_5$ and let $N_A$ be a tempered A-parameter for $\SO_4$.
Assume for simplicity that $\det(N_A)$ is trivial. 
 \vskip 5pt

(i) $d(M_A, N_A) \ne 0 \Longrightarrow (M_A,N_A)$ is relevant.
\vskip 5pt

(ii) Suppose that  $(M_A,N_A)$ is relevant and $\det(M_0)$ is trivial (i.e. in cases (a) and (b)). Then $d(\pi, N_A) =1$ for any $\pi \in \Pi_{M_A}$. 
Thus
\[ d(M_A, N_A) =  
\begin{cases}
1, \text{    if $M_0 = \chi + \chi^{-1}$, $\chi$ non-quadratic;  } \\
2, \text{    if $M_0 = \chi + \chi^{-1}$, with $\chi$ quadratic.}
\end{cases}
\]
\vskip 5pt

(iii) Suppose that  $(M_A,N_A)$ is relevant and  $\det(M_0)$ is nontrivial (i.e. in cases (c) and (d)). Then we have:
\vskip 5pt

\begin{itemize}
\item in case (c),  where $M_0$ is irreducible, we have $d(\pi, N_A) =1$ for any $\pi \in \Pi_{M_A}$, so that $d(M_A, N_A) =2$. 
\vskip 5pt

\item in case (d), where $M_0$ is reducible, we have $d(\pi, N_A) =1$ for any $\pi \in \Pi_{M_A}$, so that $d(M_A, N_A) =4$, unless
$N_A = M_0 + N_0$ with $N_0$ reducible. In this latter case, we have:
\[ d(\pi^{\epsilon_1, \epsilon_2}, N_A) = 1 \Longleftrightarrow \epsilon_1 = +,  \]
so that $d(M_A, N_A) =2$. 
\end{itemize}
\vskip 5pt

(iv) In all cases where $(M_A, N_A)$ is relevant, we have
 \[ d(\pi^{++}, N) = 1. \]
\end{thm}
This theorem verifies Conjecture \ref{conj-class} in this particular case; indeed it goes further by determining $d(\pi, N_A)$ for all $\pi \in \Pi_{M_A}$. 
There is also a companion global theorem for which we refer the reader to \cite{GS}.

\vskip 10pt
\begin{example}  $\Sp_4(k) \subset \GL_4(k)$. \end{example}
Representations of $\GL_{2n}(k)$ distinguished by $\Sp_{2n}(k)$
are classified by the work of Offen and Sayag, cf. \cite{OS}. A low rank  case of their work, $\Sp_4(k) \subset \GL_4(k)$,
 is closely related to the pair  $\SO_5(k) \subset \SO_6(k)$ and  can be used to verify our conjectures which we do
now.
\vskip 5pt

In the relationship of $\GL_4(k)$ with $\SO_6(k)$, if the L-parameter of a representation $\pi$ 
of $\GL_4(k)$
is $\phi$, then the corresponding parameter of the representation of $\SO_6(k)$ with values in
$\SO_6(\C)$ is $M = \det^{-1/2}(\phi) \Lambda^2(\phi)$ where $\det^{1/2} (\phi)
$  is a square root
of $\det(\phi)$ (which must exist for a representation $\pi$  of $\GL_4(k)$ to be related to a representation of $\SO_6(k)$). 
\vskip 5pt

Since the A-parameter of the trivial representation of $\SO_5(k)$ is $[4] = \Sym^3(\C^2)$, 
 it follows by our conjecture that  the only option for the 6-dimensional orthogonal representation 
$M= \det^{-1/2} (\phi) \Lambda^2(\phi)$ are:

\begin{enumerate}
\item $[5] + [1]$;
\item $[3] + M_0$ with $M_0$ tempered  and 3-dimensional.
\end{enumerate}
Here case $(1)$ corresponds to characters of $\GL_4(k)$, whereas in case $(2)$, the  $M_0$ must be orthogonal. It can be seen that
  $M = \det^{-1/2}(\phi) \Lambda^2(\phi)$ has this shape for a 
  4-dimensional L-parameter $\phi$ 
  of Arthur type  with  A-parameter $\phi_A$ if and only if $\phi_A$ is of the form $\tau \otimes [2]$, where $\tau$ is
  a 2-dimensional tempered parameter.
  Thus, $M$ is of Arthur type with associated A-parameter
  \[ M_A =   \Lambda^2(\tau \otimes [2]) = {\rm Sym}^2(\tau) \oplus \Lambda^2(\tau) [3]. \]
Thus our conjecture is in conformity with the results of Offen-Sayag that the only representations of $\GL_4(k)$ which are distinguished 
by $\Sp_4(k)$ are either one dimensional, or have  A-parameter of the form $\tau \otimes [2]$. 
 \vskip 10pt

 \begin{example}
Distinction of $\SO_n$ by $\SO_{n-1}$.
 \end{example}
 We have considered several low rank instances of the problem of determining the representations (of Arthur type) of $\SO_n$ (resp. $\U_n$) distinguished by $\SO_{n-1}$ (resp. $\U_{n-1}$). This problem can be studied in general using the theory of theta correspondence. 
 More precisely, one can show:
 \vskip 5pt
 
 \begin{prop} 
 Let $k$ be a non-archimedean local field and consider a  non-degenerate quadratic space $V$ over $k$.
 For $a \in k$, set $V_a = \{ x \in V: q(x) =a \}$. 
 Let
 \[ G = \begin{cases} 
 \SL_2(k) \text{   if $\dim V$ is even;} \\
 {\rm Mp}_2(k), \text{(the two fold metaplectic cover of $\SL_2(k)$),    if $\dim V$ is odd.}
 \end{cases}\]
  For a fixed nontrivial additive character $\psi$ of $k$, one may consider the $\psi$-theta correspondence for $G \times \O(V)$. 
  Given an irreducible  representation $\pi$ of $\O(V)$, let $\Theta_{\psi}(\pi)$ be the big theta lift of $\pi$ to $G(k)$. 
Then for $a \in k^{\times}$, 
\[  \Theta_{\psi}(\pi)_{N,\psi_a}  \cong \begin{cases}
0, \text{     if $V_a(k)$ is empty;} \\
\Hom_{\O(U_a)} (\pi, \C), \text{   if $V_a(k)$ is nonempty;} \end{cases} \]
where in the latter case,  $U_a = x_a^{\perp}$ for some $x_a \in V_a(k)$. Here, $N$ is a maximal unipotent subgroup of $G$ and $\psi_a(x) = \psi(ax)$ is regarded as a generic character of $N(k) = k$. 
\end{prop}

\vskip 5pt

 \begin{cor} \label{distinction} Let $\pi$ be an irreducible admissible representation of $\O(V)$.
 Then $\pi$ is distinguished by $\O(W)$ with $V/W = \langle a \rangle$ if and only if the big $\psi$-theta lift of $\pi$ to $G$ is $\psi_a$-generic. In particular, one has:
 \vskip 5pt
 
 (i) If $\pi$ is distinguished by $\O(W)$ where $W$ is a codimension one nondegenerate
   subspace of $V$,   then $\pi$ has a nonzero theta lift to  $G$.
   \vskip 5pt
   
   (ii) If $\pi$ is the $\psi$-theta lift of a $\psi_a$-generic representation of $G$, then $\pi$ is
   distinguished by $\O(W)$ for a codimension one nondegenerate
   subspace $W$ of $V$ with $V/W = \langle a\rangle$.
   \end{cor}
 
 When $\dim V \geq 3$, as is well-known, the theta correspondence for $G \times \O(V)$ reduces to one 
for  $G \times \SO(V)$.  If $\sigma \in {\rm Irr}(G)$ has tempered A-parameter $M_0$, one knows precisely how to describe its small theta lift 
$\theta_{\psi}(\sigma)$ to $\SO(V)$ in terms of the local Langlands correspondence  (see \cite{Mu1, Mu2, Mu3} and \cite{AG}). From this description, 
one sees that  $\theta_{\psi}(\sigma)$ is of Arthur type with A-parameter of the form 
\[ \begin{cases}
M_0 +  [2n-3], \text{ if $\dim V = 2n$;} \\
M_0+ [2n-2], \text{ if $\dim V = 2n+1$.} \end{cases} \]
Moreover, one knows from results of M{\oe}glin \cite{M3, GI} that all representations with such an A-parameter are obtained as theta lifts from the L-packet of $G$ with L-parameter $M_0$.  From this and Corollary \ref{distinction}, we see that:

\begin{thm} Let $\pi$ be an irreducible admissible representation of $\SO(V)$ of Arthur type with A-parameter $M_A$ and suppose that
 $\pi$ is distinguished by $\SO(W)$ for some codimension one nondegenerate
   subspace $W$ of $V$.  Then  $M_A$ must have one of the following form:

\begin{itemize}
\item if $\dim V = 2n+1$ is odd, then $M_A = [2n]$ (which corresponding to the trivial representation) or $M_A = [2n-2] + M_0$, with $M_0$ tempered 2-dimensional (which corresponds to theta lifts of tempered representations of ${\rm Mp}_2(k)$);
\vskip 5pt

\item if $\dim V = 2n$ is even, then $M_A = [2n-1] + \chi$ (which corresponds to the trivial representation) or $M_A = [2n-3] + M_0$ with $M_0$ tempered 3-dimensional (which corresponds to theta lifts of tempered representations of $\SL_2(k)$).
\end{itemize}
Conversely, if $M_A$ is an A-parameter of $\SO(V)$ of the above form, then any representation of $\SO(V)$ in the associated A-packet $\Pi_{M_A}$ is distinguished by $\SO(W)$ for some codimension one nondegenerate subspace $W \subset V$. 
\end{thm}

The theorem above  is in accordance with Conjecture \ref{conj-class}. There is of course an entirely analogous theorem (and proof!)
in the unitary case, with the quasi-split $\U_2$ playing the role of $G$.

  \vskip 10pt

\section{Automorphic Descent}   \label{S:descent}
  In a series of papers \cite{GRS1, GRS2, GRS3, GRS4},  Ginzburg-Rallis-Soudry pioneered the automorphic descent or backward lifting technique which allowed them to construct the backward functorial lifting from general linear groups to classical groups.  The input of their construction is an discrete tempered global L-parameter $M$ for the classical group in question, thought of as an isobaric automorphic representation of the appropriate $\GL_N$, and the output is the unique globally generic cuspidal representation in  global L-packet determined by $M$.
  \vskip 5pt
  
   In a recent paper \cite{JZ}, D. H. Jiang and L. Zhang extended this construction so that it has the potential to construct all the automorphic members of the global L-packet associated to $M$. By their very construction,  the automorphic descent method is closely connected to the conjectures of \cite{GGP} since it involves the consideration of Bessel and Fourier-Jacobi models. In this section, we explain how these automorphic descent constructions fit with our conjectures.
  \vskip 5pt
  
  Let us consider the example of $\SO_{2n+1}$ and give a brief description of the descent construction. Suppose we are given a tempered L-parameter $M = M_1 \oplus ...\oplus M_r$ of $\SO_{2n+1}$, with $M_i$ distinct cuspidal representations of some $\GL$'s of symplectic type. We may regard $M$ as an L-parameter of $\GL_{2n}$, giving rise to an automorphic representation
  \[  \pi_M = \pi_{M_1}  \times ...\times \pi_{M_r}  \]
  of $\GL_{2n}$. Now consider the non-tempered A-parameter
  \[  \tilde{M}_A =  M \boxtimes [2]. \]
  It is an A-parameter of $\SO_{4n}$ whose associated L-parameter is
  \[  \tilde{M} = M \cdot |-|^{1/2}  \oplus M \cdot |-|^{-1/2}. \]
  One sees that the local and global component groups of $\tilde{M}$ are trivial, so that 
  the local L-packets and global L-packet of $\SO_{4n}$ associated to the L-parameter $\tilde{M}$ are singleton sets.  This unique global representation is  the Langlands quotient  $J(\pi_M, 1/2)$ of the standard module
  induced from the representation $\pi_M |-|^{1/2}$ on the Siegel parabolic subgroup of $\SO_{4n}$. The embedding of $J(\pi_M, 1/2)$ into the automorphic discrete spectrum can be constructed as an iterated residue of the corresponding Eisenstein series. 
 \vskip 5pt
 
 Consider now the Bessel coefficient of $J(\pi_M, 1/2)$ with respect to $\SO_{2n+1}$. This gives rise to an automorphic representation of $\SO_{2n+1}$ which Ginzburg-Rallis-Soudry showed to be an irreducible globally generic cuspidal representation belonging to the L-packet of $M$. In fact, since it is globally generic, it is the cuspidal representation $\sigma_M$ whose local components correspond to the trivial character of the local component group at each place.
 \vskip 5pt

 To explain this construction from the viewpoint of our conjectures in this paper,  observe that we are considering the pair $\SO_{4n} \times \SO_{2n+1}$ and we have a non-tempered A-parameter $\tilde{M}_A$ of $\SO_{4n}$. One now reasons as follows:
 \vskip 5pt
 
 \begin{itemize}
 \item  For which A-parameter $N_A$ of $\SO_{2n+1}$ can the pair $(\tilde{M}_A, N_A)$ be relevant? It is clear that the only possibility is $N_A = M$. Hence, by our global conjecture \ref{C:global}(1),  the Bessel descent  of $J(\pi_M, 1/2)$ down to $\SO_{2n+1}$ (which can be shown to be cuspidal) can only contain representations in the L-packet of $M$. At this point, however, it  is  still possible that this Bessel descent is  $0$.          \vskip 5pt
    
    \item Is there any reason to hope that the Bessel descent is nonzero? To address this, we first consider the local setting. 
 Since $J(\pi_M,1/2)$ belongs to the L-packet $\Pi_{\tilde{M}}$ which is a singleton (both locally and globally), our local conjecture \ref{conj-class} can be applied. 
One can compute the distinguished character   of the local component group $A_{\tilde{M}_v} \times A_M$ given by Conjecture \ref{conj-class}(c); in this case it turns out to be the trivial character. 
Hence, our local conjecture  \ref{conj-class}(c)  implies that at each place $v$, the multiplicity 
 \[  d(J(\pi_{M,v}, 1/2),  \sigma_{M_v})  =1. \]
 Because of the uniqueness part in Conjecture \ref{conj-class}(b), we see that $\sigma_M$ is the only element $\sigma$  in the global L-packet of $M$ such that $J(\pi_M, 1/2) \otimes \sigma$   supports an abstract Bessel period.  
 \vskip 5pt
 
\item  Finally, is the global Bessel period integral nonzero on the automorphic representation $J(\pi_M, 1/2) \otimes \sigma_M$? We can appeal to our global conjecture \ref{C:global}. The above shows that conditions (1) and (2) in Conjecture \ref{C:global} already hold, so it remains to verify the condition (3). 
But  the function $L(\tilde{M}, M, s)$ is given by
 \[    \frac{L(s+1 , M \times M) \cdot  L(s, M \times M)}{L(s+1, \Sym^2 M) \cdot  L(s+1, \Sym^2 M) \cdot L(s, \wedge^2 M) \cdot L(s+1, \wedge^2 M) \cdot L(s+2, \wedge^2 M)}. \]
 At $s= 0 $,  this is holomorphic and nonzero (alternatively, one can appeal directly to Theorem \ref{T:globalR}) . Hence, our global conjecture implies that the global Bessel period integral is nonzero on $J(\pi_M, 1/2) \times \sigma_M$. In particular, the Bessel descent of $J(\pi_M, 1/2)$ to $\SO_{2n+1}$ is nonzero and equal to $\sigma_M$.
 \end{itemize}
 \vskip 10pt
 
 We mention that there is a local analog of the above construction: the local descent map. This was considered in, for example, the papers \cite{GRS1, GRS3, GRS4, ST}.  While the local descent can be defined starting from any L-parameter of a classical group, the results obtained in these works are most complete in the following two situations:
 \vskip 5pt
 
 \begin{itemize}
 \item[(a)]  when the L-parameter is supercuspidal, i.e. a multiplicity-free sum of self dual representations of the Weil group (instead of the Weil-Deligne group)
 of the relevant sign. In this case, it was shown that the local descent is an irreducible generic supercuspidal representation of the classical group with the initially given L-parameter.  
 \vskip 5pt
 
 \item[(b)]  when the L-parameter involved is unramified. In this case, it was shown that the only unramified representation which could occur as a quotient in the local descent is the one with the given L-parameter. Such results in the unramified case is necessary for the proof of the global results described above.
 \end{itemize}
 Thus, (a) is in complete accordance with our local Conjecture \ref{conj-class} whereas (b) provides supporting evidence for the same conjecture. 
 
 \vskip 10pt
 
The twisted automorphic descent of Jiang-Zhang offers the possibility of constructing the other members of the global L-packet $M$ of $\SO_{2n+1}$ beyond the representation $\sigma_M$.  In the above construction, Ginzburg-Rallis-Soudry showed that the automorphic descent of $J(\pi_M, 1/2)$ is equal to $\sigma_M$  by showing that this descent is cuspidal and all of its irreducible summands are globally generic.  Suppose that $\sigma_{\eta}  = \otimes'_v \sigma_{\eta_v}$ is a cuspidal representation of the global L-packet of $M$. How can one construct $\sigma_{\eta}$ by an analogous process as above? 
\vskip 5pt

The idea of Jiang-Zhang is as follows.  The cuspidal representation $\sigma_{\eta}$ must possess some nonzero Bessel coefficient with respect to a smaller even orthogonal group $\SO_{2m}$. Suppose that the Bessel coefficient of $\sigma_{\eta}$ down to $\SO_{2m}$ contains a cuspidal representation $\tau_{\chi} = \otimes_v \tau_{\chi_v}$ belonging to a tempered A-parameter $N$ of $\SO_{2m}$.  By the uniqueness part of the global conjecture in \cite{GGP} (in the tempered case), one knows that $\sigma_{\eta}$ is characterized as the unique representation in the L-packet of  $M$  such that $\sigma_{\eta} \otimes \tau_{\chi}$ supports the relevant  global Bessel period for some $\tau_{\chi}$ in the L-packet of $N$. This means that we have a way of distinguishing  the representation 
$\sigma_{\eta}$  from the other members of the L-packet of $M$ (just as one uses the Whittaker-Fourier coefficient to detect the representation $\sigma_M$). If this is how one is hoping to detect the representation $\sigma_{\eta}$, then it is only reasonable that one should somehow involve the auxiliary $\tau_{\chi}$ in the descent construction of $\sigma_{\eta}$. 
 \vskip 5pt
 
 Here then is the construction of Jiang-Zhang.  We first pick a tempered L-parameter $N$ of a smaller group $\SO_{2m}$ such that the global Bessel period of $\sigma_{\eta} \otimes \tau_{\chi}$ is nonzero for some $\tau_{\chi}$ with L-parameter $N$. The existence of the data $(\SO_{2m}, N, \tau_{\chi})$ is not obvious, especially with the condition that $N$ is tempered. This is in fact taken as a hypothesis in \cite{JZ}. Hence the construction in \cite{JZ} is so far conditional, though one would like to think that this hypothesis is not unreasonable. 
 In any case, under this hypothesis, one considers the standard module 
   \[  I(\pi_M \otimes \tau, 1/2) =  \pi_M \cdot |-|^{1/2}  \rtimes   \tau_{\chi}  \]
 of $\SO_{4n+2m}$ defined by induction from the parabolic subgroup with Levi factor $\GL_{2n} \times \SO_{2m}$. Let $J(\pi_M \otimes \tau_{\chi}, 1/2)$ be the Langlands quotient of this standard module. It can be embedded into the automorphic discrete spectrum as an iterated residue of the corresponding Eisenstein series and one can then consider the  
Bessel descent  of $J(\pi_M \otimes \tau_{\chi}, 1/2)$ down to $\SO_{2n+1}$.   Jiang-Zhang verified that one gets a cuspidal automorphic representation of $\SO_{2n+1}$ all of whose irreducible components $\sigma'$ belong to the global L-packet of $M$. They then computed the Bessel period  of $\sigma'  \otimes \tau_{\chi}$ and showed that it is nonzero, thus showing that $\sigma' \cong \sigma_{\eta}$.  
  \vskip 5pt
  
  Again, we can understand this construction of Jiang-Zhang from the viewpoint of the conjectures of this paper. 
  \vskip 5pt
  
  \begin{itemize}
  \item The choice of $(\SO_{2m}, N, \tau_{\chi})$ so that $\sigma_{\eta} \otimes \tau_{\chi}$ has nonzero global Bessel period implies that 
   for each place $v$, the local character $\eta_v  \times \chi_v$ is the distinguished character of the local component group $A_{M_v} \times A_{N_v}$ from the local conjecture of \cite{GGP} for the pair $\SO_{2n+1} \times \SO_{2m}$.  
   \vskip 5pt
   
   \item The discrete automorphic representation $J(\pi_M \otimes \tau_{\chi}, 1/2)$ belongs to the A-packet of $\SO_{4n+2m}$ with non-tempered A-parameter
   \[  \tilde{M}_A  =  M \boxtimes [2]  \oplus N.   \]
  Indeed, it lies in the L-packet associated to the corresponding L-parameter 
  \[  \tilde{M} =  M \cdot |-|^{1/2}  \oplus M \cdot |-|^{-1/2} \oplus   N. \]
    One has a natural identification of the local component group
  \[       A_{\tilde{M}_v} \cong A_{N_v}  \]
  for each place $v$.   Under this identification, the character  of $A_{\tilde{M}_v}$ which corresponds to the representation 
  $J(\pi_M \otimes \tau_{\chi}, 1/2)$ is the character $\chi_v$ (which indexes the representation $\tau_{\chi_v}$ of $\SO_{2m}$). 
  \vskip 5pt
  
  \item Now if one considers the Bessel period in the context of $\SO_{4n+2m} \times \SO_{2n+1}$, one sees that the pair of A-parameters $(\tilde{M}_A, M)$ is relevant. Indeed, $M$ is the only A-parameter of $\SO_{2n+1}$ which can form a relevant pair with $\tilde{M}_A$. Hence, by our global conjecture \ref{C:global}, the only cuspidal representations of $\SO_{2n+1}$ which can potentially be contained in the Bessel descent of $J(\pi_M \otimes\tau_{\chi}, 1/2)$  are those with L-parameter $M$. 
  \vskip 5pt
  
  \item To determine  which cuspidal representations in the L-packet of $M$ are contained in the Bessel descent, we appeal to our local conjecture \ref{conj-class} for the relevant pair $(\tilde{M}_A, M)$ on $\SO_{4n+2m} \times \SO_{2n+1}$. The conjectures gives a distinguished representation with nonzero multiplicity in the local L-packet of $\tilde{M}_v \times M_v$ corresponding to a   distinguished character on the local component group
  \[  A_{\tilde{M}_v} \times   A_{M_v}  \cong A_{N_v} \times A_{M_v}. \]
 A short computation shows that this distinguished character of   $ A_{\tilde{M}_v} \times   A_{M_v}$ is equal to the distinguished character of $ A_{N_v} \times A_{M_v}$ under the above natural identification of component groups, and thus  is none other than $\chi_v \times \eta_v$. 
 Hence the global representation $J(\pi_M \otimes \tau_{\chi}, 1/2) \otimes \sigma_{\eta}$ is the unique representation in the global L-packet of $\tilde{M} \times  M$ which supports a nonzero abstract Bessel period. This implies that the only cuspidal representation that could be contained in the Bessel descent of $J(\pi_M \otimes \tau_{\chi}, 1/2)$ to $\SO_{2n+1}$ is $\sigma_{\eta}$.
 \vskip 5pt
 
 \item  It remains to show that the Bessel descent is nonzero, or equivalently that the global Bessel period of $J(\pi_M \otimes \tau_{\chi}, 1/2) \otimes \sigma_{\eta}$ is nonzero.
 For this, our global conjecture \ref{C:global} implies that it suffices to check the non-vanishing of $L(\tilde{M}, M, s)$ at $s=0$.  But Theorem \ref{T:globalR} implies that the desired non-vanishing. This shows that the cuspidal part of the Bessel descent is precisely $\sigma_{\eta}$, as desired.
 \end{itemize}
 This concludes our explanation of the automorphic descent method from the viewpoint of our conjectures. 
  \vskip 15pt

\section{L-functions: $\GL$ case}   \label{S:L-GL}
The purpose of this section is to give the proof of Theorem \ref{pole}.   For the convenience of the reader, we restate Theorem \ref{pole} here
  \vskip 5pt
  
  \begin{thm} \label{T:pole}
  Let $k$ be a non-archimedean local field and  let $(M_A, N_A)$ be a relevant pair of A-parameters for $\GL_m(k) \times \GL_n(k)$ with associated pair of L-parameters $(M,N)$.
Then the order of pole at $s=0$ of 
\[  L(M,N, s) = \frac{L(M \otimes N^{\vee}, s + ~1/2) \cdot L(M^{\vee} \otimes N, s + ~1/2)}{ L(M \otimes M^{\vee}, s + 1)\cdot  L(N \otimes N^{\vee}, s + 1)} \]
  is greater than or equal to zero.
\end{thm}

 \vskip 10pt

 Before we start making the calculations on 
the order of the pole of $L(s, M, N) $, 
we need to recall some generalities on L-functions.
For any irreducible representation $V\otimes \Sym^{i}(\C^2) \otimes \Sym^j(\C^2)$ of 
$WD(k) \times \SL_2(\C) = W(k) \times \SL_2(\C) \times \SL_2(\C)$, the L-function  $L(s,V\otimes \Sym^{i}(\C^2) \otimes \Sym^j(\C^2))$ 
has a pole at $s=1/2$ if and only if 
\begin{enumerate}
\item $V$ is the trivial representation of $W(k)$;
\vskip 5pt

\item $j>i$, $j\equiv (i+1) \bmod 2$.

\end{enumerate}
\vskip 5pt

Similarly,  the L-function  $L(s,V\otimes \Sym^{i}(\C^2) \otimes \Sym^j(\C^2))$ 
has a pole at $s=1$ if and only if 
\begin{enumerate}
\item $V$ is the trivial representation of $W(k)$;
\vskip 5pt

\item $j>i$, $j\equiv i \bmod 2$.

\end{enumerate}
\vskip 5pt
\noindent  In either of the two  cases, the pole is simple if it exists.
As a consequence of the  above, we note the following  lemma.

\begin{lemma} \label{L-f} 
For any representation $M$ of $WD(k) = W(k) \times \SL_2(\C)$,
  the order of the pole of $L(s, M \otimes \Sym^j(\C^2))$ at $s=1/2$ is the same as the order
of pole of $L(s,M \otimes \Sym^{j+1}(\C^2))$ at $s=1$, and these are the same as the corresponding orders of pole for $M$ replaced by $M^\vee$.
\end{lemma}

This lemma suggests the  introduction of a map on the set of representations of $ WD(k) \times \SL_2(\C)$ to itself which we
denote by $M \rightarrow M^+$,  and defined as follows.
For $M_A$ a representation of $WD(k) \times \SL_2(\C)$ of the form
$$M_A = \sum _{i\geq 0} M_i \otimes \Sym^i(\C^2),$$
define,
$$M_A^+ = \sum _{i\geq 0} M_i \otimes \Sym^{i+1}(\C^2).$$
From Lemma \ref{L-f}, the order of pole of $M_A$ at $s = 1/2$ is the same as the
order of pole of $M_A^+$ at $s=1$.

\vskip 10pt

 We are now ready to begin the proof of Theorem \ref{T:pole}
We write the parameters $M_A$ and $N_A$ as,
\begin{eqnarray*}
  M_A & = & A_1 + A_2 + \cdots + A_n + A_0, \\
  N_A & = & B_1 + B_2 + \cdots + B_n + B_0 ,
\end{eqnarray*}
where for $i \geq 1$, $A_i$ and $B_i$ are irreducible representations of $ WD(k) \times \SL_2(\mathbb C)$ of the form,
\begin{eqnarray*}
  A_i & = & M_i \boxtimes \Sym^{a_i-1}(\C^2), \\
    B_i & = & M_i \boxtimes \Sym^{b_i-1}(\C^2),
  \end{eqnarray*}
with $$ a_i-b_i = \pm 1,$$
and with $A_0,B_0$  tempered representations of $WD(k)$. Clearly, the following proposition proves the theorem.

\begin{prop} \label{lemma-pole}
  Let $C_1, C_2,D_1,D_2$ be irreducible representations of $WD(k) \times \SL_2(\C)$
    of the form:
  \begin{eqnarray*}
    C_i & =  & M_i \boxtimes \Sym^{a_i-1}(\C^2), \\
    D_i & = & M_i \boxtimes \Sym^{b_i-1}(\C^2),
  \end{eqnarray*}
for irreducible representations $\sigma_i$ of $WD(k)$   with $a_i-b_i = \pm 1$ for $i=1,2$. Then,
 \[L(s,C_1,C_2,D_1,D_2)= \frac{L(s+1/2, C_1\otimes D_2 + C_2 \otimes D_1 ) 
  }
  {L(s+1, C_1\otimes C_2)
     L(s+1, D_1\otimes  D_2)  } , \tag{a} \]
has a pole of order $\geq 0$ at $s=0$.

Further,
\[L(s,C_1,D_1)= \frac{L(s+1/2, C_1 \otimes D_1 )^2}  {L(s+1, C_1\otimes C_1)
     L(s+1, D_1\otimes  D_1)  } , \tag{b} \]
  has a pole of order $\geq 0$ at $s=0$.

  Finally, for any tempered representation $A_0$ of $WD(k)$,
\[L(s,C_1,D_1, A_0)= \frac{L(s+1/2, C_1 \otimes A_0 )}  {
     L(s+1, D_1\otimes  A_0)  } , \tag{c} \]
  has a pole of order $\geq 0$ at $s=0$.
  
\end{prop}
\begin{proof} We will only prove part (a) of the proposition, the other parts are proved analogously. 
It suffices to prove part (a) of the proposition in the following two cases.

  \begin{enumerate}
  \item $b_1 = a_1+1, b_2=a_2+1,$
  \item $b_1 = a_1-1, b_2=a_2+1.$
    \end{enumerate}
 We will deal with these two cases separately by a direct calculation, starting with case (1).

\vskip 10pt
Let
\[\begin{array} {ccccccc}
    C_1 &  = &  M_1 \boxtimes \Sym^i(\C^2),  &  & C_2  & =  & M_2 \boxtimes \Sym^j(\C^2), \\
    D_1 & =  & M_1 \boxtimes \Sym^{i+1}(\C^2),  &   & D_2 & =  & M_2 \boxtimes \Sym^{j+1}(\C^2).
  \end{array} \]
be irreducible representations of $WD(k) \times \SL_2(\C)$
for irreducible representations $M_1,M_2$ of $WD(k)$ where we assume without loss of generality that $i\geq j$.
It follows that
\begin{eqnarray*}
C_1\otimes D_2 + C_2 \otimes D_1    & =  & (M_1 \otimes M_2) \boxtimes \left [ \Sym^i(\C^2) \otimes \Sym^{j+1}(\C^2) + \Sym^j(\C^2)  \otimes \Sym^{i+1}(\C^2) \right ] \\
C_1\otimes C_2 + D_1 \otimes D_2  & = & (M_1 \otimes M_2) \boxtimes \left [ \Sym^i(\C^2) \otimes \Sym^j(\C^2) + \Sym^{i+1}(\C^2)   \otimes \Sym^{j+1}(\C^2) \right ].
\end{eqnarray*}
Hence,  $C_1\otimes D_2 + C_2 \otimes D_1 $ is equal to (assuming $i > j$),
$$(M_1 \otimes M_2) \boxtimes \left [ \Sym^{i+j+1}(\C^2) + \cdots + \Sym^{i-(j+1)}(\C^2)  + \Sym^{j+i+1}(\C^2) + \cdots +  \Sym^{i+1-j}(\C^2)  \right ]$$
and $C_1\otimes C_2 + D_1 \otimes D_2$ is equal to,
$$ (M_1 \otimes M_2) \boxtimes \left [ \Sym^{i+j}(\C^2) + \cdots +  \Sym^{i-j}(\C^2)  + \Sym^{i+j+2}(\C^2)  + \cdots +  \Sym^{i-j}(\C^2) \right ].$$

For any representation $M_A$ of $WD(k) \times \SL_2(\C)$, we have defined the representation $M_A^+$ earlier. 
With this notation, we have,
\[
  (C_1\otimes D_2 + C_2 \otimes D_1)^+ - ( C_1\otimes C_2 + D_1 \otimes D_2) \]
  \[ =
    (M_1 \otimes M_2) \boxtimes \left [ \Sym^{i+j+2}(\C^2) -  \Sym^{i- j}(\C^2) \right ].
\]
By Lemma \ref{L-f}, the order of pole of
  $$\frac{L(s+1/2, C_1\otimes D_2 + C_2 \otimes D_1 ) 
  }
  {L(s+1, C_1\otimes C_2)
     L(s+1, D_1\otimes  D_2)  } , $$
  at $s=0$ is the same as the order of pole of the virtual representation
  $$(C_1\otimes D_2 + C_2 \otimes D_1)^+ - ( C_1\otimes C_2 + D_1 \otimes D_2) ,$$
  of $WD(k) \times \SL_2(\C)$ at $s = 1$.
  \vskip 5pt
  
  On the other hand, for any irreducible representation
  $M$ of $WD(k)$, the order of pole at $s=1$ of the virtual representation 
  \[ M \boxtimes \left [ \Sym^{i+j+2}(\C^2) -  \Sym^{i- j}(\C^2) \right ] \]
  is $\geq 0$, since the order of pole at $s=1$ of the representation $M \boxtimes  \Sym^{a}(\C^2)$ is always greater than or equal to
the order of pole at $s=1$ of the representation $M \boxtimes  \Sym^{a-2b}(\C^2)$ for any integer $b \geq 0$. Hence 
so $L(s,C_1,C_2,D_1,D_2)$ has a pole of order $\geq 0$ at $s=0$,
  proving the proposition  in this case (assuming $i>j$ here).
\vskip 5pt

  If $i=j$, it can be seen that,
\[
  (C_1\otimes D_2 + C_2 \otimes D_1)^+ - ( C_1\otimes C_2 + D_1 \otimes D_2) \]
  \[ =    (M_1 \otimes M_2) \boxtimes \left [ \Sym^{2i+2}(\C^2) -  2\C \right ].
  \]
Since tempered representations have no poles in the region $Re(s)>0$, the term  $ 2\C$
contributes to no poles for the virtual representation $ M_1 \otimes M_2 \boxtimes \left [ \Sym^{2i+2}(\C^2) -  2\C \right ],$ leaving us with the
true representation $ M_1 \otimes M_2 \boxtimes \Sym^{2i+2}(\C^2)$ which can only contribute a non-negative number of poles at $s=1$,
completing the proof of the proposition for $i=j$ case of case (1).

\vskip 10pt

Next we consider case (2), where
\[\begin{array} {ccccccc}
    C_1 &  = &  M_1 \boxtimes \Sym^i(\C^2),  &  & C_2  & =  & M_2 \boxtimes \Sym^j(\C^2), \\
    D_1 & =  & M_1 \boxtimes \Sym^{i-1}(\C^2),  &   & D_2 & =  & M_2 \boxtimes \Sym^{j+1}(\C^2).
  \end{array} \]
for irreducible representations $M_1,M_2$ of $WD(k)$, and assume that $i \geq j$.
It follows that, 
\begin{eqnarray*} C_1\otimes D_2 + C_2 \otimes D_1 & = &  (M_1 \otimes M_2) \otimes \left [  \Sym^i(\C^2) \otimes \Sym^{j+1}(\C^2) + \Sym^j(\C^2)  \otimes \Sym^{i-1}(\C^2) \right ],\\
C_1\otimes C_2 + D_1 \otimes D_2 & = & (M_1 \otimes M_2) \otimes \left [ \Sym^i(\C^2) \otimes \Sym^j(\C^2) + \Sym^{i-1}(\C^2)   \otimes \Sym^{j+1}(\C^2) \right ].
\end{eqnarray*} 
\vskip 10pt

Assuming $i >j+1$, we deduce that  $C_1\otimes D_2 + C_2 \otimes D_1 $ is equal to,
$$  (M_1 \otimes M_2) \boxtimes \left [ \Sym^{i+j+1}(\C^2) + \cdots + \Sym^{i-(j+1)}(\C^2) + \Sym^{j+i-1}(\C^2) + \cdots +  \Sym^{i-1-j}(\C^2) \right ],$$
and $C_1\otimes C_2 + D_1 \otimes D_2$ is equal to,
$$ (M_1 \otimes M_2) \boxtimes \left [\Sym^{i+j}(\C^2) + \cdots +  \Sym^{i-j}(\C^2) + \Sym^{i+j}(\C^2)  + \cdots +  \Sym^{i-j-2}(\C^2) \right ].$$
Therefore,
\[
  (C_1\otimes D_2 + C_2 \otimes D_1)^+ - ( C_1\otimes C_2 + D_1 \otimes D_2) \]
  \[
   =    (M_1 \otimes M_2) \boxtimes \left [ \Sym^{i+j+2}(\C^2) -  \Sym^{i- j-2}(\C^2) \right ].
\]
\vskip 10pt

As in case (1),  the order of pole of
  $$\frac{L(s+1/2, C_1\otimes D_2 + C_2 \otimes D_1 ) 
  }
  {L(s+1, C_1\otimes C_2)
     L(s+1, D_1\otimes  D_2)  } , $$
  at $s=0$ is the same as the order of pole for the L-function of the virtual representation
  $$(C_1\otimes D_2 + C_2 \otimes D_1)^+ - ( C_1\otimes C_2 + D_1 \otimes D_2)  $$
  \[ = (M_1 \otimes M_2) \boxtimes \left [ \Sym^{i+j+2}(\C^2) -  \Sym^{i- j-2}(\C^2) \right ] \]
at $s=1$ which is $\geq 0$.
\vskip 5pt

If $i=j+1$, then,
\begin{eqnarray*}
  (C_1\otimes D_2 + C_2 \otimes D_1)^+ - ( C_1\otimes C_2 + D_1 \otimes D_2) 
  & =  &  (M_1 \otimes M_2) \boxtimes \left [ \Sym^{2i+1}(\C^2)  \right ],
\end{eqnarray*}
and  we find that the  order of pole of
  $$\frac{L(s+1/2, C_1\otimes D_2 + C_2 \otimes D_1 ) 
  }
  {L(s+1, C_1\otimes C_2)
     L(s+1, D_1\otimes  D_2)  } , $$
  at $s=0$ is $\geq 0$.
\vskip 5pt

If $i=j$, then
\begin{eqnarray*}
  (C_1\otimes D_2 + C_2 \otimes D_1)^+ - ( C_1\otimes C_2 + D_1 \otimes D_2) 
  & =  &  (M_1 \otimes M_2) \boxtimes \left [ \Sym^{2i+2}(\C^2) -  \C \right ],
\end{eqnarray*}
and once again, we find that the  order of pole of
  $$\frac{L(s+1/2, C_1\otimes D_2 + C_2 \otimes D_1 ) 
  }
  {L(s+1, C_1\otimes C_2)
     L(s+1, D_1\otimes  D_2)  } , $$
  at $s=0$ is $\geq 0$.
\end{proof}

\vskip 10pt

We have thus completed the proof of Theorem \ref{T:pole} (i.e. Theorem \ref{pole}).
 In fact, the order of pole in Theorem \ref{T:pole} can be explicitly determined if the action of the Deligne $\SL_2(\C)$ is trivial on the representations $M_A$ and $N_A$. More precisely, we have: 
 \vskip 5pt
 
\begin{prop}
 Let $(M_A, N_A)$ be a pair of relevant A-parameters for $(\GL_m(k),\GL_n(k))$ on which the Deligne $\SL_2(\C)$ acts trivially. Write:
 \begin{eqnarray*} M_A & = &   \sum_{i \geq 1} M_i \boxtimes \Sym^{i-1}(\C^2) = \sum_{i \geq 1} (M_i^+ + M_i^-)\boxtimes \Sym^{i-1}(\C^2) ,  \\ 
  N_A & = &   \sum_{i \geq 1} W_i\boxtimes \Sym^{i-1}(\C^2) = \sum_{i\geq 1} (N_i^+ + N_i^-) \boxtimes \Sym^{i-1}(\C^2), \end{eqnarray*} 
with $M_i, N_i$ representations of $W(k)$,
such that  $M_i^+ = N_{i+1}^-$ for $i \geq 1$, and $M_i^- = N_{i-1}^+$ for $i \geq 2$.
\vskip 5pt

Then  the order of pole at $s=0$ of $L(s,M, N)$ is given by the dimension of $W(k)$-invariants in,
\begin{eqnarray*}
 & &  \sum_{i \geq 1} \left \{\Hom[M_i, N_{i-1} ] + \Hom[M_i, N_{i+1}] -\Hom[M^+_i, N_{i-1}^-] -\Hom[M_i^-, N_{i+1}^+]\right \}, \end{eqnarray*}
with the understanding that $N_0=N^-_0=0$. 
\end{prop}

\begin{proof} The proof follows by examining the arguments made in Proposition \ref{lemma-pole} carefully, a task  which we leave to the reader.
We only  point out here that in the course of the proof of this proposition, say in case 1 of case (a), we had to deal with
  L-function of the virtual representation:
$$M \boxtimes \left [ \Sym^{i+j+2}(\C^2) -  \Sym^{i- j}(\C^2) \right ],$$
  at the point $s=1$. 
  Clearly, if $M$ is a representation of $WD(k)$ which factors through $W(k)$, such a virtual representation has neither a zero nor a pole at $s=1$
  for $i>j$, allowing us to calculate all the poles appearing in Proposition \ref{lemma-pole}.
  \end{proof}

\vskip 15pt

\section{L-functions: Classical Groups} \label{classical}
 The goal of this section is to give the proof of Theorem \ref{interlacing}.
 For the convenience of the reader, we restate the theorem here:

 \begin{thm} \label{T:interlacing}
 Let $k$ be a non-archimedean local field and  let $(M_A, N_A)$ be a pair of A-parameters for $\SO_{2m+1}(k) \times \SO_{2n}(k)$ with associated pair of L-parameters $(M,N)$.
 \vskip 5pt
 
 (i) If $(M_A, N_A)$ is a relevant pair of A-parameters, then  the order of pole at $s=0$ of 
 the function
 \[  L(M,N, s) = \frac{L(M \otimes N, s + ~1/2)}{L(\Sym^2 M \oplus \wedge^2N, s + 1)}  \]
 is greater than or equal to zero.

 \vskip 5pt
 
 (ii) Suppose  that $M_A$ and $N_A$ are multiplicity-free representations of $WD(k) \times \SL_2(\bC)$ on which the Deligne $\SL_2(\C)$ acts trivially. 
  Then, at $s = 0$,  the function
 $  L(M,N, s)$
has  a zero of order $\geq 0$. It  has neither a zero nor a pole at $s=0$ if and only if
$(M_A, N_A)$ is a relevant pair of A-parameters.
\end{thm} 
   
\vskip 10pt

We begin with the proof of (i) which is analogous to that of Theorem \ref{T:pole}.
Write the parameters $M_A$ and $N_A$ as
\begin{eqnarray*}
  M_A & = & A_1 + A_2 + \cdots + A_n + A_0, \\
  N_A & = & B_1 + B_2 + \cdots + B_n + B_0 ,
\end{eqnarray*}
where for $i \geq 1$, $A_i$ and $B_i$ are irreducible representations of $ WD(k) \times \SL_2(\mathbb C)$ of the form,
\begin{eqnarray*}
  A_i & = & M_i \otimes \Sym^{a_i}(\C^2), \\
    B_i & = & M_i \otimes \Sym^{b_i}(\C^2),
  \end{eqnarray*}
with $$ a_i-b_i = \pm 1,$$
and with $A_0,B_0$ tempered parameters of $WD(k)$. 
\vskip 5pt

Observe that  
$$\Sym^2 M_A = \sum_{i > j}  A_i \otimes A_j + \sum_i \Sym^2(A_i),$$
and similarly,
$$\Lambda^2 N_A = \sum_{i > j}  B_i \otimes B_j + \sum_i \Lambda^2(B_i).$$
Therefore, the {\it non-diagonal} contributions to  the order of pole at $s=1$ of the L-function 
$L(\Sym^2 M_A,s) L(\Lambda^2 N_A,s)$ is as in cases (a) and (c) in the statement of Proposition \ref{lemma-pole}.

\vskip 5pt

It suffices to prove the following analogue of case (b) in the statement of Proposition \ref{lemma-pole}, asserting that
 \[L(C,D, s)= \frac{L( C \otimes D, s+1/2 )}  {L( \Sym^2(C), s+1)
     L( \Lambda^2(D), s+1)  } \]
  has a pole of order $\geq 0$ at $s=0$ whenever, 
\begin{eqnarray*}
  C & = & 
M \otimes \Sym^{a}(\C^2), \\
    D & = & M \otimes \Sym^{b}(\C^2),
  \end{eqnarray*}
for $M$ an irreducible tempered representation of $WD(k)$ with $$ a-b = \pm 1.$$ 
\vskip 5pt

To calculate symmetric and exterior square of representations $C$ and $D$ of $WD(k) \times \SL_2(\C)$, note that for any two representations $V,W$ of any group $G$, we have the identity of representations,
$$\Sym^2(V \otimes W) = \Sym^2(V) \otimes \Sym^2(W) \oplus \Lambda^2(V) \otimes \Lambda^2(W),$$ 
and,
$$\Lambda^2(V \otimes W) = \Sym^2(V) \otimes \Lambda^2(W) \oplus \Lambda^2(V) \otimes \Sym^2(W).$$
\vskip 5pt

\noindent Using the well-known structure of $\Sym^2(\Sym^i(\C^2)) $ and $\Lambda^2(\Sym^i(\C^2))$ given by,
$$\Sym^2(\Sym^i(\C^2)) = \Sym^{2i}(\C^2) + \Sym^{2i-4}(\C^2) + \cdots,$$   
and,
$$\Lambda^2(\Sym^i(\C^2)) = \Sym^{2i-2}(\C^2) + \Sym^{2i-6}(\C^2) + \cdots,
$$   
one easily concludes that $L(s,C,D)$  has a pole of order $\geq 0$ at $s=1$, concluding the proof of part (i) of Theorem \ref{T:interlacing}. 

\vskip 10pt

We come now to the proof of (ii). 
Thus, let $M_A$ and $N_A$ be A-parameters for which the Deligne $\SL_2(\C)$ acts trivially. 
\vskip 10pt

 For any
 irreducible representation $\rho$ of $W(k)$, let $M_A[\rho], N_A[\rho]$
 be the $\rho$-isotypic part of $M_A, N_A$ (as a $W(k)$-module). Since the representations $M_A, N_A$
 of $W(k) \times \SL_2(\C)$ are multiplicity free and are selfdual, if  $M_A[\rho] \not = 0$, or $N_a[\rho] \not = 0$, $\rho$ must be a
 selfdual
 representation of $W(k)$.

It is easy to see that the order of zero of $L(M,N,s)$ at $s=0$ is
  the sum of  the order of zeros  of $L(M_A[\rho], N_A[\rho],s)$ at $s=0$ for various distinct irreducible representations
  $\rho$ of $W(k)$. For $\rho$,  an
  irreducible  selfdual representation of $W(k)$, let's write
  \[ M_A[\rho] =  \rho \boxtimes V  \quad \text{and} \quad N_A[\rho] =  \rho \boxtimes W \]
  for representations $V$ and $W$ of (the Arthur) $\SL_2(\C)$. Since $M_A$
  is supposed to be a symplectic representation and $N_A$ an orthogonal representation, if $\rho$ is
  orthogonal, $V$ will be a symplectic representation of (the Arthur) $\SL_2(\C)$, and
  $W$ will be an orthogonal representation of (the Arthur) $\SL_2(\C)$.
    Now using the identities:
     \begin{eqnarray*}
       \Sym^2( \rho \boxtimes V )  & = &  \Sym^2(\rho) \boxtimes \Sym^2(V)  +
       \Lambda^2(\rho) \boxtimes \Lambda^2(V) \\
              \Lambda^2( \rho \boxtimes V )  & = &  \Sym^2(\rho) \boxtimes \Lambda^2(V)  +
              \Lambda^2(\rho) \boxtimes \Sym^2(V),
                  \end{eqnarray*}
  we find that if $\rho$ is irreducible and orthogonal representation of $W(k)$,
  \[  {\rm ord}_{s=0} \left( L(M_A[\rho], N_A[\rho],s) \right)= {\rm ord}_{s=0} \left( L(V ,W,s) \right). \]

  On the other hand, if $\rho$ is an irreducible symplectic representation of $W(k)$, $W$ will be a symplectic representation of (the Arthur)
  $\SL_2(\C)$, and
  $V$ will be an orthogonal representation of (the Arthur) $\SL_2(\C)$, and
    \[  {\rm ord}_{s=0} \left( L( M_A[\rho], N_A[\rho],s) \right)= {\rm ord}_{s=0} \left( L(W ,V,s) \right). \]

  Thus it suffices to prove the theorem assuming that $WD(k)$ acts trivially on $M_A$ and $N_A$. In other words, we have multiplicity-free representations
       \begin{eqnarray*}
      M_A & = & V =  \Sym^{a_1-1}(\C^2)  + \Sym^{a_2-1}(\C^2)  + \cdots +  \Sym^{a_r-1}(\C^2), \\
      N_A & = & W = \Sym^{b_1-1}(\C^2)  + \Sym^{b_2-1}(\C^2)  + \cdots +  \Sym^{b_s-1}(\C^2),
    \end{eqnarray*}
    with $V$ symplectic and $W$ orthogonal, thus with all $a_i$ even, and $b_i$ odd.
\vskip 5pt

    We prove the theorem  using an inductive argument from the `top'. Suppose
 $d$ is the largest integer $a$ such that $ \Sym^{a-1}(\C^2)$ is contained in either $V$ or $W$. We will assume that
    $ \Sym^{d-1}(\C^2)$   appears in $V$,
    a similar argument can be given if $ \Sym^{d-1}(\C^2)$   appears in $W$.
    Suppose $e$ is the largest integer such that $ \Sym^{e-1}(\C^2)$ appears in $W$. By hypothesis,
   all the integers $a_i$ for which $  \Sym^{a_i-1}(\C^2) \subset V$ is even,
   and all the integers $b_i$ for which $ \Sym^{b_i-1}(\C^2) \subset W$ is odd, in particular,
   $d$ is even, $e$ is odd, and $d>e$.
\vskip 5pt

  Define,
    \begin{eqnarray*}
  V' & = & V - \Sym^{d-1}(\C^2), \\
  W' & = & W -   \Sym^{e-1}(\C^2).
    \end{eqnarray*}
We will prove the following proposition, which then allows us to complete the proof of  Theorem \ref{T:interlacing}(ii) by an inductive argument.

\vskip 5pt

\begin{prop}
  With the notation and assumptions as above, in particular $V$ symplectic and $W$ orthogonal multiplicity-free
  representation of the Arthur
  $\SL_2(\C)$,
the order of pole at $s=0$ of $L(V,W,s)$ defined by:
\[ L(V,W, s) = \frac{L(V \otimes W, s + ~1/2)}{L(\Sym^2 V \oplus \wedge^2W, s + 1)}, \] 
is less than or equal to  that of
    $L(V',W',s)$. Equality holds if and only if 
   $d = (e+1)$. 
   \end{prop}
\vskip 5pt

\begin{proof}
Clearly,
\begin{eqnarray*} 
V \otimes W & =  &  V' \otimes W' +  W' \otimes  \Sym^{d-1}(\C^2) + V' \otimes
   \Sym^{e-1}(\C^2)+  \Sym^{d-1}(\C^2)\otimes \Sym^{e-1}(\C^2) , \\
  \Sym^2(V) & = & \Sym^2(V') + \Sym^2 (\Sym^{d-1}(\C^2)) +   V' \otimes \Sym^{d-1}(\C^2) , \\
  \Lambda^2(W) & = & \Lambda^2(W') + \Lambda^2 (  \Sym^{e-1}(\C^2)) +   W' \otimes  \Sym^{e-1}(\C^2).
  \end{eqnarray*}
  
\noindent Therefore, 
\[  \frac{L(V,W,s)}{ L(V',W',s)} = \] 
\[ \frac{ L(
  W'\otimes  \Sym^{d-1}(\C^2) + V' \otimes  \Sym^{e-1}(\C^2) +  \Sym^{d-1}(\C^2)\otimes \Sym^{e-1}(\C^2), s+\frac{1}{2}, )}
{L( W' \otimes \Sym^{e-1}(\C^2)  + V' \otimes \Sym^{d-1}(\C^2) +   \Sym^2 \Sym^{d-1}(\C^2) + \Lambda^2 \Sym^{e-1}(\C^2), s+1)}.\]
\vskip 10pt

We will analyze the order of pole at $s=0$ of the following L-functions,

\begin{eqnarray*}  A(s) & = & \frac{ L(
    \Sym^{d-1}(\C^2)\otimes \Sym^{e-1}(\C^2),s+\frac{1}{2} )}
     {L(\Sym^2  \Sym^{d-1}(\C^2) + \Lambda^2 \Sym^{e-1}(\C^2) , s+1)}, \\
      B(s) & = & \frac{ L(
  W'\otimes  \Sym^{d-1}(\C^2), s+\frac{1}{2})}
          {L( W'\otimes  \Sym^{e-1}(\C^2), s+1)}, \\
          C(s) & = & \frac{ L(
   V' \otimes \Sym^{e-1}(\C^2), s+\frac{1}{2})}
{L(  V' \otimes  \Sym^{d-1}(\C^2), s+1)}. \end{eqnarray*}

\vskip 5pt

\begin{itemize}
\item Analyzing $A(s)$: 
 Since $d$ is even, $e$ is odd, and $d>e$,  $A(s)$ has a zero at $s=0$ of order
 $$ a = -e+ (d+e-1)/2 = (d-e-1)/2.$$
\vskip 5pt

\item Analyzing $B(s)$: 
   Again, since $d$ is even, $e$ is odd, and $d>e$,
 and $e$ is bigger than all integers $b$ for which
 $\Sym^{b-1}(\C^2)   $ is contained in $W'$, it follows that, in writing $W'$ as a sum of irreducible
 pieces  $\Sym^{b-1}(\C^2)$'s, the $ \Sym^{b-1}(\C^2)   $'s contribute  the same number of
 poles in the denominator as in the numerator of $B(s)$. Therefore
  $B(s)$ has neither a zero nor a pole at $s=0$.
\vskip 5pt

\item Analyzing $C(s)$:
  Again, since $d$ is even, $e$ is odd, and $d>e$, the number poles at $s=0$ in the
 denominator of $C(s)$ is $\geq $    the number of poles at $s=0$ in the
 numerator of $C(s)$, the difference is contributed by  those $\Sym^{t-1}(\C^2)$, with $d> t >e$ and $t$ even,
 which may belong to $V'$.
  Thus we get  
 $$0 \leq c \leq   (d-e-1)/2,$$
 for the number of zeros for the factor $C(s)$ at $s=0$.
\end{itemize}
 \vskip 5pt
 
 In conclusion, we find that
 \[ {\rm ord}_{s=0} \left(  A(s)B(s)C(s) \right) \geq 0. \]
 Moreover, if  $A(s)B(s)C(s)$ has no zero at $s=0$, then we must have $d= e+1$. Conversely,  if $d=e+1$,
 then there are no zeros or poles for $A(s)B(s)C(s)$ at $s=0$,
completing the proof of the proposition. 
\end{proof}
\vskip 10pt

%\begin{example} Let,
%\begin{enumerate}
%\item $\sigma_1 =  \Sym^{k-1}(\C^2),$
 % \item $\sigma_2=  \Sym^{\ell-1}(\C^2),$
%\end{enumerate}
%with $k$ even and $\ell$ odd, so that $\sigma_1$ is symplectic and $\sigma_2$ orthogonal. In this case, $L(s+1/2, \sigma_1 \times \sigma_2)$ has $\min(k,\ell)$ order of pole at $s=0$, whereas  $L(s+1, \Sym^2 \sigma_1 + \Lambda^2 \sigma_2) $ has $k/2+ (\ell -1)/2$ order of pole. 
%Therefore, $$L(s,\sigma_1,\sigma_2) 
%=  \frac{L(s+1/2, \sigma_1\otimes \sigma_2)}{
 % L(s+1, \Sym^2 \sigma_1) L(s+1, \Lambda^2 \sigma_2) }, $$
%has neither a zero nor a pole at $s=0$ if and only if $k=\ell \pm 1$. \end{example}

   There seems no simple generalization of Theorem \ref{T:interlacing} to all discrete A-parameters of classical groups --- although it would be highly desirable. We give two examples of its failure, beginning with one in which we
  allow general tempered parts for $M_A$ and $N_A$,  keeping other conditions in Theorem \ref{T:interlacing} intact. Let
\begin{eqnarray*}
  M_A & = & [10] \times [1], \\
  N_A & = & [5]\times [1] + [1] \times [7] + [1] \times [9].
  \end{eqnarray*}
It can be seen
that
$$L(M,N,s) 
=  \frac{L(
  M\otimes N,s+1/2)}{
  L(\Sym^2 M, s+1) L( \Lambda^2 N, s+1) }, $$
has neither a zero nor a  pole at $s=0$, even though the pair $(M_A, N_A)$ is irrelevant.
  
\vskip 10pt

  As second example, consider the relevant pair of A-parameters
  \begin{eqnarray*}
  M_A & = & [3] \times [4] + [5] \times [4], \\
  N_A & = & [3] \times [3] + [5] \times [5].
  \end{eqnarray*}
    These satisfy all the
  conditions in Theorem \ref{T:interlacing} except that the Deligne $\SL_2(\C)$ does not act trivially. It can be seen that in this case,
  $L(
  M \otimes N, s+1/2)$ has a pole of order 25, whereas  $ L(\Sym^2 M,s+1) L( \Lambda^2 N,s+1) $ has a pole of order 20, so
    in this case $L( M,N,s)$
has a pole at $s=0$ of order 5.  


\begin{thebibliography}{MVW} 


%\bibitem[A1]{A1} J. Arthur, {\em
%On some problems suggested by the trace formula},  Lie group representations, II (College Park, Md.%, 1982/1983), 1--49, Lecture Notes in Math., 1041, Springer, Berlin, 1984.



 
\bibitem[AGRS]{AGRS}A. Aizenbud, D. Gourevitch, 
S. Rallis and G. Schiffmann, {\em Multiplicity 
One Theorems}, {Ann. of Math.} 172 (2010), 
1407-1434.

\bibitem[AP]{AP} U. Anandavardhanan and  D. Prasad, {\em Distinguished representations for SL(n).} Math. Res. Lett. 25 (2018), no. 6, 1695-1717.

 \bibitem[Art1]{Art1} J. Arthur, {\em Unipotent automorphic representations: conjectures},  Orbites unipotentes et representations, II. Asterisque No. 171-172 (1989), 13--71.

\bibitem[Art2]{Art2} J. Arthur, {\em The endoscopic classification of representations: orthogonal and symplectic groups.}
    Colloquium Publications by the American Mathematical Society, volume 61 (2013).

\bibitem[AG]{AG} H. Atobe and W.T. Gan, {\em Local theta correspondence of tempered representations and Langlands parameters.}
Invent. Math. \textbf{210} (2017), no.~2, 341--415.




  
\bibitem[BLM]{BLM} I. Badulescu, E. Lapid, A. Mínguez,  {\em
  Une condition suffisante pour l'irr\'eductibilit\'e d'une induite parabolique de $\GL(m,D)$}.
    Ann. Inst. Fourier (Grenoble) 63 (2013), no. 6, 2239-2266.

\bibitem[BZ]{BZ}J. Bernstein and A. Zelevinsky, \emph{Induced
Representations of Reductive $p$-adic groups}, Ann. Scient. Ecole Norm. Sup.
10 (1977) 441-472.


%\bibitem[BP]{BP} R. Beuzart-Plesis, {\em Endoscopie et conjecture locale raffin\'ee de Gan-Gross-Prasad pour les groupes
 % unitaires,} Compos. Math. 151 (2015), no. 7, 1309-1371.
  
\bibitem[Ch]{Ch} K. Y. Chan, {\em Homological branching law for $(\GL_{n+1}(F),\GL_n(F))$: projectivity and indecomposability}; arXiv:1905.01668

\bibitem[Cl]{Cl}	L. Clozel, \emph{Combinatorial consequences of Arthur's Conjectures and the  Burger-Sarnak method}, IMRN, no. 11, 511-523 (2004).

\bibitem[Fl]{Fl} Y. Z. Flicker, {\em On distinguished representations,} J. Reine Angew. Math. 418 (1991), 139-172. 
  
 \bibitem[GG]{GG}    W.T. Gan, N. Gurevich, {\em Restrictions of Saito-Kurokawa representations.
   With an appendix by Gordan Savin.} Contemp. Math., 488, Automorphic forms and L-functions I. Global aspects, 95-124,
   Amer. Math. Soc., Providence, RI, 2009.

\bibitem[GGP]{GGP} W. T. Gan, Benedict H. Gross and D. Prasad,
  {\em Symplectic local root numbers, central critical L-values and restriction problems in the representation theory of classical groups},   Sur les conjectures de Gross et
  Prasad. I. Ast\'erisque No. 346 (2012), 1-109.


\bibitem[GGP2]{GGP2} W. T. Gan, Benedict H. Gross and D. Prasad,
  {\em Restrictions of representations of classical groups: examples,}
  Sur les conjectures de Gross et Prasad. I. Ast\'erisque No. 346 (2012), 111-170.

 \bibitem[GI]{GI} W.~T.~Gan and A.~Ichino,
\emph{The Shimura--Waldspurger correspondence for ${\rm Mp}_{2n}$.}
Ann. of Math. (2) \textbf{188} (2018), no.~3, 965--1016.
  
%\bibitem[GJR1]{GJR1} D. Ginzburg, D. Jiang, and S. Rallis, {\em  On the nonvanishing of the central value of the Rankin-Selberg L-functions},  J. Amer. Math. Soc. 17 (2004), no. 3, 679--722.

%\bibitem[GJR2]{GJR2}D. Ginzburg, D. Jiang, and S. Rallis, {\em On the nonvanishing
%of the central value of the Rankin-Selberg L-functions, II}, Automorphic 
%Representations, L-functions and Applications: Progress and Prospects, 157-191, Ohio State Univ. Math. Res. Inst. Publ. 11, de Guyter, Berlin 2005. 

%\bibitem[GJR3]{GJR3}D. Ginzburg, D. Jiang, and S. Rallis, {\em Models for certain residual
%representations of unitary groups}, in {\em Automorphic Forms and 
%L-functions I. Global aspects}, AMS Contemporary Math., 
%vol. 488 (2009) 125-146. 

\bibitem[GP1]{GP1} B. Gross and D. Prasad, {\em
 On the decomposition of a representation of ${\rm SO}\sb n$ 
when restricted to ${\rm SO}\sb {n-1}$.}  Canad. J. Math.  44  (1992),  
no. 5, 974--1002.

\bibitem[GP2]{GP2} B. Gross and D. Prasad, {\em
 On irreducible representations of} ${\rm SO}\sb {2n+1}\times{\rm SO}\sb 
{2m}$.  Canad. J. Math.  46  (1994),  no. 5, 930--950. 
 
% \bibitem[GR]{GR} B. Gross and M. Reeder, {\em From Laplace to Langlands 
%via representations of orthogonal
%groups,} Bulletin of the AMS 43 (2006) 163-205.

\bibitem[GRS1]{GRS1} D. Ginzburg, S. Rallis and D. Soudry, {\em  On a correspondence between cuspidal representations of $\GL_{2n}$  and $\tilde{\Sp}_{2n}$},  J. Amer. Math. Soc. 12 (1999), no. 3, 849-907.
  
\bibitem[GRS2]{GRS2}  D. Ginzburg, S. Rallis and D. Soudry, {\em On explicit lifts of cusp forms from $\GL_m$  to classical groups},  Ann. of Math. (2) 150 (1999), no. 3, 807-866.
  
 \bibitem[GRS3]{GRS3}  D. Ginzburg, S. Rallis and D. Soudry, {\em  Endoscopic representations of $\tilde{\Sp}_{2n}$} J. Inst. Math. Jussieu 1 (2002), no. 1, 77-123.
 
\bibitem[GRS4]{GRS4}  D. Ginzburg,  S. Rallis and D. Soudry, {\em The descent map from automorphic representations of $\GL(n)$ to classical groups}, World Scientific Publishing Co. Pte. Ltd., Hackensack, NJ, 2011. x+339 pp.

 
\bibitem[GRO]{GRO} S. Gelbart and J. Rogawski, {\em  L-functions and Fourier-Jacobi coefficients for the unitary group ${\rm U}(3)$}, Invent. Math. 105 (1991), no. 3, 445--472.




\bibitem[GRS]{GRS} S. Gelbart, J. Rogawski, and D. Soudry, {\emph Endoscopy,
theta-liftings, and period integrals for the unitary group 
in three variables}. Annals of Math, 145 (1997), 419-476.

 

%\bibitem[GZ]{} B. Gross, and D. Zagier, {\em Heegner points and derivative of
%L-series}, Invent. Math 84 (1986), no. 2, 225-320.


%\bibitem[GPSR]{GPSR} S. Gelbart, I. Piatetski-Shapiro and S. Rallis, {\em Explicit constructions of automorphic L-functions},  Lecture Notes in Mathematics, 1254. Springer-Verlag, Berlin, 1987,  vi+152 pp.


 

  

   
 \bibitem[GS]{GS}  N. Gurevich, and D. Szpruch, {\em
   The non-tempered $\theta_{{10}}$ Arthur parameter and Gross-Prasad conjecture,} Journal of Number Theory, (2015), pp. 372-426.
    
     \bibitem[Gu]{Gu}M. Gurevitch,   {\em On restriction of unitarizable representations of
   general linear groups and the non-generic local Gan-Gross-Prasad conjecture}, to appear in J. of EMS, arXiv 1808.02640 (2018).

   
%\bibitem[Ha]{Ha} M. Harris, {\em L-functions of $2\times 2$ unitary groups and factorization of periods of Hilbert modular forms},  Journal of AMS 6 (1993), 637-719.
\bibitem[H1]{H1} J. Haan, {\em The local Gan-Gross-Prasad conjecture for $\U(3)\times \U(2)$: the non-generic case}, J. Number Theory 165 (2016), 324-354.

\bibitem[H2]{H2} J. Haan, {\em On the Fourier-Jacobi model for some endoscopic Arthur packet of $\U(3)\times \U(3)$: the nongeneric case},
 Pacific J. Math. 286 (2017), no. 1, 69-89.


\bibitem[HLR]{HLR} G. Harder, R.P. Langlands, and M. Rapoport:  Algebraische Zyklen auf Hilbert-Blumenthal-Fl\"achen, {\it J. Reine Angew. Math.}, 366 (1986), 53-120. 


%\bibitem[HK]{HK}M. Harris and S. Kudla, {\em The central critical L-value of a triple product L-function}, Annals of Mathematics, 133 (1991), 605-672. 
  
\bibitem[HT]{HT} M. Harris and R. Taylor, {\em The geometry and cohomology of some simple Shimura varieties}, With an appendix by Vladimir G. Berkovich,  Annals of Mathematics Studies, 151. Princeton University Press, Princeton, NJ, 2001.

\bibitem[Hen]{Hen} A. E. Hendrickson, {\em The Burger-Sarnak method and operations on the unitary duals of classical groups}, Ph.D. thesis, National University of Singapore,
arXiv:2002.11935 (2020).


%\bibitem[He]{He} G. Henniart, {\em Une preuve simple des conjectures de Langlands pour ${\rm GL}(n)$ sur un corps $p$-adique},  Invent. Math. 139 (2000), no. 2, 439--455.

%\bibitem[Ho]{Ho} R. Howe, {\em Transcending classical invariant theory}, J. Amer. Math. Soc. 2 (1989), no. 3, 535--552.

\bibitem[II]{II} A. Ichino and T. Ikeda, {\em On the periods of automorphic forms on special orthogonal groups and the Gross-Prasad conjecture}, Geometric and Functional Analysis, 19 (2010), no. 5, 1378-1425.
  
\bibitem[IY]{IY}  A. Ichino and S. Yamana,  {\em Periods of automorphic forms: the case of
  $(\GL_{n+1} \times \GL_n, \GL_n)$;}
  Compos. Math. 151 (2015), no. 4, 665 -712.

\bibitem[JLR]{JLR} H. Jacquet, E. Lapid and  J. Rogawski, {\em Periods of automorphic forms},  J. Amer. Math. Soc. 12 (1999), no. 1, 173-240.
%\bibitem[MVW]{MVW} C.M{\oe}glin, M.-F. Vigneras and J.-L. Waldspurger, {\em Correspondances de Howe sur un corps $p$-adique}, Lecture Notes in Mathematics, 1291. Springer-Verlag, Berlin, 1987. viii+163 pp. 
\bibitem[JPSS]{JPSS} H. Jacquet,  I.I. Piatetskii-Shapiro, J.A. Shalika,  {\em Rankin-Selberg convolutions}, Amer. J. Math. 105 (1983), no. 2, 367-464. 

%\bibitem[JZ]{JZ} D. Jiang; L. Zhang, {\em On the nonvanishing of the central value
 % of certain L-functions: unitary groups.}  J. of EMS (2018).
  
\bibitem[JZ]{JZ} D. Jiang; L. Zhang, {\em Arthur parameters and cuspidal automorphic modules of classical groups.} Ann. of Math. (2) 191 (2020), no. 3, 739 -827.

\bibitem[Ka]{Ka} A. Kable,
{\em Asai L-functions and Jacquet's conjecture.}
Amer. J. Math. 126 (2004), no. 4, 789 -820.
  
\bibitem[KMSW]{KMSW}   T. Kaletha; A. Minguez; S.W. Shin; P.-J. White, {\em Endoscopic Classification of Representations:
  Inner Forms of Unitary Groups}.  arXiv:1409.3731 (2014).

\bibitem[M1]{M1} C. M{\oe}glin, {\em Sur certains paquets d'Arthur et involution d'Aubert-Schneider-Stuhler g\'en\'eralis\'ee},  Represent. Theory 10 (2006), 86-129.

\bibitem[M2]{M2} C. Moelglin, {\em Paquets d'Arthur discrets pour un groupe classique p-adique}, in {\em Automorphic forms and L-functions II. Local aspects}, 179-257, Contemp. Math., 489, Israel Math. Conf. Proc., Amer. Math. Soc., Providence, RI, 2009.

\bibitem[M3]{M3} C. M{\oe}glin, {\em Multiplicit\'e 1 dans les paquets d'Arthur aux places p-adiques},  in {\em  On certain L-functions}, 333-374, Clay Math. Proc., 13, Amer. Math. Soc., Providence, RI, 2011.

  \bibitem[M4]{M4} C. M{\oe}glin, 
    {\em Classification et changement de base pour les séries discrètes des groupes unitaires p-adiques,}
    Pacific J. Math. 233 (2007), no. 1, 159-204.
 
\bibitem[MR1]{MR1} C. M{\oe}glin and D. Renard, {\em Sur les paquets d'Arthur des groupes classiques et unitaires non quasi-déploy\'es},   in {\em Relative aspects in representation theory, Langlands functoriality and automorphic forms}, 341-361, Lecture Notes in Math., 2221, CIRM Jean-Morlet Ser., Springer, Cham, 2018.

\bibitem[MR2]{MR2} C. M{\oe}glin and D. Renard, {\em Paquets d'Arthur des groupes classiques complexes},  in {\em  Around Langlands correspondences}, 203-256, Contemp. Math., 691, Amer. Math. Soc., Providence, RI, 2017.

\bibitem[MR3]{MR3} C. M{\oe}glin and D. Renard, {\em Sur les paquets d'Arthur des groupes unitaires et quelques cons\'equences pour les groupes classiques},   Pacific J. Math. 299 (2019), no. 1, 53-88.

   
\bibitem[Mok]{Mok} C. P. Mok, {\em Endoscopic classification of representations of quasi-split unitary groups.}
  Mem. Amer. Math. Soc. 235 (2015), no. 1108.

 
\bibitem[Mu1]{Mu1}
{G. Mui\'{c}}, 
{\em Howe correspondence for discrete series representations; the case of $(\Sp(n),\O(V))$}, 
{\it J. Reine Angew. Math}. {\bf 567} (2004), 99--150.

 
\bibitem[Mu2]{Mu2}
{G. Mui\'{c}}, 
{\em On the structure of theta lifts of discrete series for dual pairs $(\Sp(n),\O(V))$}, 
{\it Israel J. Math}. {\bf 164} (2008), 87--124.

\bibitem[Mu3]{Mu3}
{G. Mui\'{c}}, 
{\em Theta lifts of tempered representations for dual pairs $(\Sp_{2n},\O(V))$}, 
{\it Canad. J. Math}. {\bf 60} (2008), no.~6, 1306--1335.


\bibitem[OS]{OS}O. Offen and E. Sayag, {\em  On unitary representations of $\GL_{2n}$
  distinguished by the symplectic group.}
  J. Number Theory 125 (2007), no. 2, 344-355.
  
\bibitem[P]{P} D. Prasad, {\em  Invariant forms for representations of $\GL_2$ over a local field},  Amer. J. Math. 114 (1992), no. 6, 1317-1363. 

\bibitem[Ro]{Ro}J.Rogawski, {\em Automorphic representations of unitary groups
in three variables}, Annals of Mathematics Studies, vol 123, Princeton University Press, 1990.

\bibitem[Ro2]{Ro2}J. Rogawski, {\em The multiplicity formula for A-packets}, in {\em The zeta functions of Picard modular surfaces}, edited by R. Langlands, and D. Ramakrishnan, 395--419,   

\bibitem[ST]{ST} D. Soudry and Y. Tanay, {\em  On local descent for unitary groups},  J. Number Theory 146 (2015), 557-626. 
  
\bibitem[Tad]{Tad} M. Tadi\'c, {\em Classification of unitary representations in irreducible representations of general linear group
  (non-Archimedean case)}. Ann. Sci. \'Ecole Norm. Sup. (4) 19 (1986), no. 3, 335-382.
  

\bibitem[Tate]{Tate} J. Tate, {\em Number theoretic background.}
  Automorphic forms, representations and L-functions
  (Proc. Sympos. Pure Math., Oregon State Univ., Corvallis, Ore., 1977), Part 2, pp. 3–26,
  Proc. Sympos. Pure Math., XXXIII, Amer. Math. Soc., Providence, R.I., 1979. 
  
\bibitem[Ve]{Ve} A. Venkatesh, {\em The Burger-Sarnak method and operations on the unitary dual of $\GL(n)$.} Representation Theory, 
 vol. 9, pages 268-286  (2005).

\bibitem[Venk]{Venk} C.G. Venketasubramanian, {\em Representations of $\GL_{n}$ distinguished by $\GL_{n-1}$ over a $p$-adic field};  Israel Journal of Mathematics 194 (2013) 1-44.

\bibitem[Vo]{Vo}D. Vogan, {\em The local Langlands conjecture},
Contemporary Maths, AMS, 145, 305-379 (1993).
 
%\bibitem[W]{W} J.-L. Waldspurger, \emph{Sur les valeurs de certaines 
%fonctions L automorphes en leur centre de symmetrie.}
 % Compositio Math.  54  (1985),  no. 2, 173--242. 



%\bibitem[Wa1]{} J.-L. Waldspurger, {\em  Correspondance de Shimura},  J. Math. Pures Appl. (9) 59 (1980), no. 1, 1--132.

%\bibitem[Wa2]{} J.-L. Waldspurger, {\em Correspondances de Shimura et quaternions}, Forum Math. 3 (1991), no. 3, 219--307.




%\bibitem[Wa2]{Wa2}J.-L.Waldspurger, {\em  Une Formule int\'egrale reli\'ee \`a la conjecture locale de Gross-Prasad,} Compositio Mathematica, vol 146 (2010) no. 05, 1180-1290.


\bibitem[W]{W}J.-L. Waldspurger, {\em La conjecture locale de Gross-Prasad pour les repr\'esentations temp\'er\'ees
  des groupes sp\'eciaux orthogonaux,} Asterisque 347 (2012).

\bibitem[Z]{Z} C.B. Zhu, {\em Local theta correspondence and nilpotent invariants},  Representations of reductive groups, 427-450, Proc. Sympos. Pure Math., 101, Amer. Math. Soc., Providence, RI, 2019.

\bibitem[Zy]{Zy} M. Zydor, {\em Periods of automorphic forms over reductive groups},  preprint, arXiv:1903.01697

%\bibitem[We1]{} A. Weil, {\em Basic number theory},  Third edition, Die Grundlehren der Mathematischen Wissenschaften, Band 144. Springer-Verlag, New York-Berlin, 1974. xviii+325 pp.

%\bibitem[We2]{} A. Weil, {\em  Sur certains groupes d'op屍ateurs unitaires}, Acta Math. 111 1964 143--211.

%\bibitem[YZZ]{YZZ} X. Yuan, S. W. Zhang and W. Zhang, {\em Heights of CM points I: Gross-Zagier formula}, preprint,  available at {\em http://www.math.columbia.edu/~szhang/papers/HCMI.pdf}.
 
%\bibitem[Zh]{Zh} W. Zhang, {\em Relative trace formula and arithmetic
%Gross-Prasad conjecture}, preprint (2009).

%\bibitem[Z]{Z} C. Zorn, {\em Theta dichotomy and doubling gamma factors for 
%$\widetilde{\Sp}_n(F)$}, preprint, available at {\em http://www.math.ohio-state.edu/~czorn}.

\end{thebibliography}
\end{document}